\documentclass[10pt, oneside]{amsart}   	%
\usepackage{geometry,amsmath,amssymb,verbatim,tikz, ytableau,enumitem}
\usepackage{tikz-3dplot}
\usepackage[utf8]{inputenc}
\usepackage{caption,subcaption,blkarray}
\usepackage[colorlinks=true,hyperindex, linkcolor=blue, pagebackref=false, citecolor=cyan]{hyperref}

\usepackage[utf8]{inputenc}

\allowdisplaybreaks

\newtheorem{thm}{Theorem}[section]
\newtheorem{lemma}[thm]{Lemma}
\newtheorem{prop}[thm]{Proposition}
\newtheorem{cor}[thm]{Corollary}
\newtheorem{conj}[thm]{Conjecture}
\newtheorem{question}[thm]{Question}
\newtheorem{definition}[thm]{Definition}

\theoremstyle{definition}
\newtheorem{example}[thm]{Example}
\theoremstyle{remark}
\newtheorem{remark}[thm]{Remark}
\theoremstyle{remark}

\def\QQ{{\mathbb{Q}}}
\def\RR{{\mathbb{R}}}
\def\ZZ{{\mathbb{Z}}}
\def\NN{{\mathbb{N}}}

\def\FF{{\mathbb{F}}}
\def\specht{{\mathbb{S}}}

\def\II{{\mathbf{I}}}
\def\AA{{\mathbb{A}}}

\def\kk{{\mathbf{k}}}

\def\HHH{{\mathcal{H}}}
\def\FFF{{\mathcal{F}}}
\def\SSS{{\mathcal{S}}}

\def\symm{{\mathfrak{S}}}
\def\symmB{B}
\def\gr{{\mathrm{gr}}}

\def\ch{{\mathrm{ch}}}
\def\spn{{\mathrm{span}}}
\def\supp{{\mathrm{supp}}}

\def\one{{\mathbf{1}}}

\def\BBB{{\mathcal{B}}}

\def\SSS{{\mathcal{S}}}
\def\PPP{{\mathcal{P}}}

\def\cone{\mathrm{cone}}

\def\Cross{\lozenge}
\def\Aff{\mathrm{Aff}}
\def\Spec{\mathrm{Spec}}

\def\opp{\mathrm{opp}}

\def\Ehr{i}
\def\intEhr{\overline{i}}
\def\Eseries{\mathrm{E}}
\def\intEseries{\overline{\mathrm{E}}}
\def\Hilb{\mathrm{Hilb}}
\def\conv{\mathrm{conv}}
\def\pyr{\mathrm{Pyr}}

\def\spn{\mathrm{span}}
\def\sm{\mathrm{sm}}

\def\Ver{\mathrm{Ver}}
\def\Segre{\mathrm{Segre}}
\def\spn{\mathrm{span}}

\def\Hom{\mathrm{Hom}}
\def\init{\mathrm{in}}
\def\vol{\mathrm{vol}}
\def\trunc{\mathrm{trunc}}

\def\Zpoints{\mathcal{Z}}

\def\Div{\mathbb{D}}
\def\Cl{\mathrm{Rep}}

\def\aa{\mathbf{a}}
\def\bb{\mathbf{b}}

\def\dd{\mathbf{d}}
\def\pp{\mathbf{p}}
\def\qq{\mathbf{q}}

\def\vv{\mathbf{v}}
\def\xx{\mathbf{x}}
\def\yy{\mathbf{y}}
\def\zz{\mathbf{z}}
\def\ee{\mathbf{e}}
\def\origin{\mathbf{0}}

\newcommand{\interior}[1]{\mathrm{int}(#1)}
\newcommand{\boundary}[1]{\mathrm{Bd({#1})}}

\title{Harmonics and graded Ehrhart theory}

\author{Victor Reiner}
\address{School of Mathematics \\ University of Minnesota \\ Minneapolis, MN, 55455}
\email{reiner@umn.edu}
\author{Brendon Rhoades}
\address{Department of Mathematics \\ UCSD \\ La Jolla, CA, 92093}
\email{bprhoades@ucsd.edu}

\begin{document}

\begin{abstract}
The Ehrhart polynomial and Ehrhart series count lattice points in integer dilations of a lattice polytope.  We introduce and study a $q$-deformation of the Ehrhart series, based on the notions of harmonic spaces and Macaulay's inverse systems for coordinate rings of finite point configurations.  We conjecture that this $q$-Ehrhart series is a rational function, and introduce and study a bigraded algebra whose Hilbert series matches the $q$-Ehrhart series.  Defining this algebra requires a
new result on Macaulay inverse systems for
Minkowski sums of point configurations.

\end{abstract}

\maketitle


\section{Introduction}
\label{sec:Introduction}

A {\em lattice polytope} $P$ in $\RR^n$ is the convex hull of a finite set of points in $\ZZ^n$. With their deep connections to commutative algebra and toric geometry, lattice polytopes are a central object in algebraic combinatorics. Our goal 
is to introduce in a canonical way an extra $q$-parameter into two important algebro-combinatorial objects associated to
a lattice polytope $P$: the Ehrhart polynomials/series and the affine semigroup ring associated to $P$.

\subsection{Classical Ehrhart theory and a $q$-analogous conjecture}

The classical story to be generalized is one of lattice point enumeration. Given a $d$-dimensional lattice polytope $P \subset \RR^n$ with boundary $\boundary{P}$ and {\it (relative) interior} $\interior{P}=P \setminus \boundary{P}$, for $m \geq 0$, one has its $m$-fold {\em dilate}   
$mP := \{ m p \,:\, p \in P \} \subseteq \RR^n$, along with its interior $\interior{mP}$.
Define these integer point enumerators $\Ehr_P(m), \intEhr_P(m)$, and their generating functions
$\Eseries_P(t), \intEseries_P(t)$ in $\ZZ[[t]]$:
\begin{align}
    \label{eq:intro-ehrhart-count}
    \Ehr_P(m) &:= \# (\ZZ^n \cap mP) \text{ for }m \geq 0, 
     &\Eseries_P(t):=\sum_{m=0}^\infty \Ehr_P(m) t^m,\\
    \label{eq:intro-interior-ehrhart-count}
    \intEhr_P(m) &:= \# (\ZZ^n \cap \interior{mP})
    \text{ for }m \geq 1, 
     &\intEseries_P(t):=\sum_{m=1}^\infty \intEhr_P(m) t^m.
\end{align}
We summarize here some of the results from the original work of  Ehrhart \cite{ehrhart1962polyedres}, as well as Macdonald \cite{macdonald1971polynomials} and Stanley \cite{stanley1975combinatorial, stanley1980decompositions};  see also expositions and surveys in Braun \cite{Braun}, Beck and Robins \cite{BeckRobins}, Beck and Sanyal \cite{BeckSanyal},
Stanley \cite[Ch.~1]{Stanley-CCA} \cite[\S4.6]{Stanley-EC1}.

\medskip
\noindent
{\bf Classical Ehrhart Theorem.} 
{\it Let $P$ be a $d$-dimensional lattice polytope.
\begin{itemize}
    \item[(i)] Both $\Ehr_P(m)$ and $\intEhr_P(m)$ are
    polynomial functions of $m$ of degree $d$. Both generating functions $\Eseries_P(t)$ and $\intEseries_P(t)$ are rational functions which can be written over the denominator $(1-t)^{d+1}$,
    and whose numerator polynomials lie in $\ZZ[t]$, with numerator of degree at most $d$ in the case of $\Eseries_P(t)$ and degree exactly $d+1$
    in the case of $\intEseries_P(t)$.
    \item[(ii)]  The integer coefficients $\{h^*_i\}_{i=0}^d$ defined uniquely by
$$
\Eseries_P(t) = \frac{\sum_{i=0}^d h^*_i t^i}{(1-t)^{d+1}}
$$
are nonnegative, with $h^*_0=1$ and  $\sum_{i=0}^d h^*_i=v$ the normalized $d$-volume $v:=\vol_d(P)$.
\item[(iii)] For $P$ a lattice $d$-simplex
with vertices $\vv^{(1)},\ldots,\vv^{(d+1)}$, one can interpret the  $\{h^*_i\}_{i=0}^d$ as
$$
h^*_i=\# \{ (x_0,\zz) \in \Pi \cap \ZZ^{n+1}: x_{0}=i\},
$$
where $\Pi  \subseteq \RR^{n+1}=\RR^1 \times \RR^n$
is this semi-open parallelepiped 
spanned by
$\{(1,\vv^{(j)})\}_{j=1,\ldots,d+1}$: 
\begin{equation}
\label{eq: semi-open-parallelepiped}
\Pi:=\left\{ \sum_{j=1}^{d+1} c_j \cdot (1,\vv^{(j)}): 0 \leq c_i < 1\right\}.
\end{equation}
\item[(iv)] (Ehrhart-Macdonald reciprocity) The series $\Eseries_P(t), \intEseries_P(t)$ determine each other uniquely by
$$
\intEseries_P(t) = (-1)^{d+1} \Eseries_P(t^{-1}).
$$

\end{itemize}
}
\medskip

Our goal is to introduce a natural $q$-parameter into this picture\footnote{Note that Chapoton \cite{Chapoton} introduced a {\it different}, less canonical $q$-analogue, discussed in Remark~\ref{rmk:Chapoton-remark} below.}. Let $\kk$ be a field and let $\Zpoints \subseteq \kk^n$ be a finite locus of points in affine $n$-space over $\kk$. The {\em point-orbit method} (also known as the {\em orbit harmonics method}) is a machine which produces from $\Zpoints$ a graded quotient $R(\Zpoints)$ of the polynomial ring $S := \kk[x_1, \dots, x_n]$. 
Let $\II(\Zpoints) \subseteq S$ be the vanishing ideal of $\Zpoints$ given by
\begin{equation}
\label{eq:intro-vanishing-ideal}
\II(\Zpoints) := \{ f(\xx) \in S \,:\, f(\zz) = 0 \text{ for all $\zz \in \Zpoints$} \}.
\end{equation}
Since $\Zpoints$ is finite, by Lagrange interpolation one has an identification of the space $\kk[\Zpoints]$ of functions $\Zpoints \rightarrow \kk$ and the quotient ring $S/\II(\Zpoints)$. 
We let 
\begin{equation}
\label{eq:intro-point-orbit-ring}
R(\Zpoints) := S/\gr \, \II(\Zpoints)
\end{equation}
where $\gr \, \II(\Zpoints) \subseteq S$ is the {\em associated graded ideal} of $\II(\Zpoints)$, generated by the top degree homogeneous components $\tau(f)$ of nonzero polynomials $f \in \II(\Zpoints)$. The ideal $\gr \, \II(\Zpoints)$ is homogeneous, so the quotient ring $R(\Zpoints)$ is graded. Geometrically, the construction $S/\II(\Zpoints) \leadsto R(\Zpoints)$ is a flat deformation of the reduced subscheme $\Zpoints \subseteq \kk^n$ deforming linearly to a zero dimensional subscheme of degree $|\Zpoints|$ supported at the origin, as illustrated by  the picture below.

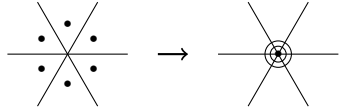
\begin{center}
 \begin{tikzpicture}[scale = 0.2]
\draw (-4,0) -- (4,0);
\draw (-2,-3.46) -- (2,3.46);
\draw (-2,3.46) -- (2,-3.46);

 \fontsize{5pt}{5pt} \selectfont
\node at (0,2) {$\bullet$};
\node at (0,-2) {$\bullet$};

\node at (-1.73,1) {$\bullet$};
\node at (-1.73,-1) {$\bullet$};
\node at (1.73,-1) {$\bullet$};
\node at (1.73,1) {$\bullet$};

\draw[thick, ->] (6,0) -- (8,0);

\draw (10,0) -- (18,0);
\draw (12,-3.46) -- (16,3.46);
\draw (12,3.46) -- (16,-3.46);

\draw (14,0) circle (15pt);
\draw(14,0) circle (25pt);
\node at (14,0) {$\bullet$};

 \end{tikzpicture}
\end{center}

The construction $\Zpoints \leadsto R(\Zpoints)$ is a canonical way to put a graded structure on the finite locus $\Zpoints$.
One has an isomorphism of $\kk$-vector spaces 
\begin{equation}
    \label{eq:orbit-harmonics-isomorphism}
    \kk[\Zpoints] = S/\II(\Zpoints) \cong R(\Zpoints).
\end{equation}
In particular, introducing a grading variable $q$, one finds that the {\it Hilbert series} 
$$
\Hilb(R(\Zpoints),q):=\sum_{d=0}^\infty  \dim_\kk R(\Zpoints)_d \cdot q^d
$$
is a polynomial in $q$ with nonnegative integer coefficients giving a geometrically motivated and canonical {\it $q$-analogue} of the cardinality of $\Zpoints$, that is,
\begin{equation}
\label{eq:hilb-is-a-q-analogue}
\left[ \Hilb(R(\Zpoints),q)\right]_{q=1}=
\dim_\kk R(\Zpoints) = \dim_\kk \kk[\Zpoints]=\# \Zpoints.
\end{equation}
Furthermore, 
when $\Zpoints$ is stable under the action of a finite subgroup $G \subseteq GL_n(\kk)$ (as suggested by the previous picture where $G \cong \symm_3$ and the transpositions in $\symm_3$ act by reflections in the three lines), the vector space isomorphism \eqref{eq:orbit-harmonics-isomorphism} is also a {\it Brauer isomorphism} of $\kk G$-modules, that is, $\kk[\Zpoints]$ and $R(\Zpoints)$ share the same composition multiplicities for all simple $\kk G$-modules.  In particular, when $|G|$ is invertible in $\kk^\times$, such as when $\kk$ characteristic zero, then \eqref{eq:orbit-harmonics-isomorphism} is a $\kk G$-module isomorphism.

However, the deformation $S/\II(\Zpoints) \leadsto R(\Zpoints)$ does not have perfect memory: since $R(\Zpoints)$ is not reduced for $|\Zpoints| > 1$, the rings $S/\II(\Zpoints)$ and $R(\Zpoints)$ are almost never isomorphic. 
Despite this amnesia, we will exhibit an unexpected multiplicative structure on direct sums of point-orbit rings (or rather their associated {\em harmonic spaces}) coming from  dilates of a lattice polytope.

Introduced by Kostant \cite{Kostant} in his study of coinvariant rings of finite reflection groups, the rings $R(\Zpoints)$ have proven to be an enormously fruitful source of graded quotients of $S$. The rings $R(\Zpoints)$ present the cohomology of Springer fibers \cite{garsia1992certain}, have ties to Coxeter-Catalan theory \cite{armstrong2015parking}, give generalized coinvariant rings with ties to Macdonald-theoretic delta operators \cite{haglund2018ordered, griffin2021ordered}, give graded modules over (0-)Hecke algebras \cite{huang2018ordered, huang2019hall}, prove cyclic sieving results \cite{oh2022cyclic}, and encode the Viennot shadow line construction of the Schensted correspondence \cite{rhoades2023increasing, liu2024viennot}. 
The rings $R(\Zpoints)$ have also been used to understand \cite{reineke2023zonotopal} Donaldson-Thomas invariants of symmetric quivers when $\Zpoints$ consists of lattice points in certain zonotopes.
In all of these examples, it is crucial to choose a point locus $\Zpoints \subseteq \kk^n$ with strategic organization.

Here we will take $\kk=\RR$ and consider the case where $\Zpoints = P \cap \ZZ^n$ consists of the integer points of an arbitrary lattice polytope $P \subseteq \RR^n$, together with its dilates.  Define two {\em $q$-Ehrhart series} for $P$, analogous to those in \eqref{eq:intro-ehrhart-count}, \eqref{eq:intro-interior-ehrhart-count}
\begin{align}
\label{eq:intro-q-ehrhart-series}
    \Eseries_P(t,q) &:= \sum_{m \, \geq \, 0} \Ehr_P(m;q) \cdot t^m, 
    \qquad \text{ where }\Ehr_P(m;q):=\Hilb(R(\ZZ^n \cap mP),q)\\
    \label{eq:intro-q-ehrhart-interior-series}
    \intEseries_P(t,q) &:= \sum_{m \, \geq \, 1} \bar{i}_P(m;q) 
    \cdot t^m, \qquad \text{ where }\bar{i}_P(m;q):=\Hilb(R(\ZZ^n \cap \interior{mP}),q).
\end{align}
Thus $\Eseries_P(t,q),\intEseries_P(t,q)$ lie in $\ZZ[q][[t]]$, and \eqref{eq:hilb-is-a-q-analogue} implies they specialize to $\Eseries_P(t), \intEseries_P(t)$ when $q \to 1$. The following is our main conjecture, on the form of $\Eseries_P(t,q), \intEseries_P(t,q)$.

\begin{conj}
\label{conj:intro-omnibus}
    Let $P$ be a $d$-dimensional lattice polytope in $\RR^n$. Then both of the series \eqref{eq:intro-q-ehrhart-series},\eqref{eq:intro-q-ehrhart-interior-series}
   lie in $\QQ(t,q)$, and are expressible as rational functions  
    $$
    \Eseries_P(t,q)=\frac{N_P(t,q)}{D_P(t,q)}
    \quad \text{ and } \quad 
        \intEseries_P(t,q)=\frac{\overline{N}_P(t,q)}{D_P(t,q)},
    $$
    over the same denominator of the form
    $D_P(t,q)=\prod_{i=1}^\nu (1-q^{a_i} t^{b_i})$, necessarily with $\nu \geq d+1$. Furthermore, there exists such an expression with all of these properties:
    \begin{itemize} 
    \item [(i)] the numerators $N_P(t,q), 
        \overline{N}_P(t,q)$ lie in $\ZZ[t,q]$,
    \item[(ii)] if $P$ is a lattice simplex, and $\nu=d+1$, then both numerators $N_P(t,q), \overline{N}_P(t,q)$ have nonnegative
coefficients as polynomials in $t,q$.
\end{itemize}
Lastly, one has this {\bf $q$-analogue of Ehrhart-Macdonald reciprocity}:
\begin{itemize}
\item[(iii)] the two series $\Eseries_P(t,q), \intEseries_P(t,q)$
determine each other via
$$
q^d  \cdot \intEseries_P(t,q) = (-1)^{d+1} \Eseries_P(t^{-1},q^{-1}).
$$
\end{itemize}
\end{conj}

\begin{remark}
\label{rmk: poles-and-degrees-in-t}
Conjecture~\ref{conj:intro-omnibus} requires $\nu \geq d+1$ since $\left[ \Eseries_P(q,t) \right]_{q=1}=\Eseries_P(t)$,
and part (i) of the Classical Ehrhart Theorem implies
$\Eseries_P(t)$
has a pole at $t=1$ of order $d+1$.
Similarly, letting $b:=\sum_{i=1}^\nu b_i$
be the $t$-degree of the denominator $D_P(t,q)$,
the numerator polynomials $N_P(t,q)$ and
$\overline{N}_P(t,q)$ would necessarily have $t$-degrees
at most $b-1$ and exactly $b$, respectively.

Regarding the fact that Conjecture~\ref{conj:intro-omnibus}(ii) has two hypotheses,  note that Example~\ref{ex: Kurylenko-Reeve-tetrahedra}
below suggests the existence of lattice simplices $P$
with $\nu > d+1$.
\end{remark}

\vskip.1in
\noindent
{\bf Example.}
    Consider a $1$-dimensional lattice polytope  $P \subset \RR^1$ of volume $v \geq 1$, that is, a line segment $P=[a,a+v]$ for some integer $a$.  Thus for $m \geq 0$, one has $mP=[ma,m(a+v)]$ and $$
    \ZZ^1 \cap mP=\{ma,ma+1,ma+2,\ldots,m(a+v)\}.
    $$
    Therefore one can calculate 
    \begin{align}
    \notag
    \II(\ZZ^1 \cap mP)&=( \,\, (x-ma)(x-ma-1) \cdots (x-ma-mv)\,\, ) \subset S=\RR[x],\\
    \notag
    \gr \, \II(\ZZ^1 \cap mP)&=(x^{mv+1}), \text{ so that } 
    R(\ZZ^1 \cap mP)=\RR[x]/(x^{mv+1}),\\
    \notag
    \Hilb(R(\ZZ^1 \cap mP);q)&=1+q+q^2+\cdots +q^{mv}=(1-q^{mv+1})/(1-q),\\
    \notag
    \Eseries_P(t,q)&=\sum_{m \geq 0} t^m \cdot (1-q^{mv+1})/(1-q)
    =\frac{1}{1-q} \sum_{m \geq 0} t^m(1-q^{mv+1})\\
     &=\frac{1}{1-q}\left( \frac{1}{1-t} - \frac{q}{1-tq^v} \right)
    = \frac{1+t(q+q^2+\cdots+q^{v-1})}{(1-t)(1-tq^v)}
     \label{eq:line-segment-Eseries}
     = \frac{1+t \cdot q [v-1]_q}{(1-t)(1-tq^v)},
    \end{align}
where we are using here a standard $q$-analogue of the positive integer $m$
$$
[m]_q:=1+q+q^2+\cdots+q^{m-1}.
$$
Note that $\Ehr_P(m)=mv+1$,
and hence, as expected, one has
$$
\Eseries_P(t)
=\sum_{m=0}^\infty (mv+1)t^m
=\frac{vt}{(1-t)^2}+\frac{1}{1-t}
=\frac{1+(v-1)t}{(1-t)^2}
=\left[\Eseries_P(t,q)\right]_{q=1}.
$$
Since $mP$ has interior lattice points $$
\ZZ^1 \cap \interior{mP}=\{ma+1,ma+2,\ldots,m(a+v)-1\},
$$
a similar calculation shows that
\begin{align}
\notag\intEseries_P(t,q)
    &= \sum_{m \geq 1} t^m \cdot (1-q^{mv-1})/(1-q)\\ 
    &= \frac{1}{1-q}\left( \frac{t}{1-t} - \frac{q^{v-1} t}{1-tq^v} \right)     =\frac{t(1+q+q^2\cdots+q^{v-2})+t^2q^{v-1}}{(1-t)(1-tq^v)}
\label{eq:line-segment-interior-Eseries}  =\frac{t[v-1]_q +t^2q^{v-1}}{(1-t)(1-tq^v)},
\end{align}
with 
$\left[ \intEseries_P(t,q) \right]_{q=1}= 
\frac{(v-1)t+t^2}{(1-t)^2}= \intEseries_P(t)$,
 as expected. 
Furthermore, the line segment $P$ is a lattice simplex,
and one can check that \eqref{eq:line-segment-Eseries}, \eqref{eq:line-segment-interior-Eseries}  are consistent with all parts (i),(ii),(iii) of Conjecture~\ref{conj:intro-omnibus}.

\subsection{The affine semigroup ring and the harmonic algebra}
For lattice $d$-polytopes $P \subset \RR^n$,
work of Stanley (see, e.g., \cite{Stanley-CCA}) explains the Classical Ehrhart Theorem by reinterpreting
$\Eseries_P(t)$
as the {\it Hilbert series} of a certain commutative graded algebra associated to $P$,
reviewed here.

One embeds $P$ in $\RR^{n+1}$ as $\{1\} \times P \subseteq \RR^1 \times \RR^n = \RR^{n+1}$. The {\em cone} over $P$ is  
$$
\cone(P) := \RR_{\geq 0} \cdot (\{1\} \times P).
$$
It is an affine polyhedral $(d+1)$-dimensional cone in $\RR^{n+1}$, and its lattice points $\ZZ^{n+1} \cap \cone(P)$ form a semigroup under addition. The {\em affine semigroup ring} of $P$ over a field $\kk$ is
\begin{equation}
    \label{eq:intro-affine-semigroup}
    A_P := \kk[\ZZ^{n+1} \cap \cone(P)].
\end{equation}
The first coordinate $x_0$ on $\RR^1 \times \RR^n =\RR^{n+1}$ endows $A_P$ with the structure of a graded algebra, 
having Krull dimension $d+1$.  Inside $A_P$, one has the (homogeneous) {\it interior ideal} $\overline{A}_P$, the $\kk$-span of all monomials corresponding to the interior lattice points $\interior{\cone(P)} \cap \ZZ^{n+1}$.
Then one has
\begin{align*}
\Eseries_P(t)&=\Hilb(A_P,t),\\
\intEseries_P(t)&=\Hilb(\overline{A}_P,t).
\end{align*}
Although $A_P$ is not always generated in degree one, there is always a tower of integral ring extensions
$$
\kk[\Theta]:=\kk[\theta_1,\ldots,\theta_{d+1}]
\subset B_P \subset A_P
$$
where $B_P$ is the subalgebra of $A_P$ generated by the monomials corresponding to the vertices of $P$, and $\Theta:=(\theta_1,\ldots,\theta_{d+1})$
is any choice of a linear (degree one) system of parameters for $B_P$; their existence is guaranteed by Noether's Normalization Lemma \cite{Noether}.  Most of the assertions (i),(ii),(iii),(iv) from the Classical Ehrhart Theorem are then explained by these four facts, respectively:
\begin{itemize}
    \item[(i)] $A_P$ is a finitely-generated $\kk[\Theta]$-module,
    and hence so is the ideal $\overline{A}_P$.
    \item[(ii)] $A_P$ is even a {\it free} $\kk[\Theta]$-module, as Hochster \cite{Hochster} showed $A_P$ is a {\it Cohen-Macaulay} ring.
\item[(iii)] When $P$ is a simplex, so that $\kk[\Theta]=B_P$, then $A_P$ has a $\kk[\Theta]$-basis given by monomials corresponding to the lattice points in the semi-open parallelepiped  \eqref{eq: semi-open-parallelepiped}.
\item[(iv)]
Work of Danilov \cite{Danilov} and Stanley \cite{Stanley-Diophantine} independently showed that $\overline{A}_P$ is isomorphic to the {\it canonical module} $\Omega A_P$ of the Cohen-Macaulay ring $A_P$.
\end{itemize}

This might inspire one to approach Conjecture~\ref{conj:intro-omnibus} by introducing a bigraded deformation of $A_P$. Since $A_P=\bigoplus_{m \geq 0} (A_P)_m$
where $(A_P)_m$ comes from the $x_{0}=m$ slice of $\cone(P)$
\begin{equation}
    \label{eq:intro-strata}
    \cone(P) \cap ( \{m\} \times \RR^n) =\{m\} \times  mP,
\end{equation}
one might look for such a bigraded object by defining a graded multiplication on the direct sum
\begin{equation}
    \label{eq:intro-direct-sum}
    \bigoplus_{m  \geq  0} R(mP \cap \ZZ^n).
\end{equation}
That is, one would like a way to `multiply' elements $f \in R(mP \cap \ZZ^n)$ and $f' \in R(m'P \cap \ZZ^n)$ to produce a new element $f \star f' \in R((m +m')P \cap \ZZ^n)$ in a fashion which respects the polynomial degrees of $f$ and $f'$. Geometrically, this corresponds to deforming the lattice points in $mP \times \{m\} \subseteq \cone(P)$ linearly onto the coning axis and defining a multiplication between fat point loci. The relevant picture is shown below when $P = [-1,1] \times [-1,1]$ is a square in $\RR^2$.
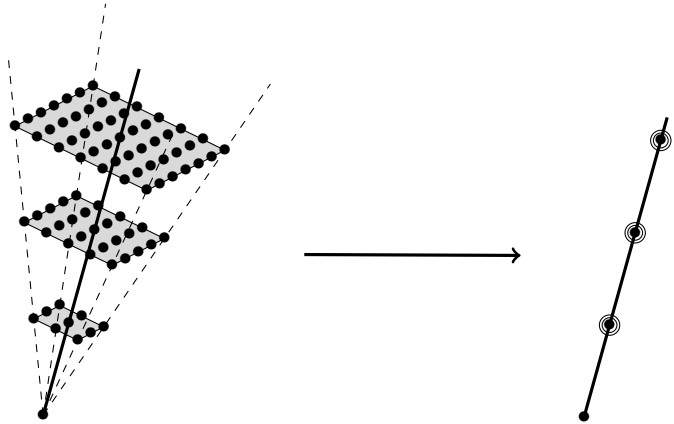
\begin{center}
    \tdplotsetmaincoords{70}{110}
    \begin{tikzpicture}[scale = 0.35]
        \tdplotsetrotatedcoords{90}{15}{-10}

        \draw[fill = gray!30, tdplot_rotated_coords] (1,1,4) -- (1,-1,4) -- (-1,-1,4) -- (-1,1,4) -- (1,1,4);
        \draw[fill = gray!30, tdplot_rotated_coords] (2,2,8) -- (2,-2,8) -- (-2,-2,8) -- (-2,2,8) -- (2,2,8);
         \draw[fill = gray!30, tdplot_rotated_coords] (3,3,12) -- (3,-3,12) -- (-3,-3,12) -- (-3,3,12) -- (3,3,12);
    
        \node[tdplot_rotated_coords] at (0,0,0) {$\bullet$};

        \node[tdplot_rotated_coords] at (0,0,4) {$\bullet$};
        \node[tdplot_rotated_coords] at (1,0,4) {$\bullet$};
        \node[tdplot_rotated_coords] at (1,1,4) {$\bullet$};
        \node[tdplot_rotated_coords] at (0,1,4) {$\bullet$};
        \node[tdplot_rotated_coords] at (-1,0,4) {$\bullet$};
        \node[tdplot_rotated_coords] at (0,-1,4) {$\bullet$};
        \node[tdplot_rotated_coords] at (1,-1,4) {$\bullet$};
        \node[tdplot_rotated_coords] at (-1,1,4) {$\bullet$};
        \node[tdplot_rotated_coords] at (-1,-1,4) {$\bullet$};

        \node[tdplot_rotated_coords] at (0,0,8) {$\bullet$};
        \node[tdplot_rotated_coords] at (1,0,8) {$\bullet$};
        \node[tdplot_rotated_coords] at (2,0,8) {$\bullet$};
        \node[tdplot_rotated_coords] at (0,1,8) {$\bullet$};
        \node[tdplot_rotated_coords] at (0,2,8) {$\bullet$};
        \node[tdplot_rotated_coords] at (1,1,8) {$\bullet$};
        \node[tdplot_rotated_coords] at (1,2,8) {$\bullet$};
        \node[tdplot_rotated_coords] at (2,1,8) {$\bullet$};
        \node[tdplot_rotated_coords] at (2,2,8) {$\bullet$};
        \node[tdplot_rotated_coords] at (-1,0,8) {$\bullet$};
        \node[tdplot_rotated_coords] at (-2,0,8) {$\bullet$};
        \node[tdplot_rotated_coords] at (0,-1,8) {$\bullet$};
        \node[tdplot_rotated_coords] at (0,-2,8) {$\bullet$};
        \node[tdplot_rotated_coords] at (-1,-1,8) {$\bullet$};
        \node[tdplot_rotated_coords] at (-2,-2,8) {$\bullet$};
        \node[tdplot_rotated_coords] at (-1,-2,8) {$\bullet$};
        \node[tdplot_rotated_coords] at (-2,-1,8) {$\bullet$};
        \node[tdplot_rotated_coords] at (1,-1,8) {$\bullet$};
        \node[tdplot_rotated_coords] at (2,-1,8) {$\bullet$};
        \node[tdplot_rotated_coords] at (1,-2,8) {$\bullet$};
        \node[tdplot_rotated_coords] at (2,-2,8) {$\bullet$};
        \node[tdplot_rotated_coords] at (-2,2,8) {$\bullet$};
        \node[tdplot_rotated_coords] at (-2,1,8) {$\bullet$};
        \node[tdplot_rotated_coords] at (-1,2,8) {$\bullet$};
        \node[tdplot_rotated_coords] at (-1,1,8) {$\bullet$};

        \node[tdplot_rotated_coords] at (-3,-3,12) {$\bullet$};
        \node[tdplot_rotated_coords] at (3,-3,12) {$\bullet$};
        \node[tdplot_rotated_coords] at (-3,3,12) {$\bullet$};
        \node[tdplot_rotated_coords] at (3,3,12) {$\bullet$};
        \node[tdplot_rotated_coords] at (2,2,12) {$\bullet$};
        \node[tdplot_rotated_coords] at (-2,2,12) {$\bullet$};
        \node[tdplot_rotated_coords] at (2,-2,12) {$\bullet$};
        \node[tdplot_rotated_coords] at (-2,-2,12) {$\bullet$};
        \node[tdplot_rotated_coords] at (-1,-1,12) {$\bullet$};
        \node[tdplot_rotated_coords] at (1,-1,12) {$\bullet$};
        \node[tdplot_rotated_coords] at (-1,1,12) {$\bullet$};
        \node[tdplot_rotated_coords] at (1,1,12) {$\bullet$};
        \node[tdplot_rotated_coords] at (0,0,12) {$\bullet$};
        \node[tdplot_rotated_coords] at (1,2,12) {$\bullet$};
        \node[tdplot_rotated_coords] at (1,3,12) {$\bullet$};
        \node[tdplot_rotated_coords] at (2,1,12) {$\bullet$};
        \node[tdplot_rotated_coords] at (3,1,12) {$\bullet$};
        \node[tdplot_rotated_coords] at (2,3,12) {$\bullet$};
        \node[tdplot_rotated_coords] at (3,2,12) {$\bullet$};
       \node[tdplot_rotated_coords] at (-1,2,12) {$\bullet$};
        \node[tdplot_rotated_coords] at (-1,3,12) {$\bullet$};
        \node[tdplot_rotated_coords] at (-2,1,12) {$\bullet$};
        \node[tdplot_rotated_coords] at (-3,1,12) {$\bullet$};
        \node[tdplot_rotated_coords] at (-2,3,12) {$\bullet$};
        \node[tdplot_rotated_coords] at (-3,2,12) {$\bullet$};
               \node[tdplot_rotated_coords] at (1,-2,12) {$\bullet$};
        \node[tdplot_rotated_coords] at (1,-3,12) {$\bullet$};
        \node[tdplot_rotated_coords] at (2,-1,12) {$\bullet$};
        \node[tdplot_rotated_coords] at (3,-1,12) {$\bullet$};
        \node[tdplot_rotated_coords] at (2,-3,12) {$\bullet$};
        \node[tdplot_rotated_coords] at (3,-2,12) {$\bullet$};
               \node[tdplot_rotated_coords] at (-1,-2,12) {$\bullet$};
        \node[tdplot_rotated_coords] at (-1,-3,12) {$\bullet$};
        \node[tdplot_rotated_coords] at (-2,-1,12) {$\bullet$};
        \node[tdplot_rotated_coords] at (-3,-1,12) {$\bullet$};
        \node[tdplot_rotated_coords] at (-2,-3,12) {$\bullet$};
        \node[tdplot_rotated_coords] at (-3,-2,12) {$\bullet$};
        \node[tdplot_rotated_coords] at (0,3,12) {$\bullet$};
        \node[tdplot_rotated_coords] at (0,2,12) {$\bullet$};
        \node[tdplot_rotated_coords] at (0,1,12) {$\bullet$};
        \node[tdplot_rotated_coords] at (0,-3,12) {$\bullet$};
        \node[tdplot_rotated_coords] at (0,-2,12) {$\bullet$};
        \node[tdplot_rotated_coords] at (0,-1,12) {$\bullet$};
        \node[tdplot_rotated_coords] at (1,0,12) {$\bullet$};
        \node[tdplot_rotated_coords] at (2,0,12) {$\bullet$};
        \node[tdplot_rotated_coords] at (3,0,12) {$\bullet$};
        \node[tdplot_rotated_coords] at (-1,0,12) {$\bullet$};
        \node[tdplot_rotated_coords] at (-2,0,12) {$\bullet$};
        \node[tdplot_rotated_coords] at (-3,0,12) {$\bullet$};

        \draw[very thick, tdplot_rotated_coords] (0,0,0) -- (0,0,15);
        \draw[dashed, tdplot_rotated_coords] (0,0,0) -- (3.75,3.75,15);
        \draw[dashed, tdplot_rotated_coords] (0,0,0) -- (-3.75,3.75,15);
        \draw[dashed, tdplot_rotated_coords] (0,0,0) -- (3.75,-3.75,15);
        \draw[dashed, tdplot_rotated_coords] (0,0,0) -- (-3.75,-3.75,15);

        \draw[very thick, tdplot_rotated_coords] (15,15,2.5) -- (15,15,15.5);

        \node[tdplot_rotated_coords] at (15,15,2.5) {$\bullet$};
        \node[tdplot_rotated_coords] at (15,15,6.5) {$\bullet$};
        \node[tdplot_rotated_coords] at (15,15,10.5) {$\bullet$};
        \node[tdplot_rotated_coords] at (15,15,14.5) {$\bullet$};
        \draw[tdplot_rotated_coords] (15,15,6.5) circle (8pt);
        \draw[tdplot_rotated_coords] (15,15,6.5) circle (11pt); 
        \draw[tdplot_rotated_coords] (15,15,10.5) circle (8pt);
        \draw[tdplot_rotated_coords] (15,15,10.5) circle (11pt);  
        \draw[tdplot_rotated_coords] (15,15,14.5) circle (8pt);
        \draw[tdplot_rotated_coords] (15,15,14.5) circle (11pt);  

        \draw[->, very thick, tdplot_rotated_coords] (6,6,8) -- (12,12,9);
    \end{tikzpicture}
\end{center}   
The na\"ive approach of looking at the ring $R(\cone(P) \cap \ZZ^{n+1})$ has two problems: 
\begin{itemize}
    \item the transformation $S/\II(\Zpoints) \leadsto R(\Zpoints)$ deforms a locus $\Zpoints \subseteq \kk^{n+1}$ to the origin, but we want to deform to the coning axis, and
    \item if $P$ is full-dimensional, then $\cone(P) \cap \ZZ^{n+1}$ is Zariski-dense and $\II(\cone(P) \cap \ZZ^{n+1}) = 0$. Consequently, the deformation $S/\II(\Zpoints) \leadsto R(\Zpoints)$ for $\Zpoints = P \cap \ZZ^{n+1}$ remembers only the affine span of $P$, not $P$ itself.
\end{itemize}

We achieve the task of defining a multiplication on the direct sum \eqref{eq:intro-direct-sum} in Section~\ref{sec:Harmonic} by replacing each summand $R(mP \cap \ZZ^n)$ with its {\em harmonic space}  (or {\em Macaulay inverse system}).
The resulting bigraded ring $\HHH_P$ is the {\em harmonic algebra} attached to $P$. The bigraded Hilbert series of $\HHH_P$ equals the $q$-Ehrhart series $\Eseries_P(t,q)$. Although not isomorphic to $A_P$ in general, we conjecture  that $\HHH_P$ enjoys analogous algebraic properties (finite generation, Cohen-Macaulayness, identification of the canonical module) that would explain much of Conjecture~\ref{conj:intro-omnibus}.

In spite of all of these properties remaining conjectural, we are able to
show that the $q$-Ehrhart series $\Eseries_P(t,q)$ and harmonic algebra $\HHH_P$ behave in a predictable fashion when performing three well-studied operations on lattice polytopes $P$: 
\begin{itemize}
\item {\it dilation} by an positive integer factor $d$, sending $P$ to $dP$,
\item {\it Cartesian product}, sending $P \subset \RR^n$ and $Q \subset \RR^m$ to $P \times Q \subset \RR^{n+m}$,
\item {\it free join},
sending $P,Q$ to  
$P * Q \subseteq \RR^{1+n+m}$ defined by
\begin{equation}
    P*Q := \{ (t,t\pp,(1-t)\qq) \,:\, 0 \leq t \leq 1, \, \pp \in P, \, \qq \in Q \}.
\end{equation}
\end{itemize}
The next result is proven in Section~\ref{sec: dilations-products-joins}, and relates the above operations to the constructions of
{\it Veronese subalgebras, Segre products} and {\it graded tensor products} for  harmonic algebras $\HHH_P, \HHH_Q$.

\begin{thm}
    \label{thm:intro-three-constructions-on-series}
    Let $P, Q$ be lattice polytopes.
    \begin{itemize}
        \item[(i)] For positive integers $d$, the dilation $dP$ has $\Eseries_{dP}(t,q)$ given by
        $$
        \Eseries_{dP}(t,q)=\sum_{m=0}^\infty \Ehr_P(dm;q) t^m. 
        $$   
        \item[(ii)] The Cartesian product $P\times Q$ has $\Eseries_{P \times Q}(t,q)$ given by the Hadamard product of series
        $$
        \Eseries_{P \times Q}(t,q)=
        \sum_{m=0}^\infty \Ehr_P(m;q) \cdot  \Ehr_Q(m;q) \cdot t^{m}.
        $$
        \item[(iii)] The free join $P * Q$ has $\Eseries_{P * Q}(t,q)$ given by
        $$
        \Eseries_{P * Q}(t,q)=
        \frac{1-t}{1-qt} \cdot 
        \Eseries_P(t,q) \cdot \Eseries_Q(t,q).
        $$
    \end{itemize}
\end{thm}

A further pleasant feature of the harmonic algebra arises when one considers Stanley's {\it two poset polytopes} \cite{stanley1986two} associated to a finite poset: its {\it order polytope} and its {\it chain polytope}.  Although these two lattice polytopes look very different, Stanley showed that they share the same Ehrhart series.  It will turn out (see Section~\ref{sec: chain-order-polytope} below) that all of our conjectures hold for both of these families of polytopes, that they share the {\it same $q$-Ehrhart series}, and even share the {\it same harmonic algebras}.  This is in contrast to the fact that their affine semigroup rings are generally {\it not} isomorphic.

\subsection{Harmonic spaces and Minkowski addition}
The fact that the harmonic algebra $\HHH_P$ is closed under multiplication is not obvious, and rests upon a surprising new property of harmonic spaces for arbitrary finite loci $\Zpoints, \Zpoints' \subseteq \kk^n$. Define their {\em Minkowski sum} to be the finite point locus
\begin{equation}
    \label{eq:intro-mikowski-sum-definition}
    \Zpoints + \Zpoints' := \{ \zz + \zz' \,:\, \zz \in \Zpoints, \, \zz' \in \Zpoints' \}.
\end{equation}
The point-orbit rings $R(\Zpoints), R(\Zpoints'),$ and $R(\Zpoints + \Zpoints')$ are graded quotients of $S = \kk[x_1, \dots, x_n]$. When $\kk$ has characteristic zero, the partial differentiation action of $S$ on itself gives rise to a $\kk$-linear perfect pairing $\langle -, - \rangle$ on each homogeneous component $S_d$ of $S$, and we may replace 
$R(\Zpoints), R(\Zpoints), R(\Zpoints + \Zpoints')$
by their {\it harmonic spaces}
\begin{equation}
    \label{eq:intro-minkowski-harmonics}
    V_\Zpoints := \gr \, \II(\Zpoints)^\perp, \quad \quad
    V_{\Zpoints'} := \gr \, \II(\Zpoints')^\perp, \quad \quad
    V_{\Zpoints + \Zpoints'} := \gr \, \II(\Zpoints + \Zpoints')^\perp.
\end{equation}
If $I \subseteq S$ is a homogeneous ideal, the harmonic space $I^\perp \subseteq S$ is a graded subspace with the same Hilbert series as $S/I$. Since elements of $I^\perp$ are honest polynomials $f$, whereas elements of $S/I$ are cosets $f + I$, working in $I^\perp$  avoids  coset-related issues which arise in proving, e.g., linear independence results. On the other hand, unlike the graded ring $S/I$, the subspace $I^\perp$ has the defect of not being closed under multiplication.  Nevertheless, in Section~\ref{sec:Minkowski} we prove the following.

\begin{thm}
\label{thm:minkowski-closure}
For any pair of finite point loci $\Zpoints, \Zpoints'$ in $\kk^n$ over any field $\kk$, one has 
$$
V_\Zpoints \cdot V_{\Zpoints'} 
    \subseteq V_{\Zpoints + \Zpoints'}.
$$
\end{thm}

\noindent
This containment $V_\Zpoints \cdot V_{\Zpoints'} 
    \subseteq V_{\Zpoints + \Zpoints'}$
may be interpreted as saying that the rings $R(\Zpoints)$  `remember' the structure of Minkowski sums via multiplication of their harmonic spaces. As
the deformation $S/\II(\Zpoints) \leadsto R(\Zpoints)$ does not respect ring structure, we find Theorem~\ref{thm:minkowski-closure} a bit unexpected\footnote{It is reminiscent of another  unexpected fact, about standard monomials for $\II(\Zpoints), \II(\Zpoints'), \II(\Zpoints+\Zpoints')$ with respect to a
chosen monomial ordering on $S=\RR[\xx]$, observed by
F. Gundlach \cite{Gundlach}; see Remark \ref{rmk:first-Gundlach-remark} below.}.

Note Theorem~\ref{thm:minkowski-closure}
is stated for finite point loci inside $\kk^n$ where $\kk$ is {\it any} field, not just $\kk=\RR$ or a field of characteristic zero. 
 This requires defining harmonic spaces $V_\Zpoints$ over an arbitrary field $\kk$, which occurs already in the theory of {\it Macaulay's inverse systems} over all fields $\kk$ discussed, e.g., in Geramita \cite{Geramita}, and reviewed in Section~\ref{sec:Minkowski} below.  These more general harmonic spaces $V_{\Zpoints}$ are defined not inside a polynomial ring over $\kk$, but rather in the {\em divided power algebra} over $\kk$.  When $\kk$ has characteristic zero, 
 these two rings are the same, and the definitions of $V_\Zpoints$ coincide.

The remainder of the paper is structured as follows.  

Section~\ref{sec:Background} reviews
commutative algebra of associated graded ideals and rings, Gr\"obner bases, harmonic spaces and Macaulay's inverse systems (first in
characteristic zero, and then over all fields).  It then briefly reviews some aspects of groups acting on rings and representation theory.

Section~\ref{sec:Ehrhart} defines the $q$-Ehrhart series $\Eseries_P(t,q)$, reviews Conjecture~\ref{conj:intro-omnibus}, and then examines several families of examples.  It also
incorporates symmetries of $P$ in an equivariant $q$-Ehrhart series $\Eseries^G_P(t,q)$, and computes some
highly symmetric examples, such as simplices
and cross-polytopes.

Section~\ref{sec:Minkowski} develops the proof of Theorem~\ref{thm:minkowski-closure} on Minkowski sums, working over arbitrary fields.

Section~\ref{sec:Harmonic} then uses Theorem~\ref{thm:minkowski-closure} to define
the harmonic algebra $\HHH_P$, and states
Conjecture~\ref{conj:harmonic-algebra-omnibus},
explaining its connection to Conjecture~\ref{conj:intro-omnibus}.
It also studies the examples of order polytopes and chain polytopes of finite posets,
mentioned earlier.

Section~\ref{sec: dilations-products-joins}
examines the behavior of $\Eseries_P(t,q)$
and $\HHH_P$ under the lattice polytope operations of dilation, Cartesian product, and joins.  In particular, it proves Theorem~\ref{thm:intro-three-constructions-on-series}.

\section*{Acknowledgements}
The authors thank Ben Braun, Winfried Bruns, Sarah Faridi, Takayuki Hibi, Katharina Jochemko, Martina Juhnke-Kubitzke, Sophie Rehberg and Raman Sanyal for helpful conversations.  They thank Ian Cavey for help in streamlining the proof of Theorem~\ref{thm:minkowski-closure}, and thank Christian Haase for pointing them to Balletti's database \cite{Balletti}.  They are especially grateful to Vadym Kurylenko for his computations appearing in Remark~\ref{rmk: annoying-triangle}, equation~\eqref{eq:Vadym-Reeve-tetrahedron-counterexample} and {\tt ExtraData.pdf}. Authors partially supported by NSF grants DMS-1745638 and DMS-2246846, respectively.

\section{Background}
\label{sec:Background}

\subsection{Commutative algebra} Let $\kk$ be a field, let $n \geq 0$, and let $\xx = (x_1, \dots, x_n)$ be a list of $n$ variables. We write $S := \kk[\xx] = \kk[x_1, \dots, x_n]$ for the polynomial ring in $x_1, \dots, x_n$ over $\kk$ with its standard grading induced by $\deg(x_i) = 1$ for all $i$.

We will consider various graded $\kk$-subspaces and quotients of $S$, as well as other rings. If $V = \bigoplus_{i \geq 0} V_i$ is a graded $\kk$-vector space with each piece $V_i$ finite-dimensional and $q$ is variable, the {\em Hilbert series} of $V$ is the formal power series
\begin{equation}
    \label{eq:background-hilbert}
    \Hilb(V,q) := \sum_{i  \geq  0} \dim_\kk(V_i) \cdot q^i.
\end{equation}
More generally, if $V = \bigoplus_{i,j \geq 0} V_{i,j}$ is a bigraded vector space, the {\em bigraded Hilbert series} is
\begin{equation}
    \label{eq:background-bigraded-hilbert}
    \Hilb(V,t,q) := \sum_{i,j  \geq  0} \dim_\kk(V_{i,j}) \cdot t^i q^j.
\end{equation}

Given $f \in S$ a nonzero polynomial, write $\tau(f) \in S$ for the top degree homogeneous component of $f$. That is, if $f = f_d + \cdots + f_1 + f_0$ with $f_i$ homogeneous of degree $i$ and $f_d \neq 0$, we have $\tau(f) = f_d$. 
If $I \subseteq S$ is an ideal, the {\em associated graded ideal} $\gr \, I \subseteq S$ is given by
\begin{equation}
    \label{eq:background-associated-graded}
    \gr \, I := ( \tau(f) \,:\, f \in I \setminus \{0\} ) \subseteq S.
\end{equation}
The ideal $\gr \, I \subseteq S$ is homogeneous by construction, so that $S / \gr \, I$ is a graded ring.
In fact, we wish to explain why it is isomorphic to the {\it associated graded ring} 
\begin{equation}
\label{eq:associated-graded-ring}
\gr_\FFF(S/I)=\bigoplus_{d=0}^\infty F_d/F_{d-1} = F_0 \oplus F_1/F_0 \oplus F_2/F_1 \oplus \cdots
\end{equation}
for the ascending filtration $\FFF=\{F_d\}_{d=0,1,2,\ldots}$ on $S/I$
\begin{equation}
\label{eq:ascending-filtration-on-S/I}
(\kk=)F_0 \subseteq F_1 \subseteq F_2 \subseteq \cdots \subseteq S/I
\end{equation}
where $F_d$ is the image of the polynomials $S_{\leq d}:=S_0\oplus S_1 \oplus \cdots \oplus S_d$ of degree at most $d$ under the surjection $S \twoheadrightarrow S/I$. 
Note that this filtration satisfies
$F_i \cdot F_j \subseteq F_{i+j}$, so that
the graded multiplication $F_i/F_{i-1} \times F_j/F_{j-1} \rightarrow F_{i+j}/F_{i+j-1}$
in $\gr_\FFF(S/I)$ is well-defined.

\begin{prop}
    \label{prop:graded-quotient-is-a-gr}
    Define a $\kk$-algebra map 
    $\varphi: S=\kk[\xx] \longrightarrow \gr_\FFF(S/I)$
sending $x_i \mapsto \bar{x}_i$ in $F_1/F_0$.
\begin{itemize}
    \item[(i)] The map $\varphi$ is surjective, with kernel $\gr\,I$, inducing an $\NN$-graded $\kk$-algebra isomorphism
$$
S/\gr\, I \cong \gr_\FFF(S/I).
$$
\item[(ii)] Consequently, any homogeneous polynomials $\{f_j\}_{j \in J}$
whose images $\{f_j +\gr\, I\}_{j \in J}$
give a $\kk$-basis of $S/\gr\, I$
will also have their images $\{f_j + I\}_{j \in J}$
giving a $\kk$-basis of $S/I$.

\item[(iii)]
In particular, whenever $S/I$ is Artinian, that is, $\dim_\kk S/I$ is {\it finite}, one can view {\it $\Hilb(S/\gr\, I,q)$ as a $q$-analogue of $\dim_\kk(S/I)$}
in this sense:
\begin{equation*}
\left[ \Hilb(S/\gr\, I,q) \right]_{q=1}
=\dim_\kk (S/\gr\, I) = \dim_\kk(S/I).
\end{equation*}
\end{itemize}
\end{prop}
\begin{proof}
For (i), the surjectivity of $\varphi$ holds because $S$ is generated by $x_1,\ldots,x_n$.  Hence one has $S_{\leq d}=\spn_\kk\{x_{i_1} \cdots x_{i_\ell}\}_{\ell \leq d}$,
and therefore $F_d=\spn_\kk\{\bar{x}_{i_1} \cdots \bar{x}_{i_\ell}\}_{\ell \leq d}$.

To show $\ker(\varphi) \supseteq \gr\,I$,
we check
$\tau(f) \in \ker(\varphi)$ for $f \in I$.  If
$f=\sum_{i=0}^d f_i$ and $f_d\neq 0$, then
$$
\varphi(\tau(f))=\overline{f_d(\xx)}
\equiv -\sum_{i=0}^{d-1} \overline{f_i(\xx)} \,\, \bmod{I},
$$
so that $\varphi(\tau(f)) =0$ in
$F_d/F_{d-1}$.

To prove $\ker(\varphi) \subseteq \gr\, I$,
it suffices to show every
{\it homogeneous} $f$ in $\ker(\varphi)$ lies
in $\gr \, I$. If $\deg(f)=d$, then $f \in \ker(\varphi)$ implies $\bar{f} \in F_{d-1} \bmod{I}$, say $f=f'+f''$ with
$f' \in S_{\leq d-1}, f'' \in I$.  But then
$f''=f-f'$ has $\tau(f'')=f$, so $f \in \gr\, I$.

For (ii), it suffices to check that for each $m \geq 0$, the set $\{f_j+I: j \in J, \deg(f_j) \leq m\}$ is a $\kk$-basis for $F_m=\mathrm{im}(S_{\leq m} \hookrightarrow S \twoheadrightarrow S/I)$.  However, this follows by induction on $m$, since our hypotheses imply that 
$\{f_j+F_{m-1}: j \in J, \deg(f_j)= m\}$ is a $\kk$-basis for $F_m/F_{m-1}$.

Assertion (iii) then follows immediately from (ii).
\end{proof}

We have another consequence in the case $\dim_\kk S/I = d < \infty$:  the quotient $S/I$
will be determined by the intersection $I_{\leq d-1}:=I \cap S_{\leq d-1}$ with 
the first summand in the $\kk$-vector space direct sum decomposition
\begin{equation}
\label{eq:polynomials-split-at-degree-d}
S=S_{\leq d-1} \oplus S_{\geq d}
\quad 
\text{ where }S_{\geq d}:=\bigoplus_{m=d}^\infty S_m=
(x_1,\ldots,x_n)^{d}.
\end{equation}

\begin{lemma}
\label{lem:technical-Artinian-GB-lemma}
Assume the ideal $I \subseteq S=\kk[\xx]$ has $d:=\dim_\kk S/I$ finite.  
\begin{itemize}
    \item[(i)] The inclusion $S_{\leq m}/I_{\leq m} \hookrightarrow S/I$ is an isomorphism for all $m \geq d-1$.
    \item[(ii)] The graded ring $S/\gr\, I$ vanishes in degrees $m \geq d$, that is, $(S/\gr\, I)_{\geq d}=0$.
\end{itemize}
\end{lemma}
\begin{proof}
 Recall that
$
S/\gr \, I \cong \gr_\FFF(S/I)=\bigoplus_{m=0}^\infty F_m/F_{m-1}
$
for the filtration $\FFF=\{F_m\}_{m=0,1,2,\ldots}$ on $S/I$
from \eqref{eq:ascending-filtration-on-S/I},
in which $F_m$ is the image of the composite
$S_{\leq m} \hookrightarrow S \twoheadrightarrow S/I$.
The composite has kernel $I_{\leq m}:=I \cap S_{\leq m}$,
so $S_{\leq m}/I_{\leq m} \cong  F_m$, and hence
\begin{equation}
\label{eq:inhomogeneous-ideal-cumulative-hilb}
\dim_\kk S_{\leq m}/I_{\leq m} 
= \dim_\kk  F_m = \dim_\kk (S/\gr \, I)_{\leq m}. 
\end{equation}
Note that since $S/\gr \, I$ is a graded $\kk$-algebra generated
in degree one, its nonzero graded components form an initial segment of degrees.  Since $\dim_\kk S/\gr \, I=\dim_\kk S/I=d$, one concludes that $\dim_\kk(S/\gr \, I)_{\leq m}=d$ for all $m \geq d-1$, and consequently,
$\dim_\kk(S/\gr \, I)_m=0$ for all $m \geq d$, proving assertion (ii).
For assertion (i), note \eqref{eq:inhomogeneous-ideal-cumulative-hilb} also shows
that $\dim_\kk S_{\leq m}/I_{\leq m}=d$ for all $m \geq d-1$, and therefore the inclusion $S_{\leq m}/I_{\leq m}
\hookrightarrow S/I$ must be an isomorphism.
\end{proof}

If $I \subseteq S$ is an ideal with generating set $I = (f_1, \dots, f_s)$, we have $(\tau(f_1), \dots, \tau(f_s)) \subseteq \gr \, I$, but this containment is strict in general. A finite generating set for $\gr \, I$ may be computed using graded term orderings and Gr\"obner theory as follows; see Cox, Little, O'Shea \cite{CoxLittleOShea} for more background.

\begin{definition}\rm
A total order $\preceq$ on the monomials of $S$ is a {\em term order} if 
$1 \preceq m$ for all monomials $m$, and whenever $m_1 \preceq m_2$ one also has $m_1 m_3 \preceq m_2 m_3$ for all monomials $m_1, m_2, m_3$.
\end{definition}

For $\prec$ a term order and $f \in S \setminus \{0\}$, write $\init_\prec(f)$ for the $\prec$-largest monomial appearing in $f$.

\begin{example}
The {\em lexicographical} term order is defined by $x_1^{a_1} \cdots x_n^{a_n} <_{lex} x_1^{b_1} \cdots x_n^{b_n}$ if there exists $i$ such that $a_1 = b_1, \dots, a_{i-1} = b_{i-1},$ and $a_i < b_i$.  The {\em graded lex} term order is defined by
\begin{center}
    $x^a <_{grlex} x^b$ if 
    $\deg x^a < \deg x^b$ or 
    ($\deg x^a = \deg x^b$ and $x^a <_{lex} x^b$)
\end{center}
where $a = (a_1, \dots, a_n), b = (b_1, \dots, b_n) \in \ZZ_{\geq 0}^n$. 
\end{example}

A term order $\preceq$ is {\em graded} if $m \prec m'$ whenever $\deg m < \deg m'$.  Equivalently, $\preceq$ is graded if and only if one has for all $f \in S \setminus \{0\}$ that
\begin{equation}
\label{eq:graded-term-order-init-relation}
\init_\prec(f) = \init_\prec(\tau(f)).
\end{equation}
\begin{example}
Lexicographic order $<_{lex}$ is {\it not} graded for $n \geq 2$, but $<_{grlex}$ is always graded. 
\end{example}

Let $I \subseteq S$ be an ideal and let $\prec$ be a term order. The {\em initial ideal} of $I$ is the monomial ideal
\begin{equation}
    \label{eq:background-initial-ideal}
    \init_\prec(I) := ( \init_\prec(f) \,:\, f \in I \setminus\{0\} ) \subseteq S.
\end{equation}

\begin{definition} \rm
A finite subset $G \subseteq I$ is a {\em Gr\"obner basis} of $I$ if $\init_\prec(I) = (\init_\prec(g) \,:\, g \in G)$. Equivalently, for
every $f$ in $I$ there exists some $g$ in $G$
with $\init_\prec(g)$ dividing $\init_\prec(f)$.
\end{definition}

One can show that a Gr\"obner basis $G$
for $I$ always generates $I$ as an ideal. We also have the following useful $\kk$-basis
for $S/I$.  Say that monomial $m$ in $S$ is a {\em standard monomial} of $I$ (with respect to $\prec$) if $m \notin \init_\prec(I)$. Equivalently, this means that $\init_\prec(g)$ does not divide $m$ for all $g \in G$, where $G$ is a Gr\"obner basis of $I$ with respect to $\prec$. Then the set 
\begin{equation}
    \label{eq:background-standard-monomial-basis}\
    \{m+I \,:\, \text{$m$ a standard monomial of $I$} \}
\end{equation}
is a $\kk$-basis of the quotient ring $S/I$.  It is uniquely determined by the term order $\prec$, and called the {\em standard monomial basis} of $S/I$. 
The following can then be proven easily using \eqref{eq:graded-term-order-init-relation}.

\begin{prop}
\label{prop:graded-term-orders-and-gr}
Fix a 
graded term order $\prec$ on $S$.
Then for any ideal $I \subset S$, 
a Gr\"obner basis $G$ for $I$
(with respect to $\prec$) gives rise to a 
Gr\"obner basis with respect to $\prec$
$$
\tau(G):=(\tau(g) \,:\, g \in G)
$$
for the homogeneous ideal
$\gr \, I$.  Consequently, $\tau(G)$ also generates $\gr \, I$ as an ideal:
$$
\gr \, I=(\tau(g): g \in G)
$$
Furthermore, $I, \gr \, I$ share the same set of standard monomials $\BBB$ with respect to $\prec$, 
which descend to $\kk$-bases $\{m+I\}_{m \in \BBB}$ and $\{m+\gr \, I\}_{m \in \BBB}$
for $S/I$ and $S/\gr \, I$, respectively.
\end{prop}

\begin{remark}
    \label{rmk:first-Gundlach-remark}
When $I=\II(\Zpoints)$ is the 
vanishing ideal in $S=\kk[\xx]$ for a finite point set $\Zpoints \subseteq \kk^n$, Gundlach
\cite[Lem.~4]{Gundlach} gives a very interesting alternate characterization of the $\prec$-standard monomials for $\II(\Zpoints)$.
Let $\kk[\Zpoints]$ denote the $\kk$-vector space of 
all functions $f:\Zpoints \rightarrow \kk$, with pointwise addition and $\kk$-scaling.  Endow $\kk^\Zpoints$ with a nondegenerate $\kk$-bilinear form
$(-,-): \kk[\Zpoints] \times \kk[\Zpoints] \rightarrow \kk$ given by
$$
(f_1,f_2):=\sum_{\zz \in \Zpoints} f_1(z) f_2(z).
$$
Let $U^\perp$ denote perp with respect to $(-,-)$ for $\kk$-subspaces $U \subseteq \kk^\Zpoints$.
Nondegeneracy of $(-,-)$ implies 
$U_1 \subsetneq U_2 \Leftrightarrow U_1^\perp \supsetneq U_2^\perp$.
By multivariate Lagrange interpolation, the map $S \rightarrow \kk[\Zpoints]$
restricting polynomials $f(\xx)$ to functions $\bar{f}$ 
on $\Zpoints$ is surjective. Since its kernel is $\II(\Zpoints)$, it gives a $\kk$-vector space isomorphism $S/\II(\Zpoints) \rightarrow \kk[\Zpoints]$.
This implies, that for any $f \in \kk[\Zpoints] \setminus \{0\}$, there must exist some
monomials $m$ in $S$ for which $(f,\bar{m}) \neq 0$.
Consequently, {\it having fixed the monomial order $\prec$}, for each $f \in \kk[\Zpoints] \setminus \{0\}$, there will be a $\prec$-smallest such monomial associated to $f$,  since $\prec$ is a well-ordering:
$$
\sm_\prec(f):=\min_\prec\{\text{ monomials }m\in S:(f,\bar{m}) \neq 0 \in \kk \,\,\}.
$$

\begin{prop}
  \label{prop:Gundlach-prop} 
  \cite[Lem.4]{Gundlach}
  For any finite subset $\Zpoints \subset \kk^n$,
  and for any choice of monomial order $\prec$ on $S=\kk[\xx]$, one has this equality of sets: 
  $$
  \{ \prec\text{-standard monomials for }\II(\Zpoints)\} = \{\sm_\prec(f): f \in \kk[\Zpoints] \setminus \{0\}\}.
  $$
\end{prop}
\begin{proof}
For any monomial $m$ in $S$ one has the following:
\begin{align*}
m \text{ is }\prec\text{-standard for }\II(\Zpoints)
 & \quad \Leftrightarrow \quad
    \not\exists \,\,  g \in \II(\Zpoints) \text{ with }
    \init_\prec(g) \text{ dividing } m\\
 & \quad \Leftrightarrow \quad
    \not\exists \,\, \hat{g} \in \II(\Zpoints) \text{ with }
    \init_\prec(\hat{g})=m\\
     & \quad \Leftrightarrow \quad
    \not\exists \,\,  \hat{g} \in \II(\Zpoints) \text{ of the form }\hat{g}=m+\sum_{m': m' \prec m} c_{m'} m'\text{ with }c_{m'} \in \kk\\
    & \quad \Leftrightarrow \quad
    \bar{m} \not\in \spn_\kk\{\bar{m}': m' \prec m\} \text{ inside } S/\II(\Zpoints) =\kk[\Zpoints].\\
     & \quad \Leftrightarrow \quad
    \spn_\kk\{\bar{m}': m' \prec m\} \subsetneq 
    \spn_\kk\left( \{\bar{m}\} \cup \{\bar{m}': m' \prec m\}\right)  \\
    & \quad \Leftrightarrow \quad
    \spn_\kk\{\bar{m}': m' \prec m\}^\perp \supsetneq 
    \spn_\kk\left( \{\bar{m}\} \cup \{\bar{m}': m' \prec m\}\right)^\perp  \\
    & \quad \Leftrightarrow \quad
    \exists f \in \spn_\kk\{\bar{m}': m' \prec m\}^\perp \setminus 
    \spn_\kk\left( \{\bar{m}\} \cup \{\bar{m}': m' \prec m\}\right)^\perp  \\
    & \quad \Leftrightarrow \quad
    \exists f \in \kk[\Zpoints] \setminus \{0\}
    \text{ with }\sm(f)=m.\qedhere
\end{align*}
\end{proof}

\end{remark}

\subsection{Homogeneous harmonics in characteristic zero}
\label{sec:harmonics-char-zero}

A good reference for much of this material is Geramita \cite[\S 2]{Geramita}.
Let $\kk$ be a field of characteristic zero.  We wish to set up two polynomial rings over $\kk$, one of which acts on the other by partial derivatives.  Let $\kk^n$ and its $\kk$-dual $(\kk^n)^*$ have dual ordered $\kk$-bases $(y_1,\ldots,y_n)$ and $(x_1,\ldots,x_n)$
with respect to the usual $\kk$-bilinear pairing of functionals and vectors 
\begin{equation}
\label{eq:functional-vector-pairing}
\langle -,-\rangle: (\kk^n)^* \times 
\kk^n \rightarrow \kk
\end{equation}
so that $\langle x_i, y_j\rangle=\delta_{ij}$.
If one considers the polynomial algebras
\begin{align*}
S&:=\kk[\xx]=\kk[x_1,\ldots,x_n],\\
\Div&:=\kk[\yy]=\kk[y_1,\ldots,y_n]
\end{align*}
then one can extend this to a $\Div$-valued pairing 
\begin{equation}
\label{eq:apolarity-pairing}
\odot: S \times \Div \rightarrow \Div
\end{equation}
by requiring that each $x_i$ act on $\Div$ as a derivation. That is,
$x_i$ acts as $\frac{\partial}{\partial y_i}: \Div \rightarrow \Div$, and for polynomials $f(\xx), g(\yy)$, one has
$$
f(\xx) \odot g(\yy):=f\left(\frac{\partial}{\partial y_1},\ldots,\frac{\partial}{\partial y_n}\right) g(\yy).
$$
In this way, one obtains an $S$-module structure on $\Div$.  This $S$-module structure on $\Div$ is degree-lowering for the usual gradings on $S,\Div$ in which $\deg(x_i)=\deg(y_j)=1$, in the sense that it restricts to a map
$
\odot: S_m \times \Div_{m'} \rightarrow \Div_{m'-m}.
$
This lets one extend the pairing
$\langle -,- \rangle$ from \eqref{eq:functional-vector-pairing} to a $\kk$-bilinear pairing
$\langle -, - \rangle: S \times \Div \rightarrow \kk$ defined by
\begin{equation}
    \label{eq:background-pairing}
    \langle f(\xx), g(\yy) \rangle :=
    \text{the constant term of $f \odot g$}
\end{equation}
Employing an exponential notation for monomials $\xx^\aa:=x_1^{a_1} \cdots x_n^{a_n}$
in $S$, where $\aa=(a_1,\ldots,a_n)$ lies in $\{0,1,2,\ldots\}^n$, and 
similarly for monomials $\yy^\bb$ in $\Div$,
one can check that 
$$
\xx^\aa \odot \yy^\bb
=\begin{cases}
\prod_{i=1}^n \frac{b_i!}{(b_i-a_i)!}
\cdot \yy^{\bb-\aa} &\text{ if }a_i \leq b_i \text{ for }i=1,\ldots,n,\\
0& \text{ otherwise.}
\end{cases}
$$
One sees that $\langle -,-\rangle$ pairs orthonormally the $\kk$-dual bases
$
\{ \xx^\aa \}
\text{ and }
\left\{ \frac{\yy^\aa}{a_1! \cdots a_n!}\right\},
$
and hence restricts to a 
perfect $\kk$-linear pairing on the (finite-dimensional!) spaces 
\begin{align*}
\langle -,-\rangle: &S_m \times \Div_m \longrightarrow \kk\\
\langle -,-\rangle: &S_{\leq m} \times \Div_{\leq m} \longrightarrow \kk.
\end{align*}
Note that $S_m$ and $\Div_{m'}$ are perpendicular with respect to the pairing $\langle -,- \rangle$
whenever $m \neq m'$.

\begin{definition} \rm 
\label{def:harmonic-space}
For any homogeneous ideal $I \subseteq S=\kk[\xx]$, the {\it harmonic space} (or {\it Macaulay inverse system}) $I^\perp \subseteq \Div=\kk[\yy]$ is the graded vector space
\begin{align}
    \label{eq:background-harmonic-space}
    I^\perp &:= \{ g(\yy) \in \Div\, :\, \langle f(\xx), g(\yy) \rangle = 0 \text{ for all $f \in I$} \}\\
    \label{eq:background-harmonic-space-reformulated}
    &=\{ g(\yy) \in \Div\, :\, f(\xx) \odot g(\yy) = 0 \text{ for all $f \in I$} \}.
\end{align}
\end{definition}
The equality of the two sets on the right in
\eqref{eq:background-harmonic-space}, 
\eqref{eq:background-harmonic-space-reformulated}
is justified as follows.  If $f \odot g=0$ then $\langle f,g \rangle=0$, showing the set 
from 
\eqref{eq:background-harmonic-space-reformulated}
is contained in the set from \eqref{eq:background-harmonic-space}.  For the reverse inclusion, note
that if $f(\xx) \odot g(\yy) \neq 0$ for some $f(\xx)$ in $I$,
say $f(\xx) \odot g(\yy)  =\sum_\aa c_\aa \yy^\aa$ with some $c_\aa \neq 0$, then the ideal $I$ contains $\xx^\aa f(\xx)$ with 
$\langle \xx^\aa f(\xx), g(\yy)\rangle = c_\aa \prod_i a_i! \neq 0$.

Note that in each degree $m$, the perfect pairing $\langle -,-\rangle: S_m \times \Div_m \rightarrow \kk$ gives a $\kk$-vector space isomorphism $\Div_m \rightarrow S_m^*$ 
sending $g(\yy) \mapsto \langle -,g\rangle$.
This induces a $\kk$-vector space isomorphism 
\begin{equation}
\label{eq:perp-contragredient-to-quotient}
I^\perp_m \rightarrow (S_m/I_m)^*,
\end{equation}
showing that $\dim_\kk I^\perp = \dim_\kk (S_m/I_m)^*$.  Hence as graded $\kk$-vector spaces one has
\begin{equation}
\label{eq:perp-and-quotient-share-Hilbert-series}
\Hilb(I^\perp,q)=\Hilb(S/I,q).
\end{equation}

\subsection{Homogeneous harmonics for all fields: divided powers}
\label{sec:harmonics-over-all-fields}

In order to define harmonic spaces over arbitrary fields, as needed in Theorem~\ref{thm:minkowski-closure}, we will need to replace $\Div=\kk[\yy]$ with a {\it divided power algebra} over a field $\kk$.  We therefore review divided power algebras here; a reader who is content with seeing 
Theorem~\ref{thm:minkowski-closure} stated and/or proven only
in characteristic zero can mostly skip this section. Useful references for this material are Eisenbud \cite[A2.4]{Eisenbud} and Geramita \cite[\S 9]{Geramita}.

\begin{definition} \rm
Let $\kk$ be any field. For $n=0,1,2,\ldots$, let $\yy = (y_1, \dots, y_n)$ be a list of $n$ variables, thought of as the $\kk$-basis for $\kk^n$.
Then the {\em divided power algebra} of rank $n$ over $\kk$ is defined as a $\kk$-vector space $\Div_\kk(\yy)$ with ``monomial" $\kk$-basis given by the symbols $\yy^{(\aa)}:=y_1^{(a_1)} \cdots y_n^{(a_n)}$ for $\aa=(a_1, \dots, a_n) \in \{0,1,2,\ldots\}^n$.
One can then define a multiplication on $\Div_\kk(\yy)$ which is $\kk$-bilinear and determined on monomials by the rule
\begin{equation}
\label{eq:divided-power-multiplication}
   \yy^{(\aa)} \cdot \yy^{(\bb)}=(y_1^{(a_1)} \cdots y_n^{(a_n)}) \cdot (y_1^{(b_1)} \cdots y_n^{(b_n)}) := {a_1 + b_1 \choose a_1,b_1} \cdots {a_n + b_n \choose a_n,b_n} y_1^{(a_1 + b_1)} \cdots y_n^{(a_n + b_n)}
\end{equation}
where the binomial coefficients ${a_i + b_i \choose a_i, b_i}$ are regarded as elements of $\kk$ in the natural way;
some will vanish when $\kk$ has positive characteristic.
This makes $\Div_\kk(\yy)$ an associative, commutative $\kk$-algebra with unit $1 = y_1^{(0)} \cdots y_n^{(0)}$.  It has a grading
$\Div_\kk(\yy) = \bigoplus_{m \geq 0} \Div_\kk(\yy)_m$ where $\Div_\kk(\yy)_m$ has $\kk$-basis $\{\yy^{(\aa)} \,:\, a_1 + \cdots + a_n = m \}$.
\end{definition}

Roughly speaking, the symbol $y_i^{(d)} \in \Div_\kk(\yy)$ plays the role of $`` y_i^d/d! "$, even when $d! = 0$ in $\kk$. When $\kk$ has characteristic zero, the
 $\kk$-vector space isomorphism $\Div_\kk(\yy) \cong \kk[\yy]$ given by
$$
\begin{array}{rcl}
\Div_\kk(\yy) &\longrightarrow& \kk[\yy]\\
\yy^{(\aa)}=y_1^{(a_1)} y_2^{(a_2)} \cdots  y_n^{(a_n)} &\longmapsto& 
\frac{ y_1^{a_1} y_2^{a_2 }\cdots y_n^{a_n}}
{a_1! a_2! \cdots a_n!}
\end{array}
$$
is a ring isomorphism.  In fact, with conventions $y_i^{(0)}:=1$ and $y_i^{(1)}:=y_i$, the set map  $y_i \mapsto y_i^{(d)}$ for $i=1,2,\ldots,n$ extends to what is called a {\it system of divided powers} on $\Div_\kk(\yy)$:  a collection of maps for $d=0,1,2,\ldots$
$$
\begin{array}{rcl}
\Div_\kk(\yy)_m &\longrightarrow &\Div_\kk(\yy)_{dm}\\
a &\longmapsto &a^{(d)}
\end{array}
$$
satisfying these axioms (modeled on properties of the maps $a \mapsto \frac{a^d}{d!}$ that exist whenever $\kk \supseteq \QQ$):
\begin{align}
    a^{(0)}&=1,
    a^{(1)}=z\\
\label{eq:divided-law-of-exponents}
a^{(d)} a^{(e)} &= \binom{d+e}{d,e} a^{(d+e)}\\
    \left( a^{(d)} \right)^{(e)} 
    &= \frac{1}{e!} \binom{de}{d,d,\ldots,d} a^{(de)}\\
    (a b)^{(d)}&=d! \cdot a^{(d)} b^{(d)} 
    \\
\label{eq:beginner-binomial-thm}    (a+b)^{(d)}&=\sum_{\substack{(d_1,d_2):\\d_1+d_2=d}} a^{(d_1)} b^{(d_2)}
\end{align}
For example, the reader might wish to check that
iterating \eqref{eq:divided-law-of-exponents} implies
$a^d=(a^{(1)})^d=d! \cdot a^{(d)}$, so that one
has no choice but to define $a^{(d)}=a^d/d!$ 
whenever $d! \in \kk^\times$.
Similarly, if one iterates \eqref{eq:beginner-binomial-thm},
which  Eisenbud \cite[A2.4]{Eisenbud} calls the ``beginner's binomial theorem", one obtains the 
``beginner's multinomial theorem":
\begin{equation}
\label{eq:beginner-multinomial-thm} 
(a_1+\cdots+a_m)^{(d)}=\sum_{\substack{(d_1,\ldots,d_m):\\ d_1+\cdots+d_m=d}} a_1^{(d_1)} \cdots a_m^{(d_m)}.  
\end{equation}
It is also not hard to check that one has a graded $\kk$-algebra isomorphism
\begin{equation}
    \Div_\kk(\yy) \cong \Div_\kk(y_1) \otimes_\kk \cdots \otimes_\kk \Div_\kk(y_n).
\end{equation}

For the sake of defining harmonics and inverse systems, let $S: = \kk[\xx] = \kk[x_1, \dots, x_n]$ as before. The algebra $\Div_\kk(\yy)$ attains a unique $S$-module structure $\odot: S \times \Div_\kk(\yy) \rightarrow \Div_\kk(\yy)$ by having $x_i$ act on $\Div_\kk(\yy)$ as a derivation, extending the rule
\begin{equation}
    x_i \odot y_j^{(k)} = \begin{cases}
        y_i^{(k-1)} & \text{ if }i = j \text{ and } k \geq 1, \\
        0 & \text{otherwise.}
    \end{cases}
\end{equation}
This gives rise to a $\kk$-bilinear pairing $\langle - , - \rangle: S \times \Div_\kk(\yy) \rightarrow R$ given by 
\begin{equation}
    \langle f, g \rangle := \text{ the constant term of $f \odot g$}
\end{equation}
under which
\begin{equation}
\label{eq:divided-pairing}
    \langle \xx^\aa, \yy^{(\bb)} \rangle = \langle x_1^{a_1} \cdots x_n^{a_n}, y_1^{(b_1)} \cdots y_n^{(b_n)} \rangle = \begin{cases} 1 & \text{if $a_i = b_i$ for all $i$,} \\ 0 & \text{otherwise.}
    \end{cases}
\end{equation}
In particular, 
the $\kk$-bilinear pairing $\langle - , - \rangle: S \times \Div_\kk(\yy) \rightarrow \kk$
generalizes the one from \eqref{eq:background-pairing} when $\kk \supseteq \QQ$.  It again leads to perfect $\kk$-bilinear pairings on these finite-dimensional spaces:
\begin{align*}
\langle -,-\rangle: &S_m \times \Div_\kk(\yy)_m \longrightarrow \kk\\
\langle -,-\rangle: &S_{\leq m} \times \Div_\kk(\yy)_{\leq m} \longrightarrow \kk.
\end{align*}

This leads to the following
generalization of Definition~\ref{def:harmonic-space}.

\begin{definition} \rm 
For $\kk$ any field and any homogeneous ideal $I \subseteq S=\kk[\xx]$, the {\it harmonic space} (or {\it Macaulay inverse system}) $I^\perp \subseteq \Div_\kk(\yy)$ is the graded $\kk$-vector space
\begin{align*}
I^\perp &:= \{ g(\yy) \in \Div_\kk(\yy) \, :\, \langle f(\xx), g(\yy) \rangle = 0 \text{ for all $f \in I$} \}\\
    &=\{ g(\yy) \in \Div_\kk(\yy) \, :\, f(\xx) \odot g(\yy) = 0 \text{ for all $f \in I$} \}.
\end{align*}
\end{definition}

Note that, as before, the perfect pairing $\langle -,-\rangle: S_m \times \Div_\kk(\yy)_m \rightarrow \kk$ gives a $\kk$-linear isomorphism $\Div_\kk(\yy)_m \rightarrow S_m^*$ 
sending $g(\yy) \mapsto \langle -,g\rangle$.
As in \eqref{eq:perp-contragredient-to-quotient},
this induces a $\kk$-linear isomorphism 
\begin{equation}
\label{eq:perp-contragredient-to-quotient-divided}
I^\perp_m \rightarrow (S_m/I_m)^*,
\end{equation}
showing $\dim_\kk I^\perp = \dim_\kk (S_m/I_m)^*$.  Hence one has this generalization of \eqref{eq:perp-and-quotient-share-Hilbert-series}:
\begin{equation}
\label{eq:perp-and-quotient-share-Hilbert-series-divided}
\Hilb(I^\perp,q)=\Hilb(S/I,q).
\end{equation}

\subsection{Symmetry}
The natural action of $GL_n(\kk)$ on $\kk^n$ and a given basis $y_1,\ldots,y_n$ induces a left-action on the polynomials $\Div = \kk[y_1, \dots, y_n]$ by linear substitutions.  The contragredient action on $(\kk^n)^*$ precomposes functionals with $h^{-1}$,
that is, sending $h: f \mapsto f\circ h^{-1}$,
thereby acting on $x_1,\ldots,x_n$, as well as on the polynomials $S=\kk[x_1,\ldots,x_n]$.  Explicitly, 
if $h$ in $GL_n(\kk)$ acts in the basis $y_1,\ldots,y_n$ via the matrix $A$ in $\kk^{n \times n}$, then 
\begin{align*}
h&: g(\yy)\longmapsto g(A\yy),\\
h&: f(\xx) \longmapsto f((A^{-1})^t\xx).
\end{align*}
Note that the pairing
$\langle -,-\rangle: (\kk^n)^* \times 
\kk^n \rightarrow \kk
$
between functionals and vectors in 
\eqref{eq:functional-vector-pairing}
satisfies a certain {\it invariance} with respect to
these $GL_n(\kk)$-actions: for any linear functional $f$
in $(\kk^n)^*$, vector $y$ in $\kk^n$, and $h$ in $GL_n(k)$, one has
$$
\langle h(f),h(y) \rangle
=(f \circ h^{-1})(h(y)) =f(y)
=\langle f,y \rangle.
$$
Consequently, the pairings $\odot$ and $\langle -,- \rangle$ are similarly
invariant with respect to the $GL_n(\kk)$-action:
\begin{align}
h(f(\xx)) \odot h(g(\yy))&= f(\xx) \odot g(\yy),\\
\label{eq:invariance-of-perfect-pairings}
\langle h(f(\xx)), h(g(\yy)) \rangle&= 
\langle f(\xx), g(\yy) \rangle.
\end{align}

\begin{remark}
For arbitrary fields $\kk$ when one replaces the
polynomial algebra $\Div=\kk[\yy]$ with
the divided power algebra $\Div_\kk(\yy)$,
it is still true the $GL_n(\kk)$-action on 
$\kk^n=\Div_\kk(\yy)_1$ extends to an action via
graded $\kk$-algebra automorphisms on all of $\Div_\kk(\yy)$.  This fact is more apparent when one constructs multiplication in $\Div_\kk(\yy)$ as the {\it (graded) dual} of the coalgebra structure
$\Delta:S \rightarrow S \otimes S$ on $S=\kk[\xx]$ in which  $\Delta(x_i)=1 \otimes x_i + x_i \otimes 1$, that is, each $x \in S_1=(\kk^n)^*$ is {\it primitive}.  See
Akin, Buchsbaum and Weyman \cite[\S I.4]{AkinBuchsbaumWeyman}, Eisenbud \cite[A2.4]{Eisenbud}, Geramita \cite[\S 9]{Geramita} for more on this alternate construction of
$\Div_\kk(\yy)$.
\end{remark}

A consequence of the $GL_n(\kk)$-{\it invariance} \eqref{eq:invariance-of-perfect-pairings} is that the isomorphisms \eqref{eq:perp-contragredient-to-quotient},
\eqref{eq:perp-contragredient-to-quotient-divided} become $GL_n(\kk)$-{\it equivariant}. This shows that for each $m$, one has isomorphisms $I^\perp_m \cong (S/I)^*$ as $GL_n(\kk)$-representations.
The same holds for the action of any subgroup $G \subseteq GL_n(\kk)$ on $S, \Div$ by restriction.

\subsection{Representation theory}
We will be interested in polytopes and 
point loci in $\RR^n$ with symmetry,
and  wish to keep track
of the representations of their symmetry groups on the various $\kk$-vector spaces that we construct.
We review one way to do such bookkeeping,
using the language of $\kk[G]$-modules and
representation rings.

\begin{definition} \rm
For a field $\kk$ and finite group $G$, define its {\it representation ring} $\Cl_\kk(G)$ as follows. 
\begin{itemize}
    \item As a $\ZZ$-module,
$\Cl_\kk(G)$ is the quotient of free $\ZZ$-module
with $\ZZ$-basis elements $[V]$ indexed by all
isomorphism classes finite-dimensional $\kk[G]$-modules $V$, in which one mods out by the submodule
$\ZZ$-spanned by all relations
\begin{equation}
\label{eq:direct-sum-relation-in-rep-ring}
[V \oplus V']-([V]+[V']).
\end{equation}
\item As a $\ZZ$-algebra, its multiplication is induced by the rule 
$$
[V] \cdot [V']:=[V \otimes V']. $$
\item The operation $V \mapsto V^*$ of
taking the contragredient $\kk[G]$-module
leads to a $\ZZ$-automorphism and involution on $(-)^*: \Cl_\kk(G)\rightarrow \Cl_\kk(G)$
$$
[V]^*:=[V^*].
$$
\end{itemize}
\end{definition}
Whenever $\#G$ lies in $\kk^\times$, Maschke's Theorem
asserts that $\kk[G]$-modules are completely reducible, which shows that $\Cl_\kk(G)$ is a free $\ZZ$-module on the $\ZZ$-basis $[V_1],[V_2],\ldots,[V_N]$
where $V_1,\ldots,V_N$ are the non-isomorphic {\it simple/irreducible} $\kk[G]$-modules.

For any field $\kk$, the map sending $[V]$ to its {\it character} $\ch(V): G \rightarrow \kk$ defined by
$$
\ch(V)(g):=\mathrm{trace}(g:V \rightarrow V),
$$
becomes an algebra map from $\Cl_\kk(G)$ into 
the {\it ring of class functions} $\{f: G \rightarrow \kk \}$, that is, functions which are constant on $G$-conjugacy classes. The ring of class functions is given pointwise addition, multiplication, and the involution $f \mapsto f^*$ defined by $f^*(g)=f(g^{-1})$. Whenever $\kk$ has characteristic zero, this algebra map is injective, and in particular, two $\kk[G]$-modules $V,V'$ are isomorphic (that is, $[V]=[V']$) if and only if they have the same character
$\ch(V)=\ch(V')$.


More generally, for graded $\kk[G]$-modules
$V = \bigoplus_{m \geq 0} V_m$, with each $V_m$ a finite-dimensional $\kk[G]$-module,
we will track the representation with a power series in $\Cl_\kk(G)[[q]]$:
\begin{equation}
    \label{eq:background-graded-character}
    [V]_q := \sum_{m \, \geq \, 0} [V_m] \cdot q^i.
\end{equation}
In particular, when $I \subseteq S$ is a homogeneous ideal which is stable under the action of a finite subroup $G$ of $GL_n(\kk)$, both the quotient $S/I$ and the harmonic space $I^\perp$ inherit the structure of graded $\kk[G]$-modules. Then \eqref{eq:perp-contragredient-to-quotient-divided}
implies a graded 
$\kk[G]$-module isomorphism
$I^\perp \cong (S/I)^*$,
and hence $[I^\perp]_q=[S/I]^*_q$ in $\Cl_\kk(G)[[q]]$.

On the other hand, we will also consider
potentially {\it inhomogeneous} ideals
$I \subseteq S$ that are stable under the action of a finite subgroup $G \subseteq GL_n(\kk)$,
e.g., $I = \II(\Zpoints)$ where $\Zpoints \subseteq \kk^n$ is a $G$-stable locus.

\begin{prop}
\label{prop:gr-is-Brauer-isomorphic}
Assume $\# G$ lies in $\kk^\times$.
For any ideal $I \subseteq S=\kk[\xx]$ which is stable under
a finite subgroup $G$ of $GL_n(\kk)$, 
if $\dim_\kk S/I$ finite, then one has a 
$\kk[G]$-module
isomorphism $S/I \cong S/\gr\, I$. 
\end{prop}
\begin{proof}
Recall \eqref{eq:associated-graded-ring}
gave an isomorphism
$S/\gr\, I \cong \gr_\FFF(S/I)=\bigoplus_{d=0}^M F^d/F^{d-1}$,
where the sum on the right is finite here due to the finiteness assumption on $\dim_\kk S/I$.
This isomorphism is easily seen to be $G$-equivariant.
Since the action of $G$ preserves degree, the filtration $\FFF$ of $S/I$ is $G$-stable, and its filtration factors match the $\kk[G]$-module structures on the
graded components of $S/\gr\, I$. Complete reducibility of $\kk[G]$-modules then shows $S/I \cong S/\gr\, I$.
\end{proof}

This proof also shows, for any field $\kk$, the $\kk[G]$-modules $S/\gr \, I$ and $S/I$
are Brauer-isomorphic.

\section{The $q$-Ehrhart series and Conjecture~\ref{conj:intro-omnibus}}
\label{sec:Ehrhart}

Throughout this section, we take $\kk=\RR$ and consider finite point loci
$\Zpoints \subset \RR^n$.  The locus $\Zpoints$ has two associated ideals inside 
$
S=\RR[\xx]=\RR[x_1,\ldots,x_n]:
$ 
the (inhomogeneous) {\it vanishing ideal} 
$$
\II(\Zpoints) = \{ f(\xx) \in S \,:\, f(\zz) = 0 \text{ for all $\zz \in \Zpoints$} \}
$$
and its {\it associated graded ideal} 
$$
\gr \, \II(\Zpoints):=(\tau(f): f \in \II(\Zpoints) ).
$$
Within the polynomial ring $\Div=\RR[\yy]=\RR[y_1,\ldots,y_n]$
in the dual variables $\yy$, the
{\it harmonic space} $\left( \gr \, \II(\Zpoints) \right)^\perp$ of $\gr \, \II(\Zpoints)$ will play a crucial role in our work. To reduce notational clutter, we write
 \begin{equation}
     \label{eq:v-harmonic-notation}
     V_\Zpoints := \left( \gr \, \II(\Zpoints) \right)^\perp
 \end{equation}
 for this harmonic space. This notation emphasizes the role of $V_\Zpoints$ as a graded subspace of $\Div$ which is almost never closed under multiplication.
 On the other hand, \eqref{eq:perp-and-quotient-share-Hilbert-series} shows that, as a graded vector space, it has the same Hilbert series (actually a polynomial here) as the quotient ring $R(\Zpoints) := S/\gr \, \II(\Zpoints)$ that was defined in \eqref{eq:intro-point-orbit-ring}
 $$
 \Hilb(V_\Zpoints,q)=\Hilb(R(\Zpoints),q).
 $$
 Note Proposition~\ref{prop:graded-quotient-is-a-gr}(iii)  shows  this Hilbert series is a $q$-analogue of the cardinality $\#\Zpoints$, that is,
 $$
 \left[ \Hilb(V_\Zpoints,q) \right]_{q=1}=\left[\Hilb(R(\Zpoints),q)\right]=\dim_\kk(S/\II(\Zpoints)=\#\Zpoints.
 $$
 
 Furthermore, when $G$ is a finite subgroup of $GL_n(\RR)$
 that preserves $\Zpoints$ setwise, it acts via ring
 automorphisms and (graded) $\RR[G]$-modules
 on all of the objects under consideration:  
 $$
 S,
 \,\,
 \Div, 
 \,\,
 \II(\Zpoints),
 \,\, 
 \gr\,  \II(\Zpoints),
 \,\,
 S/\II(\Zpoints), 
 \,\,
 R(\Zpoints),
 \,\,
 V_\Zpoints.
 $$
 Because finite-dimensional $\RR[G]$-modules $V$
 are all self-contragredient ($V^* \cong V$), one can check that \eqref{eq:perp-contragredient-to-quotient}
 implies graded $\RR[G]$-module isomorphisms
 \begin{equation}
 \label{eq:graded-rep-isos-for-point-loci}
 V_\Zpoints \cong R(\Zpoints)
 \end{equation}
 and then Proposition~\ref{prop:gr-is-Brauer-isomorphic} implies a further ungraded $\kk[G]$-module isomorphisms
 \begin{equation}
 \label{eq:ungraded-rep-isos-for-point-loci}
 V_\Zpoints \cong R(\Zpoints) \cong S/\II(\Zpoints),
 \end{equation}
 which are all three isomorphic to the $\RR[G]$-permutation module on the points $\Zpoints$.

\subsection{Definition of $q$-Ehrhart series and the conjecture} 

As in the introduction, a {\it lattice polytope} $P \subset \RR^n$ is the convex hull of a finite set of
points in the lattice $\ZZ^n$.  
For each integer $m = 0,1,2,\ldots$ one obtains a finite point locus $\ZZ^n \cap mP$. 
For each  $m \geq 1$, one has
the interior point locus $\ZZ^n \cap \interior{mP}$, where $\interior{P} = P \setminus \boundary{P}$ is the (relative) interior where one removes the union $\boundary{P}$ of all boundary faces of $P$.

\begin{definition} \rm
 \label{def:q-Ehrhart} 
 For a lattice polytope $P \subset \RR^n$, define two {\em $q$-Ehrhart series} in $\ZZ[q][[t]]$:
    \begin{align}
    \Eseries_P(t,q)
    &:= \sum_{m=0}^\infty \Ehr_P(m;q) \cdot t^m,\\ 
    \notag
    &\text{ where }
    \Ehr_P(m;q):=\Hilb(V_{\ZZ^n \cap mP}, q) =  \Hilb(R(\ZZ^n \cap mP), q),
    \\
    \intEseries_P(t,q)
    &:= \sum_{m=1}^\infty \bar{i}_P(m;q) \cdot t^m\\ 
    \notag
    &\quad \text{ where }\bar{i}_P(m;q)=\Hilb(V_{\ZZ^n \cap \interior{mP}}, q) = \Hilb(R(\ZZ^n \cap \interior{mP}), q).
    \end{align}
\end{definition}
\noindent
Note the series $\Eseries_P(t,q),\intEseries_P(t,q)$ reduce to the classical Ehrhart series 
$\Eseries_P(t), \intEseries_P(t)$ at $q \to 1$.

\begin{example}
\label{ex:repeat-from-intro}
Recall the Example in the Introduction looked at the general $1$-dimensional lattice polytope $P=[a,a+v] \subset \RR^1$
with $a,v \in \ZZ$ and volume $v \geq 1$,
finding that
\begin{align}
\label{eq:one-dimensional-Eseries-repeat}
\Eseries_P(t,q)&
     = \frac{1+tq [v-1]_q}{(1-t)(1-tq^v)},\\
\intEseries_P(t,q)
    &=\frac{t[v-1]_q +t^2q^{v-1}}{(1-t)(1-tq^v)},
\end{align}
where $ [m]_q:=1+q+q^2+\cdots+q^{m-1}$.
\end{example}

Note that Example~\ref{ex:repeat-from-intro} shows that
both $\Eseries_P(t,q),\intEseries_P(t,q)$ for lattice polytopes $P \subset \RR^1$ depend only upon a single parameter, which one could take either to be the volume $v$, or the number of lattice points $\#\ZZ^1 \cap P=v+1$, or the $h^*$-vector entry $h^*_1=v-1$.   This illustrates a certain affine-lattice invariance
that one might expect, similar to classical Ehrhart theory.  Recall that the group of {\it affine transformations} $\Aff(\RR^n)$ of $\RR^n$ is a semidirect product $\Aff(\RR^n)=GL_n(\RR) \ltimes \RR^n$, where $GL_n(\RR)$ is the subgroup fixing the origin, and $\RR^n$ is the translation subgroup.  This restricts to a semidirect product decomposition
$\Aff(\ZZ^n)=GL_n(\ZZ) \ltimes \ZZ^n$.

\begin{prop}
\label{prop:affine-invariance}
For any $g \in \Aff(\ZZ^n)$, one has
$\Eseries_{gP}(t,q)=\Eseries_P(t,q)$ and 
$\intEseries_{gP}(t,q)=\intEseries_P(t,q)$.
\end{prop}
\begin{proof}
    We give the argument for $\Eseries_{gP}(t,q)$; the argument for $\intEseries_{gP}(t,q)$ is
    similar.  
    
    Note that for each $m=0,1,2,\ldots$, the point locus $\ZZ^n \cap m \cdot gP$ is an affine transformation of the locus
    $\ZZ^n \cap mP$, namely by an affine transformation $h$ whose translation vector $m \cdot v$ is scaled by $m$ from the translation vector $v$ of $g$. 
    It therefore suffices to check for finite point loci $\Zpoints \subseteq \RR^n$ and for any $h \in \Aff(\RR^n)$, that one has    
    \begin{equation}
    \label{eq:Hilb-of-R(Z)-is-affine-invariant}
        \Hilb(R(h\Zpoints),q)=\Hilb(R(\Zpoints),q). 
    \end{equation}
    Note that $h \in \Aff(\RR^n)$ acting on $\xx$ variables by $h\xx=h_0\xx+v$ for some $h_0 \in GL_n(\RR), v \in \RR^n$ has
    $$
    \tau(h(f))(\xx)=\tau(f(h_0\xx+v))=\tau(f)(h_0\xx).
    $$
    Therefore $\gr\, \II(h \Zpoints)= h_0 (\gr \,\II(\Zpoints))$.  Since $h_0\in GL_n(\RR)$ acts via a graded $\RR$-algebra automorphism on $S=\RR[\xx]$, this means
    it induces a graded $\RR$-algebra isomorphism $R(\Zpoints) \cong R(h \Zpoints)$ implying \eqref{eq:Hilb-of-R(Z)-is-affine-invariant}:
    $$
    R(\Zpoints)=S/\gr \, \II(\Zpoints) \overset{h_0}{\longrightarrow}
    S/h_0(\gr\, \II(\Zpoints))= S/\gr\, \II(h\Zpoints) =R(h \Zpoints).\qedhere
    $$
\end{proof}

Before discussing more examples, recall the main conjecture from the
Introduction.

\vskip.1in
\noindent
{\bf Conjecture~\ref{conj:intro-omnibus}.}
{\it
    Let $P$ be a $d$-dimensional lattice polytope in $\RR^n$. Then both of the series \eqref{eq:intro-q-ehrhart-series},\eqref{eq:intro-q-ehrhart-interior-series}
   lie in $\QQ(t,q)$, and are expressible as rational functions  
    $$
    \Eseries_P(t,q)=\frac{N_P(t,q)}{D_P(t,q)}
    \quad \text{ and } \quad 
        \intEseries_P(t,q)=\frac{\overline{N}_P(t,q)}{D_P(t,q)},
    $$
    over the same denominator of the form
    $D_P(t,q)=\prod_{i=1}^\nu (1-q^{a_i} t^{b_i})$, necessarily with $\nu \geq d+1$. Furthermore, there exists such an expression with all of these properties:
    \begin{itemize} 
    \item [(i)] The numerators $N_P(t,q), 
        \overline{N}_P(t,q)$ lie in $\ZZ[t,q]$.
    \item[(ii)] If $P$ is a lattice simplex, and $\nu=d+1$, then both numerators $N_P(t,q), \overline{N}_P(t,q)$ have nonnegative
coefficients as polynomials in $t,q$.
\item[(iii)] The two series $\Eseries_P(t,q), \intEseries_P(t,q)$
determine each other via
$$
q^d  \cdot \intEseries_P(t,q) = (-1)^{d+1} \Eseries_P(t^{-1},q^{-1}).
$$
\end{itemize}
}

\subsection{Examples: lattice polygons}

\begin{figure}
$$
\begin{array}{|c|c|c|c|c|}\hline
\text{normalized}&\text{vertices} &h_P^*(t)=& &\Aff(\ZZ^2)\text{-equivalent}\\
 \text{area of }P& \text{ of }P &1+h^*_1 t+h^*_2 t^2 &\Eseries_P(t,q)&\text{to antiblocking?}\\
 \hline\hline
1&(0,0),(1,0),(0,1)
\begin{tikzpicture}[scale = 0.3]
        \draw[fill = black!05] (0,0) -- (0,1) -- (1,0) -- (0,0);
        \draw[step=1.0,black,thin] (0,-0) grid (1,1);    
        \node at (0,0) {$\bullet$};
        \node at (0,1) {$\bullet$};
        \node at (1,0) {$\bullet$};
    \end{tikzpicture}
& 1& \frac{1}{ (1-t)(1-qt)^2 }&\text{Yes}\\ \hline
2&(0,0),(1,0),(1,2)
\begin{tikzpicture}[scale = 0.3]
        \draw[fill = black!05] (0,0) -- (1,0) -- (1,2) -- (0,0);
        \draw[step=1.0,black,thin] (0,0) grid (1,2);    
        \node at (0,0) {$\bullet$};
        \node at (1,0) {$\bullet$};
        \node at (1,1) {$\bullet$};
        \node at (1,2) {$\bullet$};
    \end{tikzpicture}
& 1+t& \frac{1+qt}{ (1-t)(1-qt)(1-q^2t) }& \text{Yes}\\ 
2&(0,0),(1,0),(0,1),(1,1)
\begin{tikzpicture}[scale = 0.3]
        \draw[fill = black!05] (0,0) -- (1,0) -- (1,1) -- (0,1) -- (0,0);
        \draw[step=1.0,black,thin] (0,0) grid (1,1);    
        \node at (0,0) {$\bullet$};
        \node at (1,0) {$\bullet$};
        \node at (0,1) {$\bullet$};
        \node at (1,1) {$\bullet$};
    \end{tikzpicture}
&
1 + t&
\frac{1 + qt}{(1- t)(1-qt)(1-q^2 t)} &\text{Yes}\\ \hline
3&(0,0),(1,0),(1,3)
\begin{tikzpicture}[scale = 0.3]
        \draw[fill = black!05] (0,0) -- (1,0) -- (1,3) -- (0,0);
        \draw[step=1.0,black,thin] (0,0) grid (1,3);    
        \node at (0,0) {$\bullet$};
        \node at (1,0) {$\bullet$};
        \node at (1,1) {$\bullet$};
        \node at (1,2) {$\bullet$};
        \node at (1,3) {$\bullet$};
    \end{tikzpicture}
&1 + 2t&
\frac{1 +qt +q^2t}{(1 - t)(1 - qt)(1 - q^3t)} &\text{Yes} \\ 
3&(0,0),(1,0),(2,3)
\begin{tikzpicture}[scale = 0.3]
        \draw[fill = black!05] (0,0) -- (1,0) -- (2,3) -- (0,0);
        \draw[step=1.0,black,thin] (0,0) grid (2,3);    
        \node at (0,0) {$\bullet$};
        \node at (1,0) {$\bullet$};
        \node at (1,1) {$\bullet$};
        \node at (2,3) {$\bullet$};
    \end{tikzpicture}
&1 + t + t^2&
\frac{(1 + qt)(1 + qt + q^2t^2)}{(1 - t)(1 - q^2t)(1 - q^3t^2)} &\text{No}\\
3&(0,0),(1,0),(0,1),(-2,1)
\begin{tikzpicture}[scale = 0.3]
        \draw[fill = black!05] (0,0) -- (1,0) -- (0,1) -- (-2,1) -- (0,0);
        \draw[step=1.0,black,thin] (-2,0) grid (1,1);    
        \node at (0,0) {$\bullet$};
        \node at (1,0) {$\bullet$};
        \node at (0,1) {$\bullet$};
        \node at (-2,1) {$\bullet$};
        \node at (-1,1) {$\bullet$};
    \end{tikzpicture}
&1 + 2t &
\frac{1 + qt - q^2t^2 - q^3t^2}{(1 - t)(1 - qt)(1 - q^2t)^2}&\text{Yes}\\\hline
\end{array}
$$
\caption{$\Eseries_P(t,q)$ for lattice polygons $P$ of area at most $3$, up to $\Aff(\ZZ^2)$.}
\label{fig:q-Ehrhart-examples}
\end{figure}

\begin{figure}
$$
\begin{array}{|c|c|c|c|c|}\hline
\text{normalized}&\text{vertices} &h_P^*(t)=& &\Aff(\ZZ^2)\text{-equivalent}\\
 \text{area of }P& \text{ of }P &1+h^*_1 t+h^*_2 t^2 &\Eseries_P(t,q)&\text{to antiblocking?}\\
 \hline\hline

4&(0,0),(1,0),(1,4)
\begin{tikzpicture}[scale = 0.3]
        \draw[fill = black!05] (0,0) -- (1,0) -- (1,4) -- (0,0);
        \draw[step=1.0,black,thin] (0,0) grid (1,4);    
        \node at (0,0) {$\bullet$};
        \node at (1,0) {$\bullet$};
        \node at (1,4) {$\bullet$};
        \node at (1,1) {$\bullet$};
        \node at (1,2) {$\bullet$};
        \node at (1,3) {$\bullet$};
    \end{tikzpicture}
& 1 + 3t&
\frac{1 + t(q + q^2 + q^3)}{(1 - t)(1 - qt)(1 - q^4t)}&\text{Yes}\\
4&(0,0),(1,0),(3,4) 
\begin{tikzpicture}[scale = 0.3]
        \draw[fill = black!05] (0,0) -- (1,0) -- (3,4) -- (0,0);
        \draw[step=1.0,black,thin] (0,0) grid (3,4);    
        \node at (0,0) {$\bullet$};
        \node at (1,0) {$\bullet$};
        \node at (1,1) {$\bullet$};
        \node at (2,2) {$\bullet$};
        \node at (3,4) {$\bullet$};
    \end{tikzpicture}
& (1 + t)^2&
\frac{(1 + qt)^2}{(1 - t)(1 - q^2t)^2}&\text{No} \\
4&(0,0),(2,0),(0,2)
\begin{tikzpicture}[scale = 0.3]
        \draw[fill = black!05] (0,0) -- (2,0) -- (0,2) -- (0,0);
        \draw[step=1.0,black,thin] (0,0) grid (2,2);    
        \node at (0,0) {$\bullet$};
        \node at (1,0) {$\bullet$};
        \node at (2,0) {$\bullet$};
        \node at (0,1) {$\bullet$};
        \node at (0,2) {$\bullet$};
        \node at (1,1) {$\bullet$};
    \end{tikzpicture}& 1 + 3t&
\frac{1 + 2q t + q^2t}{(1 - t)(1 - q^2t)^2}
&\text{Yes} \\
4&(0,0),(1,0),(0,1),(-3,1)
\begin{tikzpicture}[scale = 0.3]
        \draw[fill = black!05] (0,0) -- (1,0) -- (0,1) -- (-3,1) -- (0,0);
        \draw[step=1.0,black,thin] (-3,0) grid (1,1);    
        \node at (0,0) {$\bullet$};
        \node at (1,0) {$\bullet$};
        \node at (0,1) {$\bullet$};
        \node at (-1,1) {$\bullet$};
        \node at (-2,1) {$\bullet$};
        \node at (-3,1) {$\bullet$};
    \end{tikzpicture}&
1 + 3t&
\frac{1 + qt + q^2t - q^2t^2 - q^3t^2 - q^4t^2}{(1 - t)(1 - qt)(1 - q^2t)(1 - q^3 t)} &\text{Yes}\\
4&(0,0),(1,0),(0,2),(1,2)
\begin{tikzpicture}[scale = 0.3]
        \draw[fill = black!05] (0,0) -- (1,0) -- (1,2) -- (0,2) -- (0,0);
        \draw[step=1.0,black,thin] (0,0) grid (1,2);    
        \node at (0,0) {$\bullet$};
        \node at (1,0) {$\bullet$};
        \node at (1,2) {$\bullet$};
        \node at (0,2) {$\bullet$};
        \node at (0,1) {$\bullet$};
        \node at (1,1) {$\bullet$};
\end{tikzpicture}
    & 1 + 3t &
\frac{1 + qt + q^2t - q^2t^2 - q^3t^2 - q^4t^2}
{(1 - t)(1 - qt)(1 - q^2t)(1 -q^3t)} &\text{Yes} \\
4&(0,0),(2,0),(0,1),(1,-1)
\begin{tikzpicture}[scale = 0.3]
        \draw[fill = black!05] (0,0) -- (1,-1) -- (2,0) -- (0,1) -- (0,0);
        \draw[step=1.0,black,thin] (0,-1) grid (2,1);    
        \node at (0,0) {$\bullet$};
        \node at (2,0) {$\bullet$};
        \node at (0,1) {$\bullet$};
        \node at (1,0) {$\bullet$};
        \node at (1,-1) {$\bullet$};
\end{tikzpicture}& (1 + t)^2&
\frac{(1 + qt)^2}{(1 - t)(1 - q^2t)^2}&\text{No} \\
4&(0,0),(1,0),(1,2),(2,2)
\begin{tikzpicture}[scale = 0.3]
        \draw[fill = black!05] (0,0) -- (1,0) -- (2,2) -- (1,2) -- (0,0);
        \draw[step=1.0,black,thin] (0,0) grid (2,2);    
        \node at (0,0) {$\bullet$};
        \node at (1,0) {$\bullet$};
        \node at (1,1) {$\bullet$};
        \node at (1,2) {$\bullet$};
        \node at (2,2) {$\bullet$};
\end{tikzpicture}
 & (1 + t)^2&
\frac{(1 + qt)^2}{(1 - t)(1 - q^2t)^2)}&\text{No} \\\hline
\end{array}
$$
\caption{$\Eseries_P(t,q)$ for lattice polygons $P$ of area $4$, up to $\Aff(\ZZ^2)$.}
\label{fig:area-4-polygons}
\end{figure}

\begin{figure}
$$
\begin{array}{|c|c|c|c|}\hline
\text{vertices} &h_P^*(t)=& &\Aff(\ZZ^2)\text{-equivalent}\\
 \text{ of }P &1+h^*_1 t+h^*_2 t^2 &\Eseries_P(t,q)&\text{to antiblocking?}\\
 \hline\hline
(0,0),(1,0),(1,5)
\begin{tikzpicture}[scale = 0.3]
        \draw[fill = black!05] (0,0) -- (1,0) -- (1,5) --  (0,0);
        \draw[step=1.0,black,thin] (0,0) grid (1,5);    
        \node at (0,0) {$\bullet$};
        \node at (1,0) {$\bullet$};
        \node at (1,5) {$\bullet$};
        \node at (1,1) {$\bullet$};
        \node at (1,2) {$\bullet$};
        \node at (1,3) {$\bullet$};
        \node at (1,4) {$\bullet$};
\end{tikzpicture}
& 1 + 4t &
\frac{1 + t(q + q^2 + q^3 + q^4)}{(1 - t)(1 - qt) (1 - q^5t)}
&\text{Yes} \\
(0,0),(1,0),(2,5) 
\begin{tikzpicture}[scale = 0.3]
        \draw[fill = black!05] (0,0) -- (1,0) -- (2,5) --  (0,0);
        \draw[step=1.0,black,thin] (0,0) grid (2,5);    
        \node at (0,0) {$\bullet$};
        \node at (1,0) {$\bullet$};
        \node at (1,1) {$\bullet$};
        \node at (1,2) {$\bullet$};
        \node at (2,5) {$\bullet$};
\end{tikzpicture}
& 1 + 2t + 2t^2 &
\frac{(1 + qt)(1 + t(q + q^2) + t^2(q^3 + q^4))}{(1 - t)(1 - q^2t)(1 - q^5t^2)}&\text{No} \\
(0,0),(1,0),(0,1),(-4,1)
\begin{tikzpicture}[scale = 0.3]
        \draw[fill = black!05] (0,0) -- (1,0) -- (0,1) --  (-4,1) -- (0,0);
        \draw[step=1.0,black,thin] (-4,0) grid (1,1);    
        \node at (0,0) {$\bullet$};
        \node at (1,0) {$\bullet$};
        \node at (0,1) {$\bullet$};
        \node at (-4,1) {$\bullet$};
        \node at (-3,1) {$\bullet$};
        \node at (-2,1) {$\bullet$};
        \node at (-1,1) {$\bullet$};
\end{tikzpicture}
& 1 + 4t&
\frac{1 + qt + q^2t + q^3t - q^2t^2 - q^3t^2 - q^4t^2 - q^5t^2}{
(1 - t)(1 - qt)(1 - q^2t)(1 - q^4t)} &\text{Yes}\\
(0,0),(2,0),(0,1),(-3,1)
\begin{tikzpicture}[scale = 0.3]
        \draw[fill = black!05] (0,0) -- (2,0) -- (0,1) --  (-3,1) -- (0,0);
        \draw[step=1.0,black,thin] (-3,0) grid (2,1);    
        \node at (0,0) {$\bullet$};
        \node at (1,0) {$\bullet$};
        \node at (2,0) {$\bullet$};
        \node at (0,1) {$\bullet$};
        \node at (-1,1) {$\bullet$};
        \node at (-2,1) {$\bullet$};
        \node at (-3,1) {$\bullet$};
\end{tikzpicture}
& 1 + 4t &
\frac{1 + qt + 2q^2t - q^2t^2 - q^3t^2 - q^4t^2 - q^5t^2}{(1 - t)(1 - qt)(1 -q^3t)^2}&\text{Yes} \\
(0,0),(1,0),(2,3),(2,1)
\begin{tikzpicture}[scale = 0.3]
        \draw[fill = black!05] (0,0) -- (1,0) -- (2,1) --  (2,3) -- (0,0);
        \draw[step=1.0,black,thin] (0,0) grid (2,3);    
        \node at (0,0) {$\bullet$};
        \node at (1,0) {$\bullet$};
        \node at (1,1) {$\bullet$};
        \node at (2,2) {$\bullet$};
        \node at (2,3) {$\bullet$};
        \node at (2,1) {$\bullet$};
\end{tikzpicture}
& 1 + 3t + t^2&
\frac{1 +2qt + 2q^2t + 2q^3t^2 + 2q^4t^2 + q^5t^3}{(1 - t)(1 - q^2t)(1 - q^5 t^2)}&\text{No}\\
(0,0),(1,0),(1,2),(2,2),(0,-1)
\begin{tikzpicture}[scale = 0.3]
        \draw[fill = black!05] (0,-1) -- (1,0) -- (2,2) --  (1,2) -- (0,0) -- (0,-1);
        \draw[step=1.0,black,thin] (0,-1) grid (2,2);    
        \node at (0,0) {$\bullet$};
        \node at (1,0) {$\bullet$};
        \node at (1,1) {$\bullet$};
        \node at (1,2) {$\bullet$};
        \node at (2,2) {$\bullet$};
        \node at (0,-1) {$\bullet$};
\end{tikzpicture}
& 1 + 3t + t^2 &
\frac{1 + 2qt + 2q^2t + 2q^3t^2 + 2q^4t^2 + q^5t^3}{
(1 - t)(1 - q^2t)(1 - q^5t^2)}&\text{No}\\\hline
\end{array}
$$
\caption{$\Eseries_P(t,q)$ for lattice polygons of area 5, up to $\Aff(\ZZ^2)$.}
\label{fig:area-5-polygons}
\end{figure}

Having computed $\Eseries_P(t,q)$ for all one-dimensional lattice polytopes (line segments) in the Introduction and Example~\ref{ex:repeat-from-intro}, one might wish to see data for lattice {\it polygons}.  Proposition~\ref{prop:affine-invariance} allows one to consider them only up to the action of
$\Aff(\ZZ^2)$.  If one bounds the normalized volume of a $d$-dimensional lattice polytope $P$ in $\RR^d$, a well-known result of 
Lagarias and Ziegler \cite{LagariasZiegler} shows
that there are only finitely many such $P$ up to the action of $\Aff(\ZZ^d)$.  Work of Balletti \cite{Balletti} gives an algorithm to list them, including an online database at \url{https://github.com/gabrieleballetti/small-lattice-polytopes} listing equivalence classes of lattice polytopes up to dimension $6$ of relatively small normalized volumes.  In particular, it includes lattice triangles up to normalized area $1000$.

Using this data, Figures~\ref{fig:q-Ehrhart-examples}, \ref{fig:area-4-polygons}, \ref{fig:area-5-polygons} present the $q$-Ehrhart series $\Eseries_P(t,q)$
of $\Aff(\ZZ^2)$-equivalence classes of lattice polygons $P$ of normalized volume at most $5$.  These were first guessed using  {\tt Macaulay2} to compute
$\Ehr_P(m;q):=\Hilb(R(\ZZ^2 \cap mP),q)$ for $m=0,1,2,\ldots$ up to some reasonably large values of $m$.
The guesses were then proven correct via some extra computation
in {\tt Macaulay2} that uses our results
on the {\it harmonic algebras} $\HHH_P$ in Section~\ref{sec:Harmonic}; see Remark~\ref{rmk:algorithm-for-q-Ehrhart-series} below for an
explanation.
 
Note the last column of the tables, indicating whether the
polygon $P$ is $\Aff(\ZZ^2)$-equivalent to one in the tamer subclass of {\it antiblocking} lattice polytopes,
discussed in Section~\ref{sec:antiblocking}. 

With the kind assistance of V.~Kurylenko, we also have computed {\it guesses} for $\Eseries_P(t,q)$ for all lattice polygons $P$ of area $6,7,8$, as well as a selection of lattice tetrahedra.  These can be 
found tabulated in the supplementary data file {\tt ExtraData.pdf} in the {\tt arXiv} version of this paper.

All the data in Figures~\ref{fig:q-Ehrhart-examples}, \ref{fig:area-4-polygons}, \ref{fig:area-5-polygons} (and {\tt ExtraData.pdf}) is consistent with Conjecture~\ref{conj:intro-omnibus}.  However, we close this subsection with remarks on 
some cautionary features.

\begin{remark}
The specialization
$\left[ \Eseries_P(t,q) \right]_{q=1}=\Eseries_P(t)$ sometimes has interesting numerator and denominator cancellations when $q \rightarrow 1$.
\end{remark}

\begin{remark}
Note that, for fixed $d$, the classical Ehrhart series $\Eseries_P(t) \subset \QQ(t)$ for a lattice $d$-polytope can be expressed as an affine-linear function of the $d$ real parameters $(h^*_1,\ldots,h^*_d)$:
$$
\Eseries_P(t)=\frac{1}{(1-t)^{d+1}}+
h^*_1 \cdot \frac{t}{(1-t)^{d+1}} +\cdots+h^*_d \cdot \frac{t^d}{(1-t)^{d+1}}.
$$
One has no such affine-linear formula for $\Eseries_P(t,q)$, already at $d=1$, since the
formulas \eqref{eq:one-dimensional-Eseries-repeat}
$$
\Eseries_P(t,q)
     = \frac{1+tq [v-1]_q}{(1-t)(1-tq^v)}
$$
give affine-linearly independent functions of $h^*_1=v-1$
where $v=\vol_1(P)$.  On the other hand this formula does express $\Eseries_P(t,q)$ for $d=1$ as a function (which is {\it not} affine-linear) of the one real parameter $h_1^*$.  When $d=2$, there can be no such function of the two parameters $(h_1^*,h_2^*)$, since one can have two lattice polygons $P,P'$ with the same classical series $\Eseries_P(t)=\Eseries_{P'}(t)$
but different $q$-Ehrhart series
$\Eseries_P(t,q) \neq \Eseries_{P'}(t,q)$; this happens
for two area $3$ lattice polygons in Figure~\ref{fig:q-Ehrhart-examples}. 
\end{remark}

\begin{remark}
\label{rmk: annoying-triangle}
Within {\tt ExtraData.pdf}, one finds a
guess for $\Eseries_P(t,q)$ for
a particular lattice triangle 
$
P=\conv\{ (0, 0), (1, 0), (3, 7) \}
$
that we found difficult to compute.  Eventually, this guess was kindly computed for us by
V. Kurylenko, using data up
through $t$-degree $26$:
$$
\Eseries_P(t,q)\overset{?}{=}\frac{N_P(t,q)}{(1-t)(1-q^8 t^3)(1-q^{21} t^8)}
$$
with a numerator polynomial $N_P(t,q)$ in $\ZZ[t,q]$ having nonnegative coefficients, and $t$-degree $11$.
We found the $t$-power $t^8$ in its denominator surprisingly large compared to other examples\footnote{This example
is one reason that we revised our
main Conjectures~\ref{conj:intro-omnibus},
\ref{conj:harmonic-algebra-omnibus} from {\tt arXiv} version 1 of this paper.}.
\end{remark}

\subsection{Examples: a few Reeve tetrahedra}
An important family of examples in
Ehrhart theory are {\it Reeve's tetrahedra} \cite[Example~3.22]{BeckRobins},
defined by 
$$
T_v=\conv\{(0,0,0),(1,0,0),(0,1,0),(1,1,v)\}\quad\text{ for }v=1,2,\ldots.
$$
The parameter $v$ is their 
normalized volume,
and they have these Ehrhart
polynomial and series:
\begin{align*}
\label{Reeve-tetrahedra-classical-data}
\Ehr_{T_v}(m)&=1+\left(2-\frac{v}{6}\right)m+m^2+\frac{v}{6}m^3\\
\Eseries_{T_v}(t)&=\frac{1+(v-1)t^2}{(1-t)^4}.
\end{align*}
Thus $T_v$ has $(h^*_0,h_1^*,h_2^*)=(1,0,v-1)$, and for $v\geq 2$, gives examples where the $h^*$-vector has internal zeroes.  This can only happen for lattice polytope $P$ lacking the following property.

\begin{definition} \rm
\label{def:IDP-property}
Say that a lattice polytope $P$ has the
 {\it integer decomposition property (IDP)}
 if 
\begin{equation}
\label{eq:IDP}
\underbrace{(\ZZ^n \cap P)+ \cdots + (\ZZ^n \cap P)}_{m \text{ times}} = \ZZ^n \cap mP
\quad \text{ for all }m \geq 1.
\end{equation}
\end{definition}
See Braun \cite{Braun} and Cox, Haase, Hibi and Higashitani \cite{cox2012integer} for more on the IDP. 
All lattice polygons have the IDP.  The smallest non-IDP lattice polytope is the Reeve tetrahedron $T_2$
shown here:
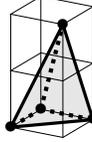
\begin{center}
    \tdplotsetmaincoords{70}{110}
    \begin{tikzpicture}[scale = 0.8]
        \tdplotsetrotatedcoords{-10}{0}{0}

        \fill [gray!20,tdplot_rotated_coords] (1,0,0) -- (0,1,0) -- (1,1,2) -- (1,0,0);

        \draw[ultra thin, tdplot_rotated_coords] (0,0,0) -- (0,1,0) -- (1,1,0) -- (1,0,0) -- (0,0,0);
        \draw[ultra thin, tdplot_rotated_coords] (0,0,1) -- (0,1,1) -- (1,1,1) -- (1,0,1) -- (0,0,1);
        \draw[ultra thin, tdplot_rotated_coords] (0,0,2) -- (0,1,2) -- (1,1,2) -- (1,0,2) -- (0,0,2);
        \draw[ultra thin, tdplot_rotated_coords] (0,0,0) -- (0,0,2);
        \draw[ultra thin, tdplot_rotated_coords] (1,0,0) -- (1,0,2);
        \draw[ultra thin, tdplot_rotated_coords] (0,1,0) -- (0,1,2);
        \draw[ultra thin, tdplot_rotated_coords] (1,1,0) -- (1,1,2);

        \draw[ultra thick, tdplot_rotated_coords,dotted] (0,0,0) -- (1,0,0);
        \draw[ultra thick, tdplot_rotated_coords,dotted] (0,0,0) -- (0,1,0);
        \draw[ultra thick, tdplot_rotated_coords,dotted] (0,0,0) -- (1,1,2);
        \draw[very thick, tdplot_rotated_coords] (1,0,0) -- (0,1,0);
        \draw[very thick, tdplot_rotated_coords] (1,0,0) -- (1,1,2);
        \draw[very thick, tdplot_rotated_coords] (0,1,0) -- (1,1,2);
    
        \node[tdplot_rotated_coords] at (0,0,0) {$\bullet$};

        \node[tdplot_rotated_coords] at (1,0,0) {$\bullet$};
        \node[tdplot_rotated_coords] at (0,1,0) {$\bullet$};
        \node[tdplot_rotated_coords] at (1,1,2) {$\bullet$};

    \end{tikzpicture}
\end{center} 
It is not IDP since
$(1,1,1) \in \ZZ^3 \cap 2P \setminus ((\ZZ^3 \cap P) + (\ZZ^3 \cap P))$. 

We have either computations or guesses for
the $q$-Ehrhart series $\Eseries_{T_v}(t,q)$
for $v=1,2,3$:
\begin{itemize}
    \item $T_1$ is a unimodular tetrahedron,
    $\Aff(\ZZ^3)$-equivalent to $\pyr(\Delta^2)$ from Section~\ref{sec:standard-simplices} below, so 
    $$
    \Eseries_{T_1}(t,q)=\frac{1}{(1-t)(1-tq)^3}.
    $$
 \item For $T_2$, using the same methods as for Figures~\ref{fig:q-Ehrhart-examples}, \ref{fig:area-4-polygons}, \ref{fig:area-5-polygons}, one can compute 
\begin{equation}
\label{eq:first-non-IDP-Reeve-q-Ehrhart}
     \Eseries_{T_2}(t,q)=\frac{(1 + qt)(1 + q^2t^2)(1 + qt + q^2t^2)}{(1 - t)(1 - qt)(1 - q^3t^2)(1 - q^4 t^3)}.
\end{equation}
 \item V.~Kurylenko computed this {\it guess} for the $q$-Ehrhart series for $T_3$, correct up to $t$-degree $16$:
\begin{equation}
\label{eq:Vadym-Reeve-tetrahedron-counterexample}
\Eseries_{T_3}(t,q)
\overset{?}{=}\frac{(1 + q t)(1-q^5 t^4)(1 + qt + 2q^2t^2 + q^3 t^2 + 2q^4 t^3 + q^5 t^4 + q^6 t^4)}
{(1-t)(1-qt)(1-q^3 t^2)(1-q^5 t^3)(1-q^6 t^4)}.
 \end{equation}
\end{itemize}
Note that, in this guess for $\Eseries_{T_3}(t,q)$, there are $5$ denominator factors, 
but $\dim(T_3)+1=4<5$.  This is the first
instance where this occurs
for a lattice simplex\footnote{This example
is another reason that we revised our
Conjectures~\ref{conj:intro-omnibus},
\ref{conj:harmonic-algebra-omnibus} from
{\tt arXiv} version 1 of this paper.},
raising the following question.

\begin{question} \rm
\label{ex: Kurylenko-Reeve-tetrahedra}
For which lattice $d$-simplices $P$
can one express 
$$
\Eseries_P(q,t)=\frac{N_P(t,q)}{\prod_{i=1}^\nu(1-t^{b_i}q^{a_i})}
$$
with $\nu=d+1$ denominator factors,
and with $N_P(t,q)$ in $\NN[t,q]$?
\begin{itemize}
\item[(a)]
Does this occur for all lattice {\it triangles}?  
\item[(b)]
Does it occur more generally for {\it all lattice simplices with the IDP property}?
\end{itemize}
\end{question}

\subsection{Examples:  antiblocking polytopes}
\label{sec:antiblocking}

For certain polytopes $P$ introduced by Fulkerson \cite{Fulkerson1971, Fulkerson1972} in the context of
combinatorial optimization, both $q$-Ehrhart series $\Eseries_P(t,q),\intEseries_P(t,q)$ have a simpler description as classical weighted lattice point enumerators, avoiding harmonic spaces.  This will allow us
to verify Conjecture~\ref{conj:intro-omnibus} for such polytopes.

In this section, 
abbreviate the nonnegative reals by $\RR_{\geq 0}:=\{z \in \RR: z \geq 0\}$, the nonnegative
orthant by $\RR_{\geq 0}^n$,
and the nonnegative integers $\NN:=\{0,1,2,\ldots\}$.
Define the componentwise partial order on $\RR_{\geq 0}^n$ via $\zz \leq \zz'$ if $z_i \leq z_i'$ for all $i$.

\begin{definition}\rm
\label{def:antiblocking}
Say that a convex polytope $P$ is {\it antiblocking} (or {\it of antiblocking type}) if $P \subset \RR_{\geq 0}^n$, and $P$ forms a (lower) order ideal in
the componentwise order on $\RR_{\geq 0}^n$, that is, whenever $\origin \leq \zz \leq \zz'$ with $\zz' \in P$ then also $\zz \in P$. 
\end{definition}

Here is an example of an antiblocking polygon $P$ inside the orthant $\RR_{\geq 0}^2$.
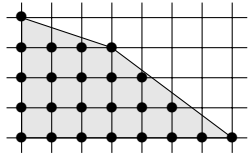
\begin{center}
    \begin{tikzpicture}[scale = 0.4]
        \draw[fill = black!10] (0,0) -- (0,4) -- (3,3) --  (7,0) -- (0,0);
        \draw[step=1.0,black,thin] (-0.5,-0.5) grid (7.5,4.5);    
        \node at (0,0) {$\bullet$};
        \node at (0,1) {$\bullet$};
        \node at (0,2) {$\bullet$};
        \node at (0,3) {$\bullet$};
        \node at (0,4) {$\bullet$};
        \node at (1,0) {$\bullet$};
        \node at (1,1) {$\bullet$};
        \node at (1,2) {$\bullet$};
        \node at (1,3) {$\bullet$};
        \node at (2,0) {$\bullet$};
        \node at (2,1) {$\bullet$};
        \node at (2,2) {$\bullet$};
        \node at (2,3) {$\bullet$};
        \node at (3,0) {$\bullet$};
        \node at (3,1) {$\bullet$};
        \node at (3,2) {$\bullet$};
        \node at (3,3) {$\bullet$};
        \node at (4,0) {$\bullet$};
        \node at (4,1) {$\bullet$};
        \node at (4,2) {$\bullet$};
        \node at (5,0) {$\bullet$};
        \node at (5,1) {$\bullet$};
        \node at (6,0) {$\bullet$};
        \node at (7,0) {$\bullet$};
    \end{tikzpicture}
\end{center}

Intersecting antiblocking polytopes with $\ZZ^n$ gives
rise to point loci $\Zpoints$ with a restrictive property that vastly simplifies the ring $R(\Zpoints)$ and harmonic space
$V_\Zpoints$.

\begin{definition} \rm
\label{def:shifted-point-locus}
Call a subset $\Zpoints \subset \RR^n$ {\it shifted} if $\Zpoints \subset \NN^n$ and $\Zpoints$ forms a (lower) order ideal in the componentwise order on $\NN^n$, that is, whenever  $\origin \leq \zz \leq \zz'$ and $\zz' \in \Zpoints$ then also $\zz \in \Zpoints$.
Equivalently, $\Zpoints \subset \NN^n$ is shifted if and only if the $\kk$-vector space
$I:=\spn_\kk\{ \xx^\aa: \aa \in \NN^n \setminus \Zpoints\}$
forms a {\it monomial ideal} inside $S:=\kk[\xx]$.
\end{definition}

\begin{lemma}
    \label{lem:lower-harmonic}
    For any finite shifted point locus $\Zpoints \subset \NN^n$, 
    \begin{itemize}
    \item[(i)] the ideal $\gr\,\II(\Zpoints) \subset S=\RR[\xx]$ is monomial, with $\RR$-basis $\{ \xx^\aa : \aa \in \NN^n \setminus \Zpoints\}$,
    \item[(ii)] the quotient ring $R(\Zpoints)=S/\gr\,\II(\Zpoints)$ of $S=\RR[\xx]$
    has $\RR$-basis $\{\bar{\xx}^\zz: \zz \in \Zpoints\}$,
    \item[(iii)] 
    the harmonic space $V_\Zpoints \subseteq S$ has $\RR$-basis 
    $\{ \yy^\zz: \zz \in \Zpoints\}$ inside $\Div=\RR[\yy]$.
    \end{itemize}
\end{lemma}

\begin{proof}
Since $\Zpoints$ is shifted, the $\RR$-vector
space $I:=\spn_\RR\{\xx^\aa: \aa \in \NN^n \setminus \Zpoints\}$  is a monomial ideal.  
Note that all assertions in the lemma would follow from showing (i), that is, $I=\gr \, \II(\Zpoints)$.

To see this, we first show the inclusion $I \subseteq \gr \, \II(\Zpoints)$. Since $\Zpoints$ is shifted, we claim that for any $\aa \in \NN^n \setminus \Zpoints$, the polynomial 
$$
f_\aa(\xx) := \prod_{i=1}^n x_i (x_i - 1) \cdots (x_i - (a_i - 1))
$$
    lies in the vanishing ideal $\II(\Zpoints)$ for $\Zpoints$.  To see this claim, assume not, so there exists some point $\zz\in \Zpoints$ having $f(\zz) \neq 0$.
    But then for each $i$ one must have $z_i \in \NN \setminus \{0,1,2,\ldots,a_i-1\}$, implying $\zz \geq \aa$, and contradicting $\aa \not\in \Zpoints$.
    Hence $\gr \,\II(\Zpoints)$ contains 
    $
    \tau(f_\aa) = \prod_{i=1}^n x_i^{a_i}=\xx^\aa,
    $
    showing $I \subseteq \gr \, \II(\Zpoints)$.
    
    To show the opposite inclusion $I \supseteq \gr \, \II(\Zpoints)$, note that the surjection
    $S/I \twoheadrightarrow S/\gr \, \II(\Zpoints)$ must be an isomorphism by dimension-counting, since
     $\dim_\RR S/I = \#\Zpoints=\dim_\RR S/\gr \, \II(\Zpoints)$.
\end{proof}

If one defines the following $q$-weight enumerator for 
finite point loci $\Zpoints \subset \NN^n$
\begin{equation}
\label{eq:naive-locus-weight-enumerator}
w_\Zpoints(q):=\sum_{\zz \in \Zpoints} q^{|\zz|} \qquad \text{ where }|\zz|:=z_1+\cdots+z_n,
\end{equation}
then the next corollary is immediate from Lemma~\ref{lem:lower-harmonic}

\begin{cor}
\label{cor:hilb-is-easier-when-shifted}
    For finite shifted point loci $\Zpoints \subset \NN^n$, one has
$$
\Hilb(R(\Zpoints),q)=\Hilb(V_\Zpoints,q)=w_\Zpoints(q).
$$
\end{cor}

We next explain how antiblocking polytopes $P$ give rise to shifted point loci.  
For $0 \leq d \leq n$, let 
$$
(\one_d,\origin_{n-d}):=(\underbrace{1,1,\ldots,1}_{d\text{ times}},
\underbrace{0,0,\ldots,0}_{n-d\text{ times}}).
$$

\begin{prop}
\label{prop:antiblocking-polytopes-yield-shifted-loci}
    Let $P$ be an antiblocking $d$-polytope in $\RR^n$.
    \begin{itemize}
        \item[(i)] The finite point locus $\ZZ^n \cap P \subset \NN^n$ is shifted, and the same holds for
        $\ZZ^n \cap mP$ for all $m \in \NN$.
        \item[(ii)] There exists a re-indexing of the coordinates in $\RR^n$ such that the translated set
        $$
        \ZZ^n \cap \interior{P} - (\one_d,\origin_{n-d}):=
        \{ \zz - (\one_d,\origin_{n-d}): \zz \in \ZZ^n \cap \interior{P}\}
        $$
        is shifted, and the same holds for the set 
        $\ZZ^n \cap \interior{mP} - (\one_d,\origin_{n-d})$
        for all $m \in \NN$.
    \end{itemize}  
\end{prop}
\begin{proof}
First note that whenever $P$ is antiblocking, then so is $mP$. Thus
one only needs to check the first part of both assertions (i),(ii) for $P$.  

Assertion (i) for $P$ is straightforward from the definitions of antiblocking and shiftedness.  

To prove assertion (ii) for $P$, consider 
the {\it support sets} 
$
\supp(\zz):=\{i=1,2,\ldots,n: z_i \neq 0\}
$
for $\zz \in P$.  Since $P$ is convex and lies in $\RR_{\geq 0}$, whenever $\zz,\zz' \in P$, the points $\zz''$ in the interior of the line segment between them have $\supp(\zz'')=\supp(\zz) \cup \supp(\zz')$.  Consequently, there exists $\zz \in P$ with 
$
\supp(\zz)=\bigcup_{\zz' \in P}\supp(\zz').
$
Reindexing coordinates so that $\supp(\zz)=\{1,2,\ldots,e\}$, one concludes that $P$ contains
the $e$-dimensional parallelepiped of vectors componentwise between $\origin$ and $\zz$,
and also $P$ is contained in the $e$-dimensional coordinate subspace $\RR^e \times \{ \origin_{n-d} \}$.  Since $P$ is $d$-dimensional, this forces $e=d$,
and the last $n-d$ coordinates vanish on all points of $P$.

Therefore in the rest of the proof that 
$\ZZ^n \cap \interior{P} - (\one_d,\origin_{n-d})$
is shifted, it suffices to project away the last $n-d$ coordinates which all vanish on $P$, and assume $d=n$.
As a full $d$-dimensional antiblocking polytope $P \subset \RR^d$, starting with any inequality description 
$$P=\{\zz \in \RR^d: \zz \geq \origin, A\zz \leq \bb\}$$ 
where $A,\bb$ have all entries in $\RR_{\geq 0}$, one can describe $\interior{P}$ as
$$
\interior{P}=\{\zz \in \RR^d: \zz > \origin, A\zz <\bb\}. 
$$
Then the points of $\ZZ^d \cap \interior{P}$
have this description:
$$
\ZZ^d \cap \interior{P}= \{\zz \in \ZZ^d: \zz \geq \one_d, A\zz <\bb\}.
$$
Hence the points of $\ZZ^d \cap \interior{P}-\one_d$ have this
description, which defines a shifted subset of $\NN^n$:
$$
\ZZ^d \cap \interior{P}-\one_d= \{\zz \in \ZZ^d: \zz \geq \origin_d, A\zz <\bb - A\one_d\}.\qedhere
$$
\end{proof}

Proposition~\ref{prop:antiblocking-polytopes-yield-shifted-loci} lets us re-cast the $q$-Ehrhart series
for an antiblocking polytope in terms of the following simpler and more classical weight enumerators, involving $w_\Zpoints(q)$
from \eqref{eq:naive-locus-weight-enumerator}.

\begin{definition} \rm
\label{def:naive-q-enumerator}
    For any lattice polytope $P \subset \RR^m$, define
    two $q$-weight enumerator series
$$
        W_P(t,q):=\sum_{m=0}^\infty w_{\ZZ^n \cap mP}(q) \cdot t^m, \quad \text{ and } \quad
        \overline{W}_P(t,q):=\sum_{m=0}^\infty w_{\ZZ^n \cap \interior{mP}}(q) \cdot t^m.
$$  
\end{definition}

Combining Proposition~\ref{prop:antiblocking-polytopes-yield-shifted-loci}, Corollary~\ref{cor:hilb-is-easier-when-shifted}, and Proposition~\ref{prop:affine-invariance}  yields this consequence.
\begin{cor}
 \label{cor:harmonic-equal-naive-Ehrhart-for-antiblockers}
    For any $d$-dimensional lattice polytope $P \subset \RR^n$ which is antiblocking, one has
    \begin{align}
       \Eseries_P(t,q)&=W_P(t,q),\\
        q^d \cdot \intEseries_P(t,q)&= \overline{W}_P(t,q).
    \end{align}
\end{cor}

Classical results now apply to the weight enumerators $W_P(t,q), \overline{W}_P(t,q)$, showing that they have rational expressions with predictable properties, for {\it any} lattice polytopes $P$.

\begin{prop}
\label{prop:naive-weighted-Ehrhart-series-are-simpler}
    Let $P \subset \RR^n$ be any $d$-dimensional lattice polytope, with vertices $\vv^{(1)},\vv^{(2)},\ldots,\vv^{(\nu)}$.
    \begin{itemize}
        \item[(i)] Both $W_P(t,q), \overline{W}_P(t,q)$ have rational expressions with the same denominator 
$$
        W_P(t,q)=\frac{N_P(t,q)}{\prod_{i=1}^\nu
        (1-tq^{|\vv^{(i)}|})}
        \quad \text{ and } \quad
        \overline{W}_P(t,q)=\frac{\overline{N}_P(t,q)}{\prod_{i=1}^\nu
        (1-tq^{|\vv^{(i)}|})}
$$
and numerators $N_P(t,q),\overline{N}_P(t,q) \in \ZZ[t,q]]$.
        \item[(ii)] For a lattice $d$-simplex $P$
        with vertices $\{ \vv^{(i)}\}_{i=1,2,\ldots,d+1}$, 
        one has these numerators in $\NN[t,q]$
        $$
        N_P(t,q)=\sum_{\zz \in \ZZ^{n+1} \cap \Pi} t^{z_{n+1}} q^{z_1+\cdots+z_n} 
        \quad \text{ and } \quad
        \overline{N}_P(t,q)=\sum_{\zz \in \ZZ^{n+1} \cap \Pi^\opp} t^{z_{n+1}} q^{z_1+\cdots+z_n},
        $$
        where $\Pi$ is the semi-open parallelepiped from \eqref{eq: semi-open-parallelepiped}, and $\Pi^{\opp}$ its ``opposite", defined by
        $$    
        \Pi:=\left\{ \sum_{j=1}^{d+1} c_j \cdot (1,\vv^{(j)}): 0 \leq c_i < 1\right\} 
        \quad \text{ and } \quad 
        \Pi^{\opp}:=\left\{ \sum_{j=1}^{d+1} c_j \cdot (1,\vv^{(j)}): 0 < c_i \leq 1\right\}.
        $$
        \item[(iii)] The two series $W_P(t,q), \overline{W}_P(t,q)$ determine each other via
        $$
        \overline{W}_P(t,q) = (-1)^{d+1} W_P(t^{-1},q^{-1}).
        $$
    \end{itemize}
\end{prop}
\begin{proof}
These are all bivariate specializations of results of Stanley on
multivariable lattice point enumerators for rational polyhedral cones, implicit in \cite{Stanley-reciprocity}, and more explicit in \cite[\S4.5]{Stanley-EC1}.  One must apply them to
$\cone(P) \subset \RR^{n+1}$ which is defined to
be the $(d+1)$-dimensional cone nonnegatively spanned by the vectors 
$\{(1,\vv^{(j)})\}_{j=1,\ldots,\nu}$. In variable set $\xx=(x_0,x_1,\ldots,x_n)$, Stanley considers
$$
E_{\cone(P)}(\xx):=\sum_{\zz \in \cone(P)} \xx^\zz
\quad \text{ and }\quad 
\overline{E}_{\cone(P)}(\xx):=\sum_{\zz \in \interior{\cone(P)}} \xx^\zz.
$$
He proves \cite[Thm. 4.5.11]{Stanley-EC1}
that one has rational expressions of the form
\begin{equation}
\label{eq:Stanley-multivariate-denominators}
E_{\cone(P)}(\xx)=\frac{N_P(\xx)}{\prod_{i=1}^\nu
        (1-\xx^{\vv^{(i)}})}
\quad \text{ and }\quad 
\overline{E}_{\cone(P)}(\xx)=
\frac{\overline{N}_P(\xx)}{\prod_{i=1}^\nu
        (1-\xx^{\vv^{(i)}})}
\end{equation}
for some polynomials $N_P(\xx),\overline{N}_P(\xx)$ in $\ZZ[\xx]$.
He also proves \cite[Cor. 4.5.8]{Stanley-EC1} that a $d$-simplex $P$ has
\begin{equation}
\label{eq:Stanley-simplex-formulas}
N_P(\xx)=\sum_{\zz \in \ZZ^{n+1} \cap \Pi} \xx^\zz
\quad \text{ and } \quad
\overline{N}_P(\xx)=\sum_{\zz \in \ZZ^{n+1} \cap \Pi^{\opp}} \xx^\zz.
\end{equation}
And he shows in 
\cite[Thm. 4.5.14]{Stanley-EC1}
(often called {\it Stanley's Reciprocity Theorem}) that
\begin{equation}
\label{eq:Stanley-reciprocity}
E_P(x_0^{-1},x_1^{-1},\ldots,x_n^{-1})
=(-1)^{d+1} \overline{E}_P(x_0,x_1,\ldots,x_n).
\end{equation}
Assertions (i),(ii),(iii) are the
specializations at $x_0=t,x_1=\cdots=x_n=q$ of \eqref{eq:Stanley-multivariate-denominators},\eqref{eq:Stanley-simplex-formulas},\eqref{eq:Stanley-reciprocity}.
\end{proof}

\begin{cor}
    Conjecture~\ref{conj:intro-omnibus} holds for all lattice polytopes $P$ which are antiblocking.
\end{cor}
\begin{proof}
Combining
Corollary~\ref{cor:harmonic-equal-naive-Ehrhart-for-antiblockers}
and Proposition~\ref{prop:naive-weighted-Ehrhart-series-are-simpler} shows that an antiblocking polytope $P$ with vertices $\{\vv^{(1)},\ldots,\vv^{(\nu)}\}$ satisfies
Conjecture~\ref{conj:intro-omnibus}
with denominator 
$D_P(t,q)=\prod_{i=1}^\nu (1-t q^{|\vv^{(i)}|})$.
\end{proof}

\begin{remark}
\label{rmk:lookahead-to-chain-polytopes}
See Section~\ref{sec: chain-order-polytope} below for a discussion of an interesting family of antiblocking polytopes mentioned
in the Introduction:  the {\it chain polytope} $C_\PPP$ associated to a finite poset $\PPP$.
It will be shown there that $C_\PPP$ shares the same $q$-Ehrhart series as the {\it order polytope}
$O_\PPP$ associated to $\PPP$, generalizing 
a result of Stanley \cite{stanley1986two} on
their classical Ehrhart series.
\end{remark}

\begin{remark}
\label{remark-first-interesting-triangle-discussion}
    Not all lattice polytopes $P$ are equivalent via $\Aff(\ZZ^d)$
    to antiblocking polytopes.  In particular, sometimes $\gr\,\II(\ZZ^d \cap P)$ is not a monomial ideal in $\RR[\xx]$, and $V_{\ZZ^d \cap P} \subset \RR[\yy]$ has no monomial $\RR$-basis.  For example, the lattice triangle $P$ shown here
    with vertices $\{(0,0),(2,1),(1,2)\}$ 
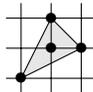
\begin{center}
    \begin{tikzpicture}[scale = 0.4]
        \draw[fill = black!10] (0,0) -- (1,2) -- (2,1) --  (0,0);
        \draw[step=1.0,black,thin] (-0.5,-0.5) grid (2.5,2.5);
        \node at (0,0) {$\bullet$};
        \node at (1,1) {$\bullet$};
        \node at (1,2) {$\bullet$};
        \node at (2,1) {$\bullet$};
    \end{tikzpicture}
\end{center}
is equivalent to the second triangle of area $3$ in Figure~\ref{fig:q-Ehrhart-examples}, with vertices $\{(0,0),(1,0),(2,3)\}$.  It has
\begin{align*}
\Eseries_P(t) &=\frac{1+ t + t^2 }{(1-t)^3},\\
\Eseries_P(t,q) &= \frac{(1+qt)(1+ q t + q^2 t^2)}{(1-t)(1- q^2 t)(1 - q^3 t^2)}.
\end{align*}
Via hand calculation or using {\tt Macaulay2}, one can check that
\begin{equation}
\gr \, \II(\ZZ^2 \cap P) = (\,\, x_1^2 - x_2^2,\,\, 2 x_1x_2 - x_2^2,\,\, x_2^3\,\,) \subseteq \RR[x_1,x_2]
\end{equation}
and this ideal cannot be generated by monomials.
Again hand calculation or {\tt Macaulay2} shows that
$$
V_\Zpoints=\spn_\RR\{1,\,\, y_1,y_2,\,\, y_1^2+y_1y_2+y_2^2\}
$$
which has no $\RR$-basis of monomials.
\end{remark}

\begin{remark}
\label{rmk:Chapoton-remark}
The $q$-weight enumerators in Definition~\ref{def:naive-q-enumerator} bear some resemblance to work of Chapoton \cite{Chapoton}.  He introduced $q$-analogues of $\Eseries_P(t)$
that first require a choice a $\ZZ$-linear functional $\lambda: \ZZ \rightarrow \ZZ$ which is nonconstant along edges of $P$
and has $\lambda(\zz^{(i)}) \geq 0$ for all vertices of $P$. 
Having made such a choice,  he defined
\begin{align*}
W_\lambda(P,q):=\sum_{\zz \in \ZZ^n \cap P} q^{\lambda(\zz)} 
\quad \text{ and } \quad
\mathrm{Ehr}_{P,\lambda}(t,q):=\sum_{m \geq 0} t^m W_\lambda(mP,q). 
\end{align*}
Therefore Corollary~\ref{cor:harmonic-equal-naive-Ehrhart-for-antiblockers} shows that if a lattice polytope $P$ is antiblocking,
and happens to have that the functional $\lambda(\zz):=|\zz|=z_1+\cdots+z_n$ is nonconstant along its edges, then $\mathrm{Ehr}_{P,\lambda}(t,q)=\Eseries_P(t,q)$. 
This $q$-analogue is also considered by Adeyemo and Szendr\H{o}i \cite{AdeyemoSzendroi},
who calculate it for standard simplices, cross-polytopes and cubes.
\end{remark}

\subsection{Definition of the equivariant $q$-Ehrhart series}

We next incorporate symmetries
of subgroups of $GL_n(\ZZ)$ acting on a lattice polytope. Such groups are often
called {\it crystallographic groups}.

\begin{definition} \rm
 \label{def:equivariant-q-Ehrhart} 
    Let $P$ be a lattice polytope in $\RR^n$ which is stable under the action of a finite crystallographic group $G \subset GL_n(\ZZ)$. Define the {\em equivariant $q$-Ehrhart series} 
    \[
    \Eseries^G_P(t,q) := \sum_{m =0}^\infty  [ V_{\ZZ^n \cap mP} ]_q \cdot t^m
    \qquad \left(
    = \sum_{m =0}^\infty  [ R(\ZZ^n \cap mP) ]_q \cdot t^m
    \right)
    \]
    where recall $[U]_q=\sum_{d \geq 0} [U_d] q^d$ is the class in $\Cl_\ZZ(G)[[q]]$ of any graded $\kk[G]$-module $U=\oplus_{d \geq 0} U_d$. Hence $\Eseries^G_P(t,q)$ lies in the power series ring $\Cl_\ZZ(G)[[t,q]]$,
    and even lies in its subring $\Cl_\ZZ(G)[q][[t]]$.
\end{definition}

When $q \to 1$, this series $\Eseries^G_P(t,q)$ specializes to Stapledon's {\it equivariant Ehrhart series} \cite{Stapledon} involving the ungraded characters of $V_{mP \cap \ZZ^n}$, or equivalently the permutation $\RR[G]$-modules $\ZZ^n \cap mP$.

\begin{example}
\label{ex:repeat-from-intro-with-symmetry}
We return to Example~\ref{ex:repeat-from-intro}, where $P$ is a lattice polytope $P \subset \RR^1$, but now assume it is stable under the action of a nontrivial crystallographic group $G$.
This implies $P=[-b,b]$ for some $b \geq 0$ and
$G = GL_1(\ZZ)=\{\pm 1\}$.
One can identify the representation ring  
$$
\Cl_\RR(G) \cong \ZZ[\epsilon]/(\epsilon^2-1) = \spn_\ZZ \{1,\epsilon\},
$$
where $1, \epsilon$ denote the isomorphism classes of the one-dimension trivial and nontrivial $\RR[G]$-modules, respectively.  A calculation using $\ZZ^1 \cap mP=\{-mb,-mb+1,\ldots,mb-1,mb\}$, similar to the one in the Introduction, shows that 
$$
\II(\ZZ^1 \cap mP)=\left( \prod_{i=-mb}^{+mb} (x-i) \right), \quad 
 \gr\, \II(\ZZ^1 \cap mP) = (x^{2mb+1}) \subset S=\RR[x]
$$
and consequently, one has 
\begin{align*}
R(\ZZ^1 \cap mP) = S/(x^{2bm+1}) &=
\spn_\RR \{1, \bar{x}, \bar{x}^2, \dots, \bar{x}^{2bm} \},\\
V_{\ZZ^1 \cap mP}  =\ker\left( \left( \frac{\partial}{\partial y}\right)^{2bm+1}
: \Div \rightarrow \Div \right)&= \spn_\RR\{1,y,y^2,\ldots,y^{2bm}\}
\subset \Div=\RR[y].
\end{align*}
Since $G=\{\pm 1\}$ negates the variables $x,y$ in $S,\Div$, one concludes (with some algebra) that
\begin{align}
    \Eseries^G_P(t,q) &= \sum_{m=0}^\infty 
      (1 + q^2 + q^4 + \cdots + q^{2mb})t^m + \sum_{m=0}^\infty (q + q^3 + q^5 + \cdots + q^{2mb-1}) \cdot \epsilon t^m\\ 
\label{eq:line-segement-equivariant-series}
&= \frac{1}{1-q^2}\left[
    \frac{1}{1-t} -\frac{q^2}{1-tq^{2b}}
    + q \epsilon\left( 
    \frac{1}{1-t} -\frac{1}{1-tq^{2b}}
    \right)
    \right]
    =\frac{1+t\left( q^2[b-1]_{q^2} + q[b]_{q^2} \epsilon\right)}
    {(1-t)(1-tq^{2b})}
\end{align}
This expression \eqref{eq:line-segement-equivariant-series} 
specializes to the non-equvariant series $\Eseries_P(t,q)$ by applying the ring homomorphism $\Cl_\RR(G)\rightarrow \ZZ$
that sets $\epsilon=1$, giving
an answer consistent with \eqref{eq:one-dimensional-Eseries-repeat} at
$v=\vol_1(P)=2b$:
$$
\Eseries_P(t,q)
=\left[ \Eseries^G_P(t,q) \right]_{\epsilon=1}
=\frac{1+t\left( q^2[b-1]_{q^2} + q[b]_{q^2}\right)}
    {(1-t)(1-tq^{2b})}
=\frac{1+t q[2b-1]_q}
    {(1-t)(1-tq^{2b})}.
$$

On the other hand, setting $q=1$ in \eqref{eq:line-segement-equivariant-series} 
should also give Stapledon's equivariant Ehrhart series $\Eseries^G_P(t)$, which we
calculate directly here.  The $G$-orbit structure permuting
$$
\ZZ^1 \cap mP=\{-mb,\ldots,-1,0,+1,\ldots,mb\}
=\{0\} \sqcup \bigsqcup_{i=1}^{mb} \{\pm i\}
$$
has one fixed point orbit $\{0\}$ with class $1$ in $\Cl_\RR[G]$
and $mb$ free orbits $\{\pm i\}_{i=1,2,\ldots,mb}$,
each with class $1+\epsilon$
in  $\Cl_\RR[G]$.  Hence one calculates that
$$
    \Eseries^G_P(t):=\sum_{m=0}^\infty
    [\ZZ^1 \cap mP] t^m
    =\sum_{m=0}^\infty \left( 1 + (1+\epsilon) mb) \right) t^m 
    =\frac{1}{1-t} + \frac{(1+\epsilon)bt}{(1-t)^2}
    =\frac{1+t((b-1)+b\epsilon)}{(1-t)^2},
$$
which agrees with the expression for $\Eseries^G_P(t,q)$ in \eqref{eq:line-segement-equivariant-series} at $q=1$.
\end{example}

\subsection{Examples: standard simplices} 
\label{sec:standard-simplices}
We compute the equivariant $q$-Ehrhart series for
two families of simplices that carry an action of the symmetric group $\symm_n$ on $n$ letters, permuting the standard basis vectors $\ee_1, \dots, \ee_n$ in $\RR^n$.

\begin{definition} \rm
The {\em standard $(n-1)$-simplex} in $\RR^n$ is 
$$
\Delta^{n-1} := \mathrm{conv}\{ \ee_1, \dots, \ee_n\},
$$
lying inside the affine subspace of $\RR^n$ where $x_1+\cdots+x_n=1$.
There is also a full $n$-dimensional simplex in $\RR^n$ which is the {\it pyramid} $\pyr(\Delta^{n-1})$ having $\Delta^{n-1}$ as its base, and apex at the origin $\origin$:
$$
\pyr(\Delta^{n-1}) := \mathrm{conv}\{\origin, \ee_1, \dots, \ee_n\}.
$$
\end{definition}

We compute the equivariant $q$-Ehrhart series of $\Delta^{n-1}, \pyr(\Delta^{n-1})$.
For this, it helps to have two facts.  
The first is an easy lemma on
power series $F(q) = \sum_{i \geq 0} a_i q^i \in R[[q]]$ for commutative rings $R$, and their 
{\em truncations to $q^{\leq m}$},  
defined as these polynomial partial sums in $R[q]$:
$$
\trunc_{\leq q^m}F(q):=\sum_{i = 0}^m a_i q^i.
$$

\begin{lemma}
\label{lem:truncation-series-lemma}
For any $F(q)$ in $R[[q]]$, one has
    $
    \sum_{m=0}^\infty t^m \cdot \trunc_{\leq q^m}F(q) 
    =  \frac{F(tq)}{1-t}
    $ 
in $R[[t,q]]$.
\end{lemma}

The second fact is a formula due to Lusztig and Stanley \cite[Prop. 4.11]{Stanley-invariant-theory} for the $\symm_n$-equivariant Hilbert series of the polynomial ring $S=\RR[\xx]$.
Recall that the simple $\RR[\symm_n]$-modules
$\{ \specht^\lambda \}$ are indexed by number partitions $\lambda$ of $n$ (written $\lambda \vdash n$), so that
$\Cl_\RR(\symm_n)$ has $\ZZ$-basis $\{[\specht^\lambda]: \lambda \vdash n\}$.

\begin{thm}
\label{thm:Lusztig-Stanley}
(Lusztig, Stanley)
The graded $\RR[\symm_n]$-module $S=\RR[\xx]$ has this class in $\Cl_\RR(\symm_n)[[q]]$:
$$
[S]_q:=\sum_{d=0}^\infty q^d [S_d]
=\frac{1}{(1-q)(1-q^2) \cdots (1-q^n)}
\sum_{\lambda \vdash n}
f^\lambda(q) \cdot [\specht^\lambda]
$$
with $f^\lambda(q)$ the $\symm_n$ fake-degree polynomial, having these sum,  product expressions \cite[Cor. 7.21.5]{Stanley-EC2}
\begin{equation}
    \label{eq:two-fake-degree-formulas}
f^\lambda(q)=\sum_T  q^{\mathrm{maj}(T)}
=q^{b(\lambda)} \frac{[n]!_q}{\prod_{x \in \lambda} [h(x)]_q} 
\end{equation}
whose notations are explained below.
\end{thm}
\noindent
Here $T$ in the summation in \eqref{eq:two-fake-degree-formulas} runs over all {\it standard Young tableaux} of shape $\lambda$, and $\mathrm{maj}(T)$ is the sum of all entries $i$ in $T$ for which $i+1$ appears
weakly southwest of $i$ (using English notation for tableaux).  The {\it $q$-factorial} is $[n]!_q:=[n]_q [n-1]_q \cdots [2]_q [1]_q$. For a partition $\lambda$ with $\ell$ parts, we have $b(\lambda):=\sum_{i=1}^\ell (i-1) \lambda_i$.
For a box $x$ in row $(i,j)$ of the Ferrers diagram of $\lambda$ (written $x \in \lambda$ in the product), the {\it hooklength} is $h(x):=\lambda_i+\lambda_j'-(i+j)+1$, with
$\lambda'$ the {\it conjugate partition} to $\lambda$.

\medskip

We first compute 
the equivariant $q$-Ehrhart series for the $n$-simplex $\pyr(\Delta^{n-1}) \subset \RR^n$.

\begin{prop}
\label{prop:pyramid-over-standard-simplex}
For the $n$-simplex $P:=\pyr(\Delta^{n-1})
=\conv\{\origin,\ee_1,\ldots,\ee_n\}\subset \RR^n$, one has
\begin{align}
\label{eq:equivariant-q-Ehrhart-for-pyramid-simplex}
\Eseries_P^{\symm_n}(t,q)
&=
\frac{1}{(1-t) \cdot \prod_{i=1}^n (1-t^i q^i)}
 \sum_{\lambda \vdash n} f^{\lambda}(tq) \cdot [\specht^\lambda],
\\
\label{eq:nonequivariant-q-Ehrhart-for-pyramid-simplex}
\Eseries_P(t,q)
&
=\frac{1}{(1-t)(1-tq)^{n}}.
\end{align}
\end{prop}
\begin{proof}
We start with \eqref{eq:equivariant-q-Ehrhart-for-pyramid-simplex}.
Note $P$ and all of its dilates $mP$ are antiblocking,
with $mP$ defined by the inequalities $\xx \geq 0$ and $x_1+\cdots+x_n \leq m$.
Hence Lemma~\ref{lem:lower-harmonic} shows that
$$
R(\ZZ^n \cap mP)=S/S_{\geq m+1}.
$$
Thus the graded $\RR[\symm_n]$-module $R(\ZZ^n \cap mP)$ is the same as the permutation module on the monomials of degree at most $m$ in $S=\RR[\xx]$.  Therefore in $\Cl_\RR(\symm_n)[[q]]$ one has
$$
[ R(\ZZ^n \cap mP) ]_q=\trunc_{q^{\leq m}}[S]_q.
$$
Applying Lemma~\ref{lem:truncation-series-lemma} then gives the equality marked (*) here
\begin{align}
\notag
\Eseries^{\symm_n}_P(t,q)
&:=\sum_{m=0}^\infty
[ R(\ZZ^n \cap mP) ]_q \cdot t^m \\
\label{eq:upshot-of-pyramid-calculation}
&\overset{(*)}{=}\frac{1}{1-t} \cdot [S]_{q \mapsto tq}
=\frac{1}{(1-t) \cdot \prod_{i=1}^m (1-t^i q^i)}
 \sum_{\lambda \vdash n} f^{\lambda}(tq) \cdot [\specht^\lambda],
\end{align}
with the last equality substituting $q \mapsto tq$ in the formula
from Theorem~\ref{thm:Lusztig-Stanley}.
This proves \eqref{eq:equivariant-q-Ehrhart-for-pyramid-simplex}.

Although one could deduce \eqref{eq:nonequivariant-q-Ehrhart-for-pyramid-simplex} from 
\eqref{eq:equivariant-q-Ehrhart-for-pyramid-simplex},
a perhaps simpler path repeats the same argument but replacing $[S]_q$ with the known non-equivariant Hilbert series $\Hilb(S,q)=1/(1-q)^n$.
\end{proof}

Now it is easier to deal with the 
equivariant series for the $(n-1)$-simplex $\Delta^{n-1} \subset \RR^n$.

\begin{prop}
\label{prop:standard-simplex}
For $P:=\Delta^{n-1}
=\conv\{\ee_1,\ldots,\ee_n\}\subset \RR^n$, one has
\begin{align}
\label{eq:equivariant-q-Ehrhart-for-standard-simplex}
\Eseries_P^{\symm_n}(t,q)
&=(1-tq) \cdot \Eseries_{\pyr(P)}^{\symm_n}(t,q)
=\frac{1}{(1-t) \cdot \prod_{i=2}^n (1-t^i q^i)}
 \sum_{\lambda \vdash n} f^{\lambda}(tq) \cdot [\specht^\lambda]
\\
\label{eq:nonequivariant-q-Ehrhart-for-standard-simplex}
\Eseries_P(t,q)
&=(1-tq) \cdot \Eseries_{\pyr(P)}(t,q)
=\frac{1}{(1-t)(1-tq)^{n-1}}. 
\end{align}
\end{prop}
\begin{proof}
We start by identifying $\gr\,\II(\ZZ^n \cap mP)$.
Since $\ZZ^n \cap mP \subset \ZZ^n \cap m\pyr(P)$,
one has 
\begin{align*}
\II(\ZZ^n \cap mP) &\supset \II(\ZZ^n \cap m\pyr(P)),\\
\gr\,\II(\ZZ^n \cap mP) &\supset \gr\, \II(\ZZ^n \cap m\pyr(P))=S_{\geq m+1}.
\end{align*}
On the other hand, since $mP$ lies in the affine hyperplane $x_1+\cdots+x_n=m$ within $\RR^n$, one also has $x_1+\cdots+x_n-m \in \II(\ZZ^n \cap mP)$ and $e_1:=x_1+\cdots+x_n \in \gr\,\II(\ZZ^n \cap mP)$.
Consequently, 
$$
\gr\,\II(\ZZ^n \cap mP) \supseteq (e_1) + S_{\geq m+1},
$$
and we claim that this inclusion is actually an equality.  To see this, note that the surjection 
$$
S/\left( (e_1)+ S_{\geq m+1} \right)
 \twoheadrightarrow 
S/\gr\,\II(\ZZ^n \cap mP) = R(\ZZ^n \cap mP)
$$
is an isomorphism via dimension-counting:
since $mP$ is the dilation of a standard $(n-1)$-simplex,
$$
\dim_\RR R(\ZZ^n \cap mP)
=\#\ZZ^n \cap mP
=\binom{m+n-1}{m},
$$
and if one abbreviates the quotient ring $\overline{S}:=S/(e_1) \cong \RR[x_1,\ldots,x_{n-1}]$,
then
$$
\dim_\RR S/\left( (e_1)+ S_{\geq m+1}) \right)
= \dim_\RR \overline{S}/ \overline{S}_{\geq m+1}
=\binom{m+n-1}{m}.
$$
From this one concludes that in $\Cl_\RR(\symm_n)[q]$ one has
$$
[R(\ZZ^n \cap mP)]_q = \trunc_{\leq q^m} [\overline{S}]_q
$$
and hence, applying Lemma~\ref{lem:truncation-series-lemma}, one has
$$
\Eseries^{\symm_n}_P(t,q)
:=\sum_{m=0}^\infty
[ R(\ZZ^n \cap mP) ]_q t^m
=\frac{1}{1-t} \cdot [\overline{S}]_{q \mapsto tq}.
$$
On the other hand, we claim that 
$
[\overline{S}]_{q}=(1-q) \cdot [S]_q.
$
This holds
because $e_1=x_1+\cdots+x_n$ is an $\symm$-invariant
nonzero divisor in $S$, so that multiplying by $e_1$ on $S$ gives rise to an exact sequence of graded $\RR[\symm_n]$-modules $0 \rightarrow S(-1) \rightarrow S \rightarrow \overline{S} \rightarrow 0$.
Comparing this with \eqref{eq:upshot-of-pyramid-calculation} then proves \eqref{eq:equivariant-q-Ehrhart-for-standard-simplex}.  

The proof of
\eqref{eq:nonequivariant-q-Ehrhart-for-standard-simplex} is again a parallel, but easier computation.
\end{proof}

\begin{remark}
We note here some notationally convenient reformulations of these results, for 
readers  familiar with the {\it ring of symmetric functions} $\Lambda=\bigoplus_{n=0}^\infty \Lambda_n$ and the {\it Frobenius characteristic isomorphism}
$
\Lambda_n \cong \Cl_\RR(\symm_n);
$
see \cite{Macdonald, Sagan, Stanley-EC2} for more background. One note of
{\bf caution:} multiplication in $\Cl_\RR(\symm_n)$ corresponds to the {\it internal} or {\it Kronecker product} on $\Lambda_n$, not
the {\it external} or {\it induction product}
corresponding to the multiplication $\Lambda_{n_1} \times \Lambda_{n_2} \rightarrow \Lambda_{n_1+n_2}$.

This Frobenius characteristic isomorphism maps the {\it Schur function} $s_\lambda \mapsto [\specht^\lambda]$, and the {\it complete homogeneous} symmetric function $h_n \mapsto [1]$. 
Using the {\it plethystic notation}, explained in Bergeron \cite[Ch.~3]{Bergeron}, Haglund \cite{Haglund-qt}, Loehr and Remmel \cite{LoehrRemmel}, one can show (see \cite[eqn. (7.7)]{Bergeron}) that the class $[S]_q$ of the 
graded $\RR[\symm_n]$-module $S$ corresponds  to
this element of $\Lambda_n[[q]]$:
$$
[S]_q \mapsto h_n[X+qX+q^2X+\cdots]=
h_n\left[\frac{X}{1-q}\right].
$$
The Lusztig-Stanley 
Theorem~\ref{thm:Lusztig-Stanley}
expanding this plethysm in Schur functions is
discussed in Haglund \cite[(1.89)]{Haglund-qt}.
Consequently, one can recast  Propositions~\ref{prop:standard-simplex} and \ref{prop:pyramid-over-standard-simplex} plethystically:
$$
\Eseries_{\pyr^n}^{\symm_n}(t,q)
= \frac{1}{1-t} \cdot h_n\left[\frac{X}{1-tq}\right]
\quad \text{and} \quad
\Eseries_{\Delta^{n-1}}^{\symm_n}(t,q)
= \frac{1-tq}{1-t} \cdot h_n\left[\frac{X}{1-tq}\right].
$$
\end{remark}

\begin{remark}
\label{rmk:join-pyramid}
    The relation $\Eseries_P(t,q)
=(1-tq) \cdot \Eseries_{\pyr(P)}(t,q)$
for $P=\Delta^{n-1}$ that appeared 
in Propositions~\ref{prop:pyramid-over-standard-simplex},
\ref{prop:standard-simplex} anticipates
the behavior of $\Eseries_P(t,q)$ under the pyramid
operation $P \mapsto \pyr(P)$, and more generally
under free joins $P*Q$ of lattice polytopes, as described in Theorem~\ref{thm:intro-three-constructions-on-series}(iii).
Pyramids are the special case $\pyr(P)=P*Q$
where $Q$ is the one-point polytope of dimension zero.
\end{remark}

\subsection{Examples: cross-polytopes}

The {\em cross-polytope} 
or {\em hyperoctahedron} $\Cross^n \subseteq \RR^n$ is defined by
$$
\Cross^n:=\conv\{ \pm \ee_1, \dots, \pm \ee_n\}.
$$
Its symmetries are the 
{\it hyperoctahedral group}
$\symmB_{n}$ of all {\it signed permutation
matrices}, acting by permuting and negating coordinates
in $\RR^n$.  We first analyze the non-equivariant $q$-Ehrhart series for $\Cross_n$,
and then incorporate the group action.

\begin{prop}
\label{prop:cross-polytope}
The $q$-Ehrhart series of the cross-polytope $\Cross^n$ is given by
\begin{equation}
E_{\Cross^n}(t,q)  
= \frac{1}{1-t} \left( \frac{1 + qt}{1 - q^2 t} \right)^n.
\end{equation}
\end{prop}

\begin{proof}
Let $P:=\Cross^n$.  Then for each $m \geq 0$, the lattice points of the dilate $mP$ are given by
\begin{equation}
    \ZZ^n \cap mP = 
    \{ \zz \in \ZZ^n 
    \,:\, |z_1| + \cdots + |z_n| \leq m \}.
\end{equation}
We wish to identify the ideal 
$\gr \, \II(\ZZ^n \cap m \cdot P) \subset S$.
For this, it helps to note that
every monomial in $S=\RR[\xx]$ can be expressed uniquely in the form that separates even and odd powers of variables
\begin{equation}
\label{eq:monomial-even-odd-expression}
\xx^{2\aa} \cdot \xx_B = x_1^{2a_1} \cdots x_n^{2a_n} \,\, \cdot \,\, \prod_{j \in B} x_j 
\end{equation}
for some pair $(\aa,B)$ with $\aa \in \NN^n$ and $B \subseteq [n]$.
We then have the following claim.
\begin{quote}
{\bf Claim:} 
   $\gr \, \II(\ZZ^n \cap m \cdot P)$
is the monomial ideal $I:=\spn_\RR
\{ \xx^{2\aa} \cdot \xx_B: |\aa|+\#B>m \}.$
\end{quote}
To verify the Claim, we first show that for any such monomial $\xx^{2\aa} \cdot \xx_B$ in $I$, so $|\aa|+\#B>m$,
if one defines the set $A:=\supp(\aa)=\{j: a_j \geq 1\}$, then $\xx^{2\aa} \cdot \xx_B$ 
is divisible by 
$$
\tau(f)=\prod_{j \in A \setminus B} x_j^{2a_j-1} \prod_{j \in B} x_j^{2a_j+1}
$$
where $f$ is the following polynomial
in $\II(\ZZ^n \cap mP)$:
$$
f(\xx) := \prod_{j \in A \setminus B} \,\,
     \prod_{i_j=-(a_j-1)}^{+(a_j-1)} (x_j - i_j) 
\cdot 
\prod_{j \in B} \,\,
     \prod_{i_j=-a_j}^{+a_j} (x_j - i_j). 
$$
To see $f$ has the appropriate vanishing to lie in
$\II(\ZZ^n \cap mP)$, note
that for any $\zz \in \ZZ^n$ where $f(\zz) \neq 0$,
each coordinate $z_j$ with $j \in A \setminus B$ must lie in
$\ZZ \setminus [-(a_j-1),+(a_j-1)]$, and hence have $|z_j| \geq a_j$; similarly, each coordinate 
$z_j$ with $j \in B$ must lie in
$\ZZ \setminus [-a_j,+a_j]$, and hence have $|z_j| \geq a_j+1$
But this shows that $\zz \not\in \ZZ^n \cap mP$, since 
$$
|\zz| =\sum_{j=1}^n |z_i|
\geq \sum_{j \in A \setminus B} a_j 
+ \sum_{j \in B} (a_j+1) 
=|\aa| +\#B> m.
$$

On the other hand, $\ZZ^n \cap mP$ has a bijection to the {\it complementary} set of monomials 
$$
\ZZ^n \cap mP=\{\zz \in \ZZ^n:\sum_{i=1}^n |z_i| \leq m\} 
\quad \longleftrightarrow \quad
\{\xx^{2\aa} \cdot \xx_B: |\aa|+\#B \leq m\},
$$
sending $\zz \longmapsto \xx^{2\aa} \cdot \xx_B$
with $B:=\{j:z_j < 0\}$ and
$$
a_j:=\begin{cases}
z_j &\text{ if }z_j \geq 0,\\
|z_j|-1 &\text{ if }z_j < 0.
\end{cases}
$$
The inverse bijection sends $\xx^{2\aa} \cdot \xx_B \longmapsto \zz$ defined by
$z_j:=a_j$ for $j \not\in B$ and $z_j=-(a_j+1)$ for $j \in B$.
This shows the inclusion
$\gr \, \II(\ZZ^n \cap mP) \supseteq
I$ is an equality, since the surjection
$$
R(\ZZ^n \cap mP)=S/\gr \, \II(\ZZ^n \cap m \cdot P) \twoheadrightarrow
S/I
$$
must be an isomorphism via dimension-counting.

To calculate $\Eseries_P(t,q)$, it helps to recast the unique expressions in \eqref{eq:monomial-even-odd-expression} as a graded $\RR$-vector space isomorphism.
Define the polynomial subalgebra $\RR[\xx^2]=\RR[x_1^2,\ldots,x_n^2]$
within $S=\RR[\xx]$, and introduce the $\RR$-linear subspace $U \subset S$
which is spanned by the squarefree monomials,
that is, $U:=\spn_\RR\{\xx_B: B \subseteq [n]\}$.
One then has the following graded $\RR$-vector space isomorphism:
\begin{equation}
\label{eq:cross-polytope-tensor-iso}
\begin{array}{rcl}
\RR[\xx^2] \,\, \otimes_\RR \,\, U &\longrightarrow& S\\
\xx^{2\aa} \otimes \xx_B & \longmapsto & \xx^{2\aa}\cdot \xx_B.
\end{array}
\end{equation}
Endow the tensor product on the left with a {\it bigrading} or {\it $\NN^2$-grading} 
in which 
$$
\deg_{\NN^2}(\xx^{2\aa} \otimes \xx_B)=(|\aa|,\#B).$$ 
Then the  bigraded Hilbert series in variables $(q_1,q_2)$
that tracks $\xx^{2\aa} \otimes \xx_B$ via $q_1^{|\aa|} q_2^{\#B}$ is
$$
\Hilb_{\NN^2}( \RR[\xx^2] \otimes_\RR U, q_1,q_2) = \left( \frac{1+q_2}{1-q_1} \right)^n.
$$
To recover the {\it usual} $\NN$-grading, tracking
$\xx^{2\aa} \otimes \xx_B$ via $q^{2|\aa|+\#B}$, one must set $q_1=q^2$ and $q_2=q$.  

On the other hand, our 
analysis above shows that the map \eqref{eq:cross-polytope-tensor-iso}, followed by the quotient map $S \twoheadrightarrow S/I \cong R(\ZZ^n \cap mP)$, induces a graded $\RR$-vector space isomorphism
\begin{equation}
\label{eq:hyperoctahedral-truncated-iso}
\left( \RR[\xx^2] \,\, \otimes_\RR \,\, U\right)_{\substack{\leq q_1^{a} q_2^{b}\text{ with}\\ a+b\leq m}} \longrightarrow R(\ZZ^n \cap mP).
\end{equation}
Consequently, using Lemma~\ref{lem:truncation-series-lemma} at the equality labeled (*) below, one has
\begin{align*}
\Eseries_P(t,q)
:=\sum_{m=0}^\infty t^m \cdot \Hilb(R(\ZZ^n \cap mP),q)
&=  \left[ 
\sum_{m=0}^\infty t^m \cdot  \trunc_{\substack{\leq q_1^{a} q_2^{b}\text{ with}\\ a+b\leq m}}
\Hilb( \RR[\xx^2] \otimes_\RR  U, q_1,q_2)
\right]_{\substack{q_1\mapsto q^2\\q_2 \mapsto q}}
\\
&\overset{(*)}{=} \left[ \frac{1}{1-t} 
\left( \frac{1+tq_2}{1-tq_1} \right)^n
\right]_{\substack{q_1\mapsto q^2\\q_2 \mapsto q}}
=\frac{1}{1-t} \left( \frac{1+tq}{1-tq^2} \right)^n. \qedhere
\end{align*}
\end{proof}

To enhance Proposition~\ref{prop:cross-polytope} to a $\symmB_n$-equivariant calculation,
one must recall the parametrization of
simple $\RR[\symmB_n]$-modules, e.g., from Geissinger and Kinch \cite{GeissingerKinch}, Macdonald \cite[Ch.~1,~App. B]{Macdonald}.
The simple $\RR[\symmB_n]$-modules
$\{ \specht^{(\lambda^+,\lambda^-)} \}_{(\lambda^+,\lambda^-)}$
are indexed by {\it ordered pairs of partitions} 
$$
(\lambda^+,\lambda^-)
\text{ where } 
|\lambda^+|=n_+, \,\, |\lambda|=n_-
\text{ with }n_+ + n_-=n.
$$
One can construct $\specht^{(\lambda^+,\lambda^-)}$ using the simple $\RR[\symm_n]$-modules $\{ \specht^\mu\}_{\mu \vdash n}$ as follows.
Introduce
\begin{itemize}
\item the operation of {\it induction} 
$$
(U_+,U_-) \longmapsto \left( U_+ \otimes U_-\right) \uparrow_{\symmB_{n_+} \times\symmB_{n_-}}^{\symmB_n},
$$
for any pairs $U_+,U_-$ of $\RR[\symmB_{n_+}], \RR[\symmB_{n_-}]$-modules,
\item the operation of {\it inflation} 
$
V \longmapsto V\Uparrow
$
of $\RR[\symm_n]$-modules
$V$ to $\RR[\symmB_n]$-modules, by precomposing with the group quotient map $\pi: \symmB_n \longrightarrow \symm_n$ that ignores the $\pm$ signs, and
\item the one-dimensional character
$\chi_\pm: \symmB_n \rightarrow \{\pm 1\}$
sending a signed permutation matrix $w$ to the product of its $\pm 1$ signs, that is, $\chi_\pm(w):=\det(w)/\det(\pi(w))$.
\end{itemize}
Then starting with the simple $\RR[\symm_n]$-modules
$\{\specht^\mu\}$, one builds 
the $\RR[\symmB_n]$-module $\specht^{(\lambda^+,\lambda^-)}$ as follows:
$$
\specht^{(\lambda^+,\lambda^-)}
:= \left(\quad
\specht^{\lambda^+}\Uparrow
\qquad \otimes \qquad
\chi_\pm  \otimes  ( \specht^{\lambda^-} \Uparrow )
\quad \right)
\big\uparrow_{\symmB_{n_+} \times \symmB_{n_-}}^{\symmB_n}.
$$

\begin{prop}
\label{prop:cross-polytope-equivariantly}
For the cross-polytope $P:=\Cross^n \subset \RR^n$, one has
in $\Cl_\RR[\symmB_n][q][[t]]$
$$
\Eseries_P^{\symmB_n}(t,q)
=\frac{1}{(1-t) \cdot \prod_{i=1}^n (1-t^i q^{2i})}
\left( \sum_{\lambda \vdash n} f^{\lambda}(tq^2) \cdot [\specht^{(\lambda,\varnothing)}]
\right)
 \left(
\sum_{i=0}^n (tq)^i  \cdot [\specht^{((n-i),(i))}]
\right)
$$
\end{prop}
\begin{proof}
The group $\symmB_n$ acts on $S=\RR[\xx]$
by permuting and negating the variables $x_1,\ldots,x_n$.
As the isomorphism \eqref{eq:hyperoctahedral-truncated-iso} is also an isomorphism of graded $\RR[\symmB_n]$-modules, the key is to calculate the class 
$$
\left[ \RR[\xx^2] \otimes_\RR U\right]_{q_1,q_2}
=
\left[ \RR[\xx^2] \right]_{q_1} 
\cdot 
\left[ U \right]_{q_2}
$$
within the representation ring $\Cl_\RR[\symmB_n][[q_1,q_2]]$
of $\NN^2$-graded $\RR[\symmB_n]$-modules.

For the right tensor factor $U=\spn_\RR \{\xx_B: B \subseteq [n]\} \subset S$, we claim that its $i^{th}$-graded component $U_i$ carries the simple $\RR[\symmB_n]$-module
$\specht^{((n-i),(i))}$.  This is because it is a direct sum of the $\binom{n}{i}$ lines which are the $\symmB_n$-images of the line $L:=\RR\cdot x_1 x_2 \cdots x_i$.
This line $L$ is stabilized setwise by the subgroup $\symmB_{n-i} \times \symmB_i$,
where the $\symmB_{n-i}$ factor acts trivially
and the $\symmB_i$ factor acts via $\chi_\pm$.
Hence 
$$
[U]_{q_2} =\sum_{i=0}^n q_2^i \cdot [\specht^{((n-i),(i))}].
$$

To analyze $\RR[\xx^2]$ as a $\RR\symm[B_n]$-module,
note that since all variables $x_i$ appear
squared as $x_i^2$, the sign changes in $B_n$ have no effect;  $B_n$ acts via inflation through the surjection $\pi:B_n \rightarrow \symm_n$.  Consequently, one can obtain
the class of the graded $\RR[\symmB_n]$-module $\RR[\xx^2]$ from that of the graded $\RR[\symm_n]$-module $S=\RR[\xx]$ given in Theorem~\ref{thm:Lusztig-Stanley}, simply by applying the inflation
map 
$$
[\specht^\lambda] \mapsto 
[\specht^\lambda \Uparrow]
=[\specht^{(\lambda,\varnothing)}].
$$
The upshot is that
$$
\left[ \RR[\xx^2] \otimes_\RR U\right]_{q_1,q_2}
=
\left[ \RR[\xx^2] \right]_{q_1} 
\cdot 
\left[ U \right]_{q_2}
=\left( 
\frac{
\sum_{\lambda \vdash n} f^{\lambda}(q_1) \cdot [\specht^{(\lambda,\varnothing)}]}
{\prod_{i=1}^n (1-q_1^i)}
\right)
\cdot \left(
\sum_{i=0}^n q_2^i  \cdot [\specht^{((n-i),(i))}]
\right).
$$
Finally, as in the proof of Proposition~\ref{prop:cross-polytope}, $\Eseries^{\symmB_n}_P(t,q)$ is obtained from $\left[ \RR[\xx^2] \otimes_\RR U\right]_{q_1,q_2}$ upon multiplying by $\frac{1}{1-t}$ and replacing  $q_1 \mapsto tq^2$ and $q_2 \mapsto tq$.
\end{proof}

\section{Minkowski closure and the proof of Theorem~\ref{thm:minkowski-closure}}
\label{sec:Minkowski}

Recall from the Introduction that we hope to approach Conjecture~\ref{conj:intro-omnibus} through the new {\it harmonic algebra} $\HHH_P$ of a lattice polytope defined in Section~\ref{sec:Harmonic} below.  The existence of
this algebra structure relies on 
Theorem~\ref{thm:minkowski-closure}, asserting that for finite point loci
$\Zpoints, \Zpoints' \subset \kk^n$ (such as $\Zpoints=\ZZ^n \cap mP$ and $\Zpoints'=\ZZ^n \cap m'P \subset \RR^n$),
their harmonic spaces always satisfy
$$
V_\Zpoints \cdot V_{\Zpoints'} \subseteq V_{\Zpoints+\Zpoints'}.
$$
The goal of this section is to prove this result.  The first two subsections collect some preparatory material.  This includes broadening the notion of the perp $I^\perp$
or Macaulay inverse system to allow not only homogeneous ideals $I \subset S=\kk[\xx]$, but arbitrary ideals.

\subsection{Vanishing ideals for finite point loci}
For $\Zpoints \subset \kk^n$ with $\kk$ a field, the {\it vanishing ideal} is 
$$
I(\Zpoints):=\{ f(\xx) \in \kk[\xx]: f(\zz)=0 \text{ for all }\zz \in \Zpoints\}.
$$
It is convenient to have a concrete generating set for $I(\Zpoints)$ for $\Zpoints$ finite, even if it has redundant generators.  The following proposition is well-known, but perhaps hard to find in the literature.

\begin{lemma}
\label{lem:generating-vanishing-ideals}
For any field $\kk$ and finite subset $\Zpoints \subset \kk^n$, its vanishing ideal has these descriptions:
\begin{align}
\label{eq:I(Z)-is-an-intersection}
I(\Zpoints) &= \bigcap_{\zz \in \Zpoints}
(x_1-z_1,\ldots,x_n-z_n)\\
\label{eq:I(Z)-is-a-product}
&= \prod_{\zz \in \Zpoints}
(x_1-z_1,\ldots,x_n-z_n)\\
\label{eq:generators-of-a-product-ideal}
&= \left( \prod_{\zz \in \Zpoints} (x_{p(\zz)}-z_{p(\zz)}) : 
\text{ all functions }p:\Zpoints \rightarrow [n] \right)
\end{align}
\end{lemma}
\begin{proof}
When $\#\Zpoints=1$ so that $\Zpoints=\{\zz\}$, the assertions all
hold, and $I(\{\zz\})$ is the maximal ideal $(x_1-z_1,\ldots,x_n-z_n)$.  From this,
\eqref{eq:I(Z)-is-an-intersection} follows by definition.

Equality \eqref{eq:I(Z)-is-a-product} then follows by induction on $\#\Zpoints$, once we check that the two ideals $I=I(\{\zz\})$ and 
$I'=I(\Zpoints \setminus \{\zz\})$ are coprime in the sense that $1 \in I+I'$: this implies the inclusion $I \cap I'\supseteq I \cdot I'$ becomes an equality.  Coprimeness of $I,I'$ follows from {\it multivariate Lagrange interpolation} (see e.g. \cite[top of p. 13]{huang2018ordered}) which provides the existence of a polynomial $f(\xx)$ that 
has $f(\zz)=1$ and $f(\zz')=0$ for all $\zz' \in \Zpoints \setminus \{\zz\}$, so $f(\xx) \in I'$:  in the quotient ring $\kk[\xx]/I \cong \kk$,  
one has $1 \equiv f(\xx) \bmod I$, and
hence $1 \in I+I'$.

The equality \eqref{eq:generators-of-a-product-ideal} also follows by induction on $\#\Zpoints$, since the product $I\cdot I'$ of two ideals generated as $I=(f_p)_{p=1,2,\ldots}, I'=(f'_q)_{q=1,2,\ldots}$ can be generated as $I \cdot I'=(f_p f'_q)_{p,q=1,2,\ldots}$.
\end{proof}

\subsection{Harmonic spaces for inhomogeneous ideals: completions and exponentials}
\label{sec:completions-and-exponentials}

It will be much easier to identify the harmonic spaces $\II(\Zpoints)^\perp$ of the {\it inhomogeneous} ideals $\II(\Zpoints) \subset S=\RR[\xx]$, rather than their homogeneous deformations $\gr \, \II(\Zpoints)$. However, the harmonic spaces $I^\perp$ for inhomogeneous ideals $I \subset \kk[\xx]$ naturally live in the power series
{\it completion} $\RR[[\yy]]$ of $\RR[\yy]$, or more generally, a completion $\hat{\Div}_\kk(\yy)$ of the divided power algebra $\Div_\kk(\yy)$.  Within these completions, the
harmonic spaces $\II(\Zpoints)^\perp$ will turn out to have a simple basis of
``exponentials" indexed by $\zz \in \Zpoints$; see Lemma~\ref{lem:inhomogeneous-exponential-divided} below.

Recall the set-up from Sections~\ref{sec:harmonics-char-zero}, \ref{sec:harmonics-over-all-fields} for harmonic spaces of homogeneous ideals in $S=\kk[\xx]$.
For $\kk$ a characteristic zero field, consider the {\it polynomial algebra} 
$$
\Div:=\kk[\yy]=\kk[y_1,\ldots,y_n]
\text{ with }\kk\text{-basis }\{\yy^\aa: \aa\in \NN^n\}.
$$
More generally, over any field $\kk$ or any commutative ring with $1$, consider
the divided power algebra 
$$
\Div:=\Div_\kk(\yy) 
\text{ with }\kk\text{-basis }
\{\yy^{(\aa)}=y_1^{(a_1)} \cdots y_n^{(a_n)}
: \aa\in \NN^n\}, \,\, \text{ where }
\yy^{(\aa)} \yy^{(\bb)}=\prod_{i=1}^n \binom{a_i+b_i}{a_i}\yy^{(\aa+\bb)}.
$$
making the identification $y_1^{(a_1)} \cdots y_n^{(a_n)}:=\frac{\yy^\aa}{a_1! \cdots a_n!}$
whenever $\kk$ has characteristic zero.
One has an $S$-module structure on $\Div$ given by 
$
\odot: S \times \Div \rightarrow \Div
$
where $x_i$ acts on $\Div$ as the derivation $\frac{\partial}{\partial y_i}$ (so that $x_i \odot y_j=\delta_{ij}$), with these formulas in the characteristic zero and arbitrary ring cases:
\begin{align*}
\xx^\aa \odot \yy^\bb
&=\begin{cases}
\prod_{i=1}^n \frac{b_i!}{(b_i-a_i)!}
\cdot \yy^{\bb-\aa} &\text{ if }a_i \leq b_i \text{ for }i=1,\ldots,n,\\
0& \text{ otherwise.}
\end{cases} \\
\xx^\aa \odot \yy^{(\bb)}
&=\begin{cases}
 \yy^{(\bb-\aa)} &\text{ if }a_i \leq b_i \text{ for }i=1,\ldots,n,\\
0& \text{ otherwise.}
\end{cases}
\end{align*}
This leads to a $\kk$-bilinear pairing
$\langle -, - \rangle: S \times \Div \rightarrow \kk$ defined by
\begin{equation}
    \label{eq:background-pairing}
    \langle f(\xx), g(\yy) \rangle :=
    \text{the constant term of $f \odot g$}.
\end{equation}
This $\langle -,-\rangle$ 
pairs the $\kk$-bases
$
\{ \xx^\aa \}_{\aa \in \NN^n}
\text{ and }
\left\{ \yy^{(\aa)} \right\}_{\aa \in \NN^n}
$
as $\langle \xx^\aa, \yy^{(\bb)} \rangle=\delta_{\aa,\bb}$, giving a {\it perfect pairing}
in each degree, identifying
$\Div_d \cong S_d^*$ and $\Div_{\leq d} \cong S_{\leq d}^*$,
for all $d \geq 0$.
In other words, $S,\Div$ are what are sometimes called {\it graded} (or {\it restricted)} $\kk$-duals.

When working with {\it inhomogeneous} ideals $I \subseteq S=\kk[\xx]$, we will
include $\Div$ in a larger ring $\hat{\Div}=\varprojlim \Div_{\leq m}$ which is the inverse
limit of the projections $\Div/\Div_{\geq m} \twoheadrightarrow \Div/\Div_{\geq m-1}$.
In other words, 
$\hat{\Div}=\{ \sum_{\aa \in \NN^n} c_\aa \yy^{(\aa)} : c_\aa \in \kk\}$
with multiplication defined $\kk$-linearly extending
the rule in \eqref{eq:divided-power-multiplication}.
We can therefore extend the $\odot$ pairing and $S$-module structure to $\hat{\Div}$ to
$
\odot: S \times \hat{\Div} \rightarrow \hat{\Div}
$
via
$
f \odot g := \partial(f)(g),
$
and extend the $\kk$-bilinear pairing 
\begin{equation}
\label{eq:completed-bilinear-pairing}
\langle -, - \rangle: S \times \hat{\Div} \rightarrow \kk
\end{equation}
via the same formula in which
$\langle f, g \rangle$ is the  constant term of $f \odot g$.
This pairing $\langle -,-\rangle$ identifies $\hat{\Div}$ isomorphically 
with the (unrestricted) $\kk$-linear dual space: 
$$
\begin{array}{rcl}
S^*:= \Hom_\kk(S,\kk)&\longrightarrow &\hat{\Div}\\
\varphi & \longmapsto & 
\displaystyle \sum_{\aa \in \NN^n} \varphi(\xx^\aa) \cdot  y^{(\aa)}.
\end{array}.
$$

 As a consequence, for any subspace $W \subset S$, its {\it annihilator or perp space}
$$
W^\perp:=\{g(\yy) \in \hat{\Div}: \langle f(\xx),g(\yy)\rangle=0
\text{ for all }f(\xx) \in W\}
$$
has a $\kk$-linear isomorphism
$$
\begin{array}{rcl}
W^\perp &\longrightarrow & (S/W)^* :=\Hom_\kk(S/W,\kk)\\
g(\yy) &\longmapsto &\langle -,g(\yy) \rangle
\end{array}
$$
In particular, when $\kk$ is a field and $S/W$ is finite-dimensional, then so is $W^\perp$, since $W^\perp \cong (S/W)^*$.

\begin{definition} \rm
For any field $\kk$, and for any
(possibly inhomogeneous) ideal $I \subseteq S=\kk[\xx]$, define its {\it harmonic space} as $I^\perp \subseteq \hat{\Div}=\hat{\Div}_\kk(\yy) (=\kk[[\yy]]\text{ if }\kk \supseteq \QQ)$.
\end{definition}

The preceding discussion shows that, for $\kk$ a field and whenever $\dim_\kk S/I$ is finite, one has
\begin{equation}
\label{eq:inhomogeneous-harmonics-still-have-correct-dimension}
\dim_\kk I^\perp = \dim_\kk (S/I)^* = \dim_\kk S/I.
\end{equation}
The associated graded ideal $\gr \, I \subseteq S$ has harmonic space $(\gr \, I)^\perp \subseteq \Div$. Before turning our attention to ideals coming from point loci, we give a general relationship between the harmonic spaces $I^\perp$ and $(\gr \, I)^\perp$ whenever $S/I$ is Artinian. 

Given a nonzero element $f \in \hat{\Div}$, let $\beta(f) \in \Div$ be the bottom degree homogeneous component of $f$. That is, if $f = \sum_{i \geq d} f_i$ with $f_i$ homogeneous of degree $i$ and $f_d \neq 0$, we have $\beta(f) = f_d$.  Also define $\beta(0) = 0$. The map $\beta: \hat{\Div} \to \Div$ is not $\kk$-linear; if $f_1 = 1+y$ and $f_2 = -1$ one has $\beta(f_1 + f_2) = y$ whereas $\beta(f_1) + \beta(f_2) = 1 - 1 = 0$. 

\begin{prop}
    \label{prop:beta-projection}
    Let $W \subseteq \hat{\Div}$ be any $\kk$-linear subspace of $\hat{\Div}$. The subspace of $\Div$ generated by its image $\beta(W)$ under $\beta$ is a $\kk$-linear subspace of $\Div$ of the same dimension as $W$.
\end{prop}
\begin{proof}
Let $\hat{\Div}_{\geq d} \subseteq \hat{\Div}$ be the $\kk$-subspace consisting of series of the form $g_d + g_{d+1} + \cdots$ where $g_i$ is homogeneous of degree $i$. We have a descending filtration on $\Div$
$$
\hat{\Div} = \hat{\Div}_{\geq 0} \supset
\hat{\Div}_{\geq 1} \supset 
\hat{\Div}_{\geq 2} \supset \cdots
$$
with $\bigcap_{d=0}^\infty \hat{\Div}_{\geq d}=\{0\}$. This induces a similar filtration
$$
W=W_{\geq 0} \supset 
W_{\geq 1} \supset 
W_{\geq 2} \supset \cdots
$$
in which $W_{\geq d}:=W \cap \hat{\Div}_{\geq d}$ where  $\bigcap_{d=0}^\infty W_d=\{0\}$. One therefore has
\begin{equation}
    \dim_\kk W = \sum_{d \, = \, 0}^{\infty} \dim_\kk(W_{\geq d}/W_{\geq d+1}).
\end{equation}

Let $\Div_{\geq d} := \hat{\Div}_{\geq d} \cap \Div$ and consider the composite map $\beta_{\geq d}: \hat{\Div}_{\geq d} \to \Div_d$ 
\begin{equation}
    \beta_{\geq d}: \hat{\Div}_{\geq d} \xrightarrow{ \, \beta \, } \Div_{\geq d} \twoheadrightarrow \Div_{\geq d}/\Div_{\geq d+1} = \Div_d
\end{equation}
where the map $\Div_{\geq d} \twoheadrightarrow \Div_{\geq d}/\Div_{\geq d+1}$ is the canonical projection. Despite the fact that $\beta: \hat{\Div} \to \Div$ is not $\kk$-linear, it is not hard to see that $\beta_{\geq d}: \hat{\Div}_{\geq d} \to \Div_d$ is $\kk$-linear and satisfies $\beta_{\geq d}(\hat{\Div}_{\geq d}) = 0$.

Write $U \subseteq \Div$ for the subspace generated by $\beta(W)$. Then $U$ is graded, and its degree $d$ graded piece $U_d$ equals $\beta_{\geq d}(W_{\geq d})$. Since $\beta_{\geq d}(\hat{\Div}_{\geq d+1}) = 0$, the linear map $\beta_{\geq d}: W_{\geq d} \to \Div_d$ induces a linear map $W_{\geq d} / W_{\geq d+1} \to \Div_d$, and it is not hard to see that this latter map is bijective. Consequently, we have a linear isomorphism $W_{\geq d} / W_{\geq d+1} \cong U_d$ and
\begin{equation}
    \dim_\kk W = \sum_{d \, = \, 0}^{\infty} \dim_\kk(W_{\geq d}/W_{\geq d+1}) = \sum_{d \, = \, 0}^{\infty} \dim_\kk U_d = \dim_\kk U.\qedhere
\end{equation}
\end{proof}

Given $g = g_d + g_{d+1} + \cdots \in \hat{\Div}$ with $g_i$ homogeneous of degree $i$ and $g_d \neq 0$, define the {\em valuation} $\nu(g) := d$. Also define $\nu(0) := 0$. The relationship between the set map $\beta$ and harmonic spaces may be stated as follows.

\begin{cor}
    \label{cor:beta-perp-identification}
    Let $I \subseteq S$ be any ideal (not necessarily homogeneous) such that $S/I$ is a finite-dimensional $\kk$-vector space, so that $I^\perp \subseteq \hat{\Div}$. Then $(\gr \, I)^\perp \subseteq \Div$ is the $\kk$-linear subspace of $\Div$ spanned by $\beta(I^\perp)$.
\end{cor}

\begin{proof}
    Let $U \subseteq \Div$ be the subspace spanned by $\beta(I^\perp)$.
    We have $$\dim_\kk S/I = \dim_\kk S / \gr \, I = \dim_\kk (\gr \, I)^\perp.$$
    Equation~\eqref{eq:inhomogeneous-harmonics-still-have-correct-dimension} and Proposition~\ref{prop:beta-projection} guarantee that $U$ shares this common dimension. It is therefore enough to establish $\beta(I^\perp) \subseteq (\gr \, I)^\perp$. Indeed, let $f \in I$ and $g \in I^\perp$. We need to show that $\tau(f) \odot \beta(g) = 0$. However, the element $\tau(f) \odot \beta(g) \in \Div$ is the degree $\nu(g) - \deg(f)$ homogeneous component of $f \odot g \in \hat{\Div}$. Since $f \odot g = 0$, we conclude that  $\tau(f) \odot \beta(g) = 0$.
\end{proof}

We apply Corollary~\ref{cor:beta-perp-identification} to ideals of the form $\II(\Zpoints)$ for finite loci $\Zpoints \subseteq \kk^n$. The harmonic spaces $\II(\Zpoints)^\perp \subseteq \hat{\Div}$ of these ideals have simple $\kk$-bases.

\begin{definition} \rm
For any field $\kk$, given $\zz \in \kk^n$, define the {\it exponential} in the completion $\hat{\Div}$
$$
\exp(\zz \cdot \yy)
:=\sum_{d=0}^\infty (z_1 y_1 + \cdots z_n y_n)^{(d)}
=\sum_{d=0}^\infty \quad
\sum_{\substack{\dd=(d_1,\ldots,d_n):\\\sum_i d_i=d}} \zz^\dd \yy^{(\dd)},
$$
where the right equality above used
\eqref{eq:beginner-multinomial-thm}.  
It can be rewritten more
familiarly when $\kk \supseteq \QQ$ as
$$
\exp(\zz \cdot \yy)
:=\sum_{d=0}^\infty \frac{(z_1 y_1 + \cdots z_n y_n)^d}{d!}\\
=\sum_{d=0}^\infty \quad 
\sum_{\substack{\dd=(d_1,\ldots,d_n):\\\sum_i d_i=d}} \zz^\dd 
\frac{\yy^\dd}{d_1! \cdots d_n!}.
$$
\end{definition}
It is not hard to check that one has
\begin{equation}
\label{eq:factoring-exponentials}
\exp(\zz \cdot \yy)=\exp( z_1 \cdot y^{(1)}_1 ) \cdots \exp( z_n \cdot y^{(1)}_n )
\end{equation}
where for $z \in \kk$ and $y \in \hat{\Div}_1$, one has
$$
\exp( z \cdot y^{(1)} ):=\sum_{d=0}^\infty z^d y^{(d)}
\left( =\exp( z \cdot y ) = \sum_{d=0}^\infty z^d \frac{y^d}{d!} \text{ if }\kk \supseteq \QQ.
\right).
$$

\begin{lemma}
    \label{lem:inhomogeneous-exponential-divided}
    For $\kk$ a field and $\Zpoints \subset \kk^n$ a finite subset,
    the harmonic space
    $\II(\Zpoints)^\perp \subseteq \hat{\Div}$ has $\kk$-basis 
    $$
    \{ \exp(\zz \cdot \yy) \,:\, \zz \in \Zpoints \}.
    $$
\end{lemma}
\begin{proof}
Note that \eqref{eq:inhomogeneous-harmonics-still-have-correct-dimension} shows that the set in the lemma has the correct size 
$$
\#\Zpoints =\dim_\kk S/\II(\Zpoints)=\dim_\kk \II(\Zpoints)^\perp.
$$
Hence it suffices to check that its elements lie in $\II(\Zpoints)^\perp$, and are $\kk$-linearly independent.

To check that $\exp(\zz \cdot \yy) \in \II(\Zpoints)^\perp$ 
for each $\zz=(z_1,\ldots,z_n) \in \Zpoints$,  
by Lemma~\ref{lem:generating-vanishing-ideals}
it suffices to check that for any $p: \Zpoints \rightarrow [n]$
one has 
$f_p(\xx) \odot \exp(\zz \cdot \yy)=0$ 
where 
$$
f_p(\xx)=\prod_{\zz' \, \in \, \Zpoints} (x_{p(\zz')} - \zz'_{p(\zz')})\\
=\left( \prod_{\zz' \, \in \, \Zpoints \setminus \{\zz\}} 
(x_{p(\zz')} - z'_{p(\zz)}) \right) \cdot (x_{p(\zz)} - z_{p(\zz)}).
$$
Thus it suffices to check 
$(x_{p(\zz)} - z_{p(\zz)}) \odot \exp(\zz \cdot \yy) = 0$.  Letting $i:=p(\zz)$, one has
\begin{align*}
(x_{p(\zz)} - z_{p(\zz)}) \odot \exp(\zz \cdot \yy) 
&=(x_i-z_i) \odot \left(  \exp(z_1 y^{(1)}_1) \cdots \exp(z_n y^{(1)}_n)\right) \\
&=\left(  \exp(z_1 y^{(1)}_1) \cdots \widehat{\exp(z_i y^{(1)}_i)} \cdots \exp(z_n y^{(1)}_n)\right) \cdot
(x_i-z_i) \odot \exp(z_i y^{(i)}_1).
\end{align*}
Thus in the end it suffices to check 
$(x_i-z) \odot \exp(z y^{(1)}_i)
=0$ for $z \in \kk$, which is not hard:
$$
x_i \odot \exp(z y^{(1)}_i)
=\sum_{d=0}^\infty x_i \odot \left( z^d y^{(d)}_i \right)
=\sum_{d=1}^\infty z^d y^{(d-1)}_i
=z \sum_{d=1}^\infty z^{d-1} y^{(d-1)}_i
=z \exp(z y^{(1)}_i).
$$

We next check that 
$\{ \exp(\zz \cdot \yy) : \zz \in \Zpoints \}$ 
are $\kk$-linearly independent within $\hat{\Div}$.
We first deal with the $n=1$ case.
When $n = 1$, if we let $N:=\#\Zpoints$, so that $\Zpoints = \{z_1, \ldots,z_N\} \subset \kk$, 
we note that inside $\hat{\Div}=\hat{\Div}_\kk(y)$, one has
\begin{align*}
    \exp(z_1 \cdot y)&=1 + z_1 y^{(1)} +  z_1^2 y^{(2)} + \cdots +  z_1^{N-1} y^{(N-1)} + \cdots \\
    \exp(z_2 \cdot y)&=1 + z_2 y^{(1)} +  z_2^2 y^{(2)} + \cdots +  z_2^{N-1} y^{(N-1)} + \cdots \\
    \vdots \qquad & \\
     \exp(z_N \cdot y)&=1 + z_N y^{(1)} +  z_N^2 y^{(2)} + \cdots +  z_N^{N-1} y^{(N-1)} + \cdots 
\end{align*}
These exponentials are $\kk$-linearly independent because $\{1,y^{(1)},y^{(2)},\ldots,y^{(N-1)}\}$ are
$\kk$-linearly independent and the Vandermonde matrix $(z_i^{j-1})_{i,j=1,2,\ldots,N}$ is invertible. 

When $n \geq 2$, it is helpful to note that
the factorization \eqref{eq:factoring-exponentials}
lets one identify 
 $\exp(\zz \cdot \xx)$ with the element 
 $\exp( z_1 y^{(1)}_1 ) \otimes \cdots \otimes \exp( z_n y^{(1)}_n )$ lying within the proper\footnote{Even in characteristic zero, one has $\kk[[y_1]] \otimes \cdots \otimes \kk[[y_n]] \subsetneq \kk[\yy]]$, without forming a {\it completed tensor product}.} subspace
$$
\hat{\Div}_\kk(y_1) \otimes \cdots \otimes \hat{\Div}_\kk(y_n) \subsetneq \hat{\Div}_{\kk}(\yy).
$$
Letting $p_i: \kk^n \twoheadrightarrow \kk$ for $i=1,2,\ldots,n$ denote
the coordinate projections, the $n=1$ case proven
above shows each subset 
$\{ \exp( z_i y^{(1)}_i) \}_{z \in p_i(\Zpoints)}
\subset \hat{\Div}_\kk(y_i)$ is $\kk$-linearly independent.  Therefore the set
\begin{equation}
\label{eq:full-grid-set}
\{ \exp(z_1 y^{(1)}_1) \otimes \cdots \otimes 
\exp( z_n y^{(1)}_n) : \zz=(z_1,\ldots,z_n) \in p_1(\Zpoints)\times \cdots \times p_n(\Zpoints) \}
\end{equation}
is $\kk$-linearly independent inside $\hat{\Div}_\kk(y_1) \otimes \cdots \otimes \hat{\Div}_\kk(y_n)$.  Since the set in \eqref{eq:full-grid-set} contains
$$
\{ \exp(z_1 y^{(1)}_1) \otimes \cdots \otimes 
\exp( z_n y^{(1)}_n): \zz=(z_1,\ldots,z_n) \in \Zpoints\}
$$
as a subset, the latter must therefore
also be $\kk$-linearly independent subset, as desired.
\end{proof}

Recall the familiar formula
\begin{equation}
\label{eq:exponential-law}
    \exp(\zz \cdot \yy) \cdot \exp(\zz' \cdot \yy) = \exp((\zz + \zz') \cdot \yy)
\end{equation}
which holds in the power series ring $\hat{\Div}=\kk[[\yy]]$;  it is an easy exercise using \eqref{eq:beginner-binomial-thm} to check that it remains valid in $\hat{\Div}=\hat{\Div}_\kk(\yy)$
for any field $\kk$.  Combining \eqref{eq:exponential-law}
 with Lemma~\ref{lem:inhomogeneous-exponential-divided} immediately gives
 the following easier and more precise 
 {\it inhomogeneous ideal} version 
of Theorem~\ref{thm:minkowski-closure}.

\begin{prop}
\label{prop:inhomogeneous-minkowski-addition}
For any field $\kk$ and finite subsets $\Zpoints, \Zpoints' \subseteq \kk^n$, one has inside $\hat{\Div}$ that
\begin{equation}
    \label{eq:minkowski-inhomogeneous-product}
    \II(\Zpoints)^\perp \cdot \II(\Zpoints')^\perp = \II(\Zpoints + \Zpoints')^\perp.
\end{equation}
\end{prop}

We have all of the tools necessary to prove the Minkowski Closure Theorem. Let us recall its statement.

\vskip.1in
\noindent
{\bf Theorem }~\ref{thm:minkowski-closure}
{\it
For any pair of finite point loci $\Zpoints, \Zpoints'$ in $\kk^n$ over any field $\kk$, one has 
$$
V_\Zpoints \cdot V_{\Zpoints'} 
    \subseteq V_{\Zpoints + \Zpoints'}.
$$
}

\begin{proof}
     Recall that $V_\Zpoints = (\gr \, \II(\Zpoints))^\perp$ and similarly for $V_{\Zpoints'}$ and $V_{\Zpoints + \Zpoints'}$. Let $g \in V_\Zpoints$ and $g' \in V_\Zpoints$ be homogeneous elements. We show that $g \cdot g' \in V_{\Zpoints + \Zpoints'}$ as follows.
    
    By Corollary~\ref{cor:beta-perp-identification} and the fact that $g, g'$ are homogeneous, there exist elements $\hat{g} \in \II(\Zpoints)^\perp$ and $\hat{g}' \in \II(\Zpoints')^\perp$ so that $\beta(\hat{g}) = g$ and $\beta(\hat{g}') = g'$. Proposition~\ref{prop:inhomogeneous-minkowski-addition} implies that $\hat{g} \cdot \hat{g}' \in \II(\Zpoints + \Zpoints')^\perp$. We have
    \begin{equation}
        g \cdot g' = \beta(\hat{g}) \cdot \beta(\hat{g}') = \beta(\hat{g} \cdot \hat{g}') \in V_{\Zpoints + \Zpoints'}
    \end{equation}
    where the membership $\beta(\hat{g} \cdot \hat{g}') \in V_{\Zpoints + \Zpoints'}$ follows from Corollary~\ref{cor:beta-perp-identification}.
\end{proof}

We conclude this subsection with a few remarks.
The first are
some interesting examples of Theorem~\ref{thm:minkowski-closure}, including two small examples with point loci in $\RR^2$. Another example discusses the relation
between Theorem~\ref{thm:minkowski-closure} for point loci in $\kk^1$, and the {\it additive combinatorics of sumsets} and the {\it Cauchy-Davenport Theorem}.  In the next subsection, we give another proof of Theorem~\ref{thm:minkowski-closure} which connects to deformation geometry
and remark on a similar-sounding result of
F. Gundlach.

\begin{example}
One can easily have a {\it proper} inclusion $V_\Zpoints \cdot V_{\Zpoints'} \subsetneq V_{\Zpoints+\Zpoints'}$. E.g., inside $\RR^2$ take 
\begin{align*}
\Zpoints & = \{(0,0),(1,0),(0,1)\},\\
\Zpoints' &= \{(0,0),(1,0),(1,1)\},\\
\Zpoints + \Zpoints' &= \{(0,0),(1,0),(0,1),(1,1),(2,0),(2,1),(1,2)\}.
\end{align*}
Although $\Zpoints \neq \Zpoints'$, the two loci $\Zpoints, \Zpoints'$ are equivalent under 
$\Aff(\ZZ^2)$, with 
\begin{align*}
\gr \, \II(\Zpoints) = \gr \, \II(\Zpoints') 
 &= (x_1^2,x_1x_2,x_2^2) \subseteq \RR[x_1,x_2],\\
V_\Zpoints = V_{\Zpoints'} &= \spn_\RR \{1,y_1,y_2\} \subset \Div_\RR(y_1,y_2)=\RR[y_1,y_2],\\
V_\Zpoints \cdot V_{\Zpoints'} &= \spn_\RR \{1,y_1,y_2,y_1^2,y_1 y_2,y_2^2\}.
\end{align*}
On the other hand, $V_{\Zpoints+\Zpoints'}$ has dimension $7=\#(\Zpoints+\Zpoints')$,
and one can compute that
\begin{align*}
\gr \, \II(\Zpoints + \Zpoints') &= (x_1^3,x_2^3,x_1x_2^2),\\
V_{\Zpoints + \Zpoints'} &= \spn_\RR \{1,y_1,y_2,y_1^2,y_1 y_2,y_2^2,\,\, {\color{red} y_1^2 y_2} \}
\end{align*}
so that $V_{\Zpoints + \Zpoints'}$ properly contains the
$6$-dimensional space $V_\Zpoints \cdot V_{\Zpoints'}$ in
this case.
\end{example}

\begin{example}
\label{ex:second-interesting-triangle-discussion}
A somewhat nontrivial instance of Theorem~\ref{thm:minkowski-closure}
occurs when 
$$
\Zpoints=\Zpoints'=\{(0,0),(1,1),(2,1),(1,2)\}=\ZZ^2 \cap P,
$$
where $P$ is the lattice triangle discussed
in Remark~\ref{remark-first-interesting-triangle-discussion}.  Here one
can check that
$$
\Zpoints+\Zpoints'=\{(0,0),(1,1),(2,1),(1,2),(2,2),(3,2),(2,3),(4,2),(3,3),(2,4)\}
=\ZZ^2 \cap 2P,
$$
as shown below:
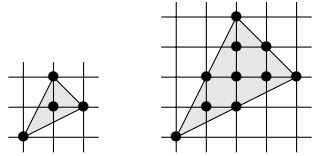
\begin{center}
    \begin{tikzpicture}[scale = 0.4]
        \draw[fill = black!10] (0,0) -- (1,2) -- (2,1) --  (0,0);
        \draw[step=1.0,black,thin] (-0.5,-0.5) grid (2.5,2.5);
        \node at (0,0) {$\bullet$};
        \node at (1,1) {$\bullet$};
        \node at (1,2) {$\bullet$};
        \node at (2,1) {$\bullet$};
    \end{tikzpicture}
    \qquad 
       \begin{tikzpicture}[scale = 0.4]
        \draw[fill = black!10] (0,0) -- (2,4) -- (4,2) --  (0,0);
        \draw[step=1.0,black,thin] (-0.5,-0.5) grid (4.5,4.5);
        \node at (0,0) {$\bullet$};
        \node at (1,1) {$\bullet$};
        \node at (1,2) {$\bullet$};
        \node at (2,1) {$\bullet$};
        \node at (2,2) {$\bullet$};
        \node at (3,2) {$\bullet$};
        \node at (2,3) {$\bullet$};
        \node at (4,2) {$\bullet$};
        \node at (3,3) {$\bullet$};
        \node at (2,4) {$\bullet$};
    \end{tikzpicture}
\end{center}
Remark~\ref{remark-first-interesting-triangle-discussion} already mentioned one can calculate by hand or {\tt Macaulay2} that
$$
V_\Zpoints(=V_{\Zpoints'})=\spn_\RR\{1,\quad y_1,y_2,\quad y_1^2+y_1y_2+y_2^2\}.
$$
with spacing to indicate segregation of $\RR$-basis elements by degree.
Similarly one can calculate that
\begin{multline*}
V_{\Zpoints + \Zpoints'} 
= \spn_\RR  \{1,\quad y_1,y_2,\quad y_1^2, y_1 y_2, y_2^2, \quad 
        y_1^3, y_1^2 y_2+y_1y_2^2, y_2^3, \quad y_1^4 + 2y_1^3 y_2 + 3y_1^2 y_2^2 + 2y_1 y_2^3 + y_2^4 \}.
\end{multline*}
    It can be checked that multiplying any two basis elements of $V_\Zpoints$ and $V_{\Zpoints'}$ results in an element of $V_{\Zpoints + \Zpoints'}$.  In this example, one has 
    equality $V_{\Zpoints} \cdot V_{\Zpoints'} = V_{\Zpoints + \Zpoints'}$.
\end{example}

\begin{example}
Let us see what Theorem~\ref{thm:minkowski-closure} says in a $1$-dimensional space $\kk^1$,
about the cardinalities $r=\#\Zpoints,r'=\#\Zpoints'$ of two finite loci $\Zpoints, \Zpoints'$, versus the cardinality $r''$ of their {\it sumset} $\Zpoints+\Zpoints'$.  As in the Example from the Introduction, one can easily check for finite $\Zpoints \subset \kk^1$ that one
has
\begin{align*}
    \II(\Zpoints)&=\left( \prod_{z \in \Zpoints} (x-z) \right) \subset S=\kk[x], \quad
    \gr\,\II(\Zpoints)= ( x^r ) \subset S=\kk[x],\\
    V_\Zpoints&=\spn_\kk \{1,y^{(1)},y^{(2)},\ldots,y^{(r-1)}\} \subset \Div=\Div_\kk(y).
\end{align*}
From this one can calculate
\begin{align*}
V_\Zpoints \cdot V_{\Zpoints'}
&= \spn_\kk \{1,y^{(1)},y^{(2)},\ldots,y^{(r-1)}\} \cdot 
\spn_\kk \{1,y^{(1)},y^{(2)},\ldots,y^{(r'-1)}\} \\
&= \spn_\kk \left\{ y^{(k)} \cdot y^{(k')} = \binom{k+k'}{k,k'} y^{(k+k')}  : 0 \leq k \leq r-1 \text{ and }0 \leq k' \leq r'-1 \right\}.
\end{align*}
Consequently, the assertion $V_\Zpoints \cdot V_{\Zpoints'} \subseteq V_{\Zpoints+\Zpoints'}$ from Theorem~\ref{thm:minkowski-closure} holds if and only if
\begin{align*}
r''-1 
 &\geq \max\left\{k+k': \exists \,\, 0 \leq k\leq r-1 \text{ and } 0 \leq k' \leq r'-1\text{ with }\binom{k+k'}{k,k'} \in \kk^\times \right\}, \text{ or equivalently,}\\
r''
 &\geq \underbrace{ \min\left\{n: (u+v)^n=\sum_{k+k'=n} \binom{n}{k,k'} u^k v^{k'} 
\text{ lies in the ideal } (u^r,v^{r'}) \subset \kk[u,v] \right\}}_{\text{ call this function }\beta_\kk(r,r')}.
\end{align*}
In other words, Theorem~\ref{thm:minkowski-closure} says that for $\Zpoints, \Zpoints' \subset \kk^1$, one has 
\begin{equation}
\label{eq:Cauchy-Davenport-like-bound}
\#(\Zpoints+\Zpoints')
\geq \beta_\kk( \#\Zpoints, \#\Zpoints').
\end{equation}
When $\kk$ has characteristic zero, one has $\beta_\kk(r,r')=r+r'-1$, and the sumset lower bound \eqref{eq:Cauchy-Davenport-like-bound} is not hard to show directly;  it is also easily seen to be sharp. When $\kk=\FF_{p^d}$ is a finite field, this lower bound is a result of Eliahou and Kervaire \cite[Thm. 2.1]{eliahou1998sumsets}, generalizing both the case $d=1$ for $\kk=\FF_p$ known as the {\it Cauchy-Davenport Theorem}, as well as the case $p=2$ for $\kk=\FF_{2^d}$ due to work of Yuzvinsky \cite{Yuzvinsky} on quadratic forms.
Eliahou and Kervaire show \cite[Thm. 2.2]{EliahouKervaire} that this lower bound \eqref{eq:Cauchy-Davenport-like-bound} is also sharp for $\kk=\FF_{p^d}$. Interestingly, their proof method for~\eqref{eq:Cauchy-Davenport-like-bound} uses associated graded ideals $\gr\, I$ and is close in spirit to our results.
\end{example}

\subsection{An alternative proof of Theorem~\ref{thm:minkowski-closure}} 
The proof of Theorem~\ref{thm:minkowski-closure} given above was concise, but does not directly relate to the point-orbit geometry of linearly deforming a locus $\Zpoints \subseteq \kk^n$ to the origin. We describe a method for proving Theorem~\ref{thm:minkowski-closure} which makes this geometry apparent by introducing a parameter $\epsilon$ to situate the locus $\Zpoints$ in a flat family over $\AA^1_\kk = \Spec(\kk[\epsilon])$ as follows.

Let $R = \kk[\epsilon]$ be the univariate polynomial ring over $\kk$ and consider the divided power algebra $\Div := \Div_R(\yy)$ with coefficients in $R$ and its completion $\hat{\Div} := \hat{\Div}_R(\yy)$. If $S := R[\xx]$, we have an $S$-module structure $\odot: S \times \hat{\Div} \to \hat{\Div}$ and an $R$-bilinear pairing $\langle -, - \rangle: S \times \hat{\Div} \to R$ as before in which $\langle f(\xx), g(\yy) \rangle$ extracts the constant term of $f \odot g$.
For $\Zpoints \subseteq R^n$, we again define the {\it vanishing ideal} in $S=R[\xx]$, and its {\it annihilator/perp} $R$-submodule of $\hat{\Div}$:
\begin{align*}
    \II_R(\Zpoints) &:= \{ f(\xx) \in S=R[\xx] \,:\, f(\zz) = 0 \text{ for all $\zz \in \Zpoints$} \} \subseteq R[\xx],\\
   \II_R(\Zpoints)^\perp &:= \{ g \in \hat{\Div} =\hat{\Div}_R(\yy) \,:\, \langle f, g \rangle = 0 \text{ for all $f \in \II_R(\Zpoints)$} \} \subseteq \hat{\Div}, 
\end{align*}
Here we emphasize the dependence on the ring $R$ via the subscript in the notation. 

Now consider for $\kk$ a field and finite locus $\Zpoints \subseteq \kk^n$, the ``rescaled locus" $\epsilon \Zpoints \subseteq R^n$ for $R=\kk[\epsilon]$, along with its vanishing ideal and annihilator/perp $R$-module
\begin{align*}
\epsilon \Zpoints := 
\{ \epsilon \zz \,:\, \zz \in \Zpoints \} \subseteq R^n,\quad
\II_R(\epsilon \Zpoints) \subset R[\xx],\quad 
\II_R(\epsilon \Zpoints)^\perp \subset \hat{\Div}_{R}(\yy).
\end{align*}
 Although the ideal $\II_{R}(\epsilon \Zpoints)$ annihilates $\exp(\epsilon \zz \cdot \yy)$ under the $\odot$-action for $\zz \in \Zpoints$, the $R$-module $\II_{R}(\epsilon \Zpoints)^\perp$ contains elements which are not in the $R$-span of $\{ \exp(\epsilon \zz \cdot \yy) \,:\, \zz \in \Zpoints \}$. In contrast with Lemma~\ref{lem:inhomogeneous-exponential-divided}, to obtain the full $R$-module $\II_{R}(\epsilon \Zpoints)^\perp$ we must {\em saturate} with respect to $\epsilon$ as follows.

\begin{lemma}
    \label{lem:divided-harmonic-space}
     The above $R$-submodule $\II_{{R}} (\epsilon \Zpoints)^\perp \subseteq \hat{\Div}_{R}(\yy)$ has this description:
    \[
    \II_{R}(\epsilon \Zpoints)^\perp = \{ f \in \hat{\Div}_{R}(\yy) \,:\, \epsilon^M \cdot f \in \mathrm{span}_{R} \{ \exp(\epsilon \zz \cdot \yy) \,:\, \zz \in \Zpoints \} \text{ for some $M \geq 0$} \}.
    \]
\end{lemma}

\begin{proof} (Sketch)
Naming the right side in the lemma as $W_\Zpoints$,  the inclusion 
$\II_{{R}} (\epsilon \Zpoints)^\perp \supseteq W_\Zpoints$ is not hard, and follows similarly to the proof of Lemma~\ref{lem:inhomogeneous-exponential-divided}.   To show $\II_{{R}} (\epsilon \Zpoints)^\perp \subseteq W_\Zpoints$, one strengthens 
the saturation property of $W_\Zpoints$, to show this 
\begin{quote}
{\bf Claim}: If $0 \neq h(\epsilon) \in R=\kk[\epsilon]$ and
$f \in \hat{\Div}_{R}(\yy)$ have $h(\epsilon) \cdot f \in W_\Zpoints$,
then $f \in W_\Zpoints$. 
    \end{quote}
By replacing $\kk$ with its algebraic closure, it is enough to establish this claim when $\kk$ is algebraically closed. To do this, start with a relation of the form
\begin{equation}
    \label{eq:w-witness-relation}
    \epsilon^M \cdot h(\epsilon) \cdot f = \sum_{\zz \in  \Zpoints} c_\zz(\epsilon) \cdot \exp(\epsilon \zz \cdot \yy)
\end{equation}
for some $c_\zz(\epsilon) \in \kk[\epsilon]$. If $\alpha \in \kk^\times$ is a nonzero root of $h(\epsilon)$, substituting $\epsilon \to \alpha$ and using the $\kk$-linear independence of $\{ \exp(\alpha \zz \cdot \yy) \,:\, \zz \in \Zpoints \}$ we get $c_\zz(\alpha) = 0$ for all $\zz \in \Zpoints$, so we may cancel a factor of $(\epsilon - \alpha)$ from both sides of \eqref{eq:w-witness-relation}. We reduce to the case where $h(\epsilon) = \epsilon^r$ for some $r$ and the claim follows easily.

Given the claim, one extends scalars in $\hat{\Div}_R(\yy)$ to $\hat{\Div}_{\kk(\epsilon)}(\yy)$ where $\kk(\epsilon)$ is the fraction field of $R=\kk[\epsilon]$. Lemma~\ref{lem:inhomogeneous-exponential-divided} applies to give the $\kk(\epsilon)$-basis $\{\exp(\epsilon \zz \cdot \yy) \,:\, \zz \in \Zpoints\}$ of $\II_{\kk(\epsilon)}(\epsilon \Zpoints)^\perp$. The inclusion $\II_R(\epsilon \Zpoints)^\perp \subseteq W_\Zpoints$ (and hence Lemma~\ref{lem:divided-harmonic-space}) is deduced from the containment $\II_R(\epsilon \Zpoints)^\perp \subseteq \II_{\kk(\epsilon)}(\epsilon \Zpoints)^\perp$,
clearing denominators, and applying the claim.
\end{proof}

Note that in $\hat{\Div}_R(\yy)$, one still has
$$
\exp( \epsilon \zz \cdot \yy)\cdot \exp( \epsilon \zz \cdot \yy)=\exp( \epsilon (\zz+\zz') \cdot \yy),
$$
and hence Lemma~\ref{lem:divided-harmonic-space} immediately implies the following analogue of 
 Proposition~\ref{prop:inhomogeneous-minkowski-addition}.

\begin{cor}
     \label{cor:divided-inhomogeneous-minkowski-addition}
     For any two finite loci $\Zpoints, \Zpoints' \subseteq \kk^n$, as $R$-submodules of $\hat{\Div}_{R}(\yy)$ one has 
     \begin{equation}
    \II_{R}(\epsilon \Zpoints)^\perp \cdot \II_{R}(\epsilon \Zpoints')^\perp \subseteq \II_{R}(\epsilon(\Zpoints + \Zpoints')).
\end{equation}
 \end{cor}

We remark that one can produce small examples with $n=1$ showing both that 
the $\epsilon$-saturation in the statement of Lemma~\ref{lem:divided-harmonic-space} is necessary, and in contrast to Proposition~\ref{prop:inhomogeneous-minkowski-addition}, that the inclusion in Corollary~\ref{cor:divided-inhomogeneous-minkowski-addition} can be proper. 

\begin{lemma}
    \label{lem:surjective-map}
    Let $e: \hat{\Div}_{R}(\yy) \rightarrow \hat{\Div}_\kk(\yy)$ be the map which evaluates $\epsilon \to 0$. For any finite locus $\Zpoints \subseteq \kk^n$, the map $e$ restricts to a surjection $e: \II_{R}(\epsilon \Zpoints)^\perp \twoheadrightarrow \gr \, \II_\kk (\Zpoints)^\perp$.
\end{lemma}

\begin{proof} (Sketch)
One checks that $e$ indeed maps 
$\hat{\Div}_{R}(\yy)$ {\it into} $\hat{\Div}_\kk(\yy)$. 
Row reduction over the field $\kk$ shows that the $\kk$-span of 
 $\{ \exp(\epsilon \zz \cdot \yy)\}_{\zz \in \Zpoints}$ 
 contains $\#\Zpoints$ elements $\{f_\zz\}_{\zz \in \Zpoints}$ with the following property: 
\begin{quote}each $f_\zz=\epsilon^{M_\zz} g_\zz$ for some nonnegative power $M_\zz$ of $\epsilon$ times another element $g_\zz$ (so $g_\zz$ also lies in  $\II_{R}(\epsilon \Zpoints)^\perp$
by Lemma~\ref{lem:divided-harmonic-space}) and the 
$e$-images $\{e(g_\zz)\}_{\zz \in \Zpoints}$ 
are $\kk$-linearly independent in $\hat{\Div}_\kk(\yy)$.
\end{quote}
Since the space $\gr \, \II_\kk(\Zpoints)^\perp$ has $\kk$-dimension $\#\Zpoints$, this shows that $e$ surjects.
\end{proof}

Lemma~\ref{lem:surjective-map} is closely related to Corollary~\ref{cor:beta-perp-identification}. Indeed, the images $e(g_\zz)$ appearing in the above proof are nothing but the bottom degree components $\beta(g_\zz)$ of the $g_\zz$.

\begin{proof}[Second proof of Theorem~\ref{thm:minkowski-closure}]
Given $g,g'$ in $V_\Zpoints, V_\Zpoints'$, 
one wishes to show that $g \cdot g'$ lies in 
$V_{\Zpoints+\Zpoints'}$.  Lift them both using Lemma~\ref{lem:surjective-map} to $g_\epsilon, g'_\epsilon$ in  $\II_{R}(\epsilon \Zpoints)^\perp,  \II_{R}(\epsilon \Zpoints')^\perp$ with $e(g_\epsilon)=g, e(g'_\epsilon)=g'$. Then Corollary~\ref{cor:divided-inhomogeneous-minkowski-addition} implies that $g_\epsilon \cdot g'_\epsilon$ lies in
    $\II_{R}(\epsilon (\Zpoints+\Zpoints'))^\perp$.
    Finally,
    $
    g \cdot g'=e(g_\epsilon) \cdot e(g_\epsilon') = e(g_\epsilon \cdot g'_\epsilon)
    $
    lies in $V_{\Zpoints+\Zpoints'}$, 
    due to Lemma~\ref{lem:surjective-map} once more.
    \end{proof}

The geometric interpretation of this second proof is as follows.  As mentioned 
in the Introduction, one can view $S/\II(\Zpoints) \leadsto R(\Zpoints)$ as a flat deformation of the reduced subscheme $\Zpoints \subseteq \kk^n$ deforming linearly to a zero dimensional subscheme of degree $\#\Zpoints$ supported at the origin.
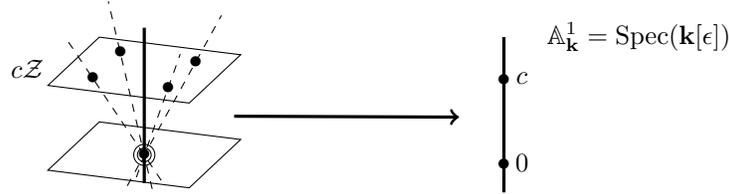
\begin{center}
    \tdplotsetmaincoords{70}{110}
    \begin{tikzpicture}[scale = 0.2]
        \tdplotsetrotatedcoords{0}{0}{0}

        \draw[fill = gray!0, tdplot_rotated_coords] (5,5,6) -- (5,-5,6) -- (-5,-5,6) -- (-5,5,6) -- (5,5,6);

        \draw[fill = gray!0, tdplot_rotated_coords] (5,5,0) -- (5,-5,0) -- (-5,-5,0) -- (-5,5,0) -- (5,5,0);

        \node[tdplot_rotated_coords] at (0,0,0) {$\bullet$};
        \draw[tdplot_rotated_coords] (0,0,0) circle (13pt);
        \draw[tdplot_rotated_coords] (0,0,0) circle (20pt);

        \node[tdplot_rotated_coords] at (-13,-13,0) {$c \Zpoints$};
        \node[tdplot_rotated_coords] at (2.67,2.67,6) {$\bullet$};
        \node[tdplot_rotated_coords] at (-2.67,2.67,6) {$\bullet$};
        \node[tdplot_rotated_coords] at (2.67,-2.67,6) {$\bullet$};
        \node[tdplot_rotated_coords] at (-2.67,-2.67,6) {$\bullet$};

        \draw[very thick, tdplot_rotated_coords] (0,0,-2) -- (0,0,9);
        \draw[dashed, tdplot_rotated_coords] (-0.88,-0.88,-2) -- (4,4,9);
        \draw[dashed, tdplot_rotated_coords] (0.88,-0.88,-2) -- (-4,4,9);
        \draw[dashed, tdplot_rotated_coords] (-0.88,0.88,-2) -- (4,-4,9);
        \draw[dashed, tdplot_rotated_coords] (0.88,0.88,-2) -- (-4,-4,9);

        \draw[very thick, tdplot_rotated_coords] (40,40,16) -- (40,40,27);
        \node[tdplot_rotated_coords] at (40,40,18) {$\bullet$};
        \node[tdplot_rotated_coords] at (42,42,19) {$0$};

        \node[tdplot_rotated_coords] at (40,40,24) {$\bullet$};
        \node[tdplot_rotated_coords] at (42,42,25) {$c$};

        \node[tdplot_rotated_coords] at (55,55,34) {$\mathbb{A}_\kk^1 = \mathrm{Spec}(\kk[\epsilon])$};

        \draw[->, very thick, tdplot_rotated_coords] (10,10,7.4) -- (35,35,19);
    \end{tikzpicture}
\end{center}   
The extension of scalars from $\kk$ to $R=\kk[\epsilon]$
in the second proof corresponds to viewing $\epsilon \Zpoints \subseteq \AA^n_\kk \times \AA^1_{\kk}$ as a flat family over $\AA^1_{\kk} = \Spec(\kk[\epsilon])$ via $\epsilon \Zpoints = \Spec(\kk[\epsilon][\xx]/\II_{\kk[\epsilon]}(\epsilon \Zpoints))$, with generic fibers over $\epsilon=c \in \kk^\times$
and special fiber over $\epsilon=0$.
Taking the limit $\epsilon \to 0$ corresponds to 
applying the map $e$ or (equivalently) the functor 
$(-) \otimes_{\kk[\epsilon]} \kk[\epsilon]/(\epsilon)$,
focusing on the special fiber.
Passage to the fraction field $\kk(\epsilon)$, as in the proof of Lemma~\ref{lem:divided-harmonic-space}, corresponds to applying 
$(-) \otimes_{\kk[\epsilon]} \kk(\epsilon)$,
which geometrically is localization over the general point $(0) \in \Spec(\kk[\epsilon])$. The containment of harmonic spaces $\II_{\kk[\epsilon]}(\Zpoints)^\perp \cdot \II_{\kk[\epsilon]}(\Zpoints')^\perp \subseteq \II_{\kk[\epsilon]}(\Zpoints + \Zpoints')^\perp$ in Corollary~\ref{cor:divided-inhomogeneous-minkowski-addition} (in contrast to the equality in Proposition~\ref{prop:inhomogeneous-minkowski-addition}) philosophically holds because the flat family $\epsilon \Zpoints$ sees behavior over its special fiber $\epsilon \to 0$ which its general fiber $\epsilon \to 1$ does not.

\begin{remark}
\label{rmk:Gundlach-second-remark}
Theorem~\ref{thm:minkowski-closure} has a resemblance to another
interesting observation of Gundlach \cite{Gundlach}.
Suppose $\kk$ has characteristic 0.
For any finite point set $\Zpoints \subseteq \kk^n$, and any choice of monomial ordering $\prec$ on $S=\kk[\xx]$,
his result quoted as Proposition~\ref{prop:Gundlach-prop} earlier
shows the set $\SSS^\prec_\Zpoints$ of all $\prec$-standard monomials for $\II(\Zpoints)$ has the following alternate description:
$\SSS^\prec_\Zpoints= \{\sm_\prec(f): f \in \kk^\Zpoints \setminus \{0\}\}.$
He then uses this to fairly quickly prove the following fact \cite[eqn.~(5)]{Gundlach} :
for any two finite point sets $\Zpoints, \Zpoints' \subseteq \kk^n$, one has
\begin{equation}
\label{eq:Gundlach-Minkowski-closure-fact}
\SSS^\prec_\Zpoints \cdot \SSS^\prec_{\Zpoints'} \subseteq \SSS^\prec_{\Zpoints + \Zpoints'}.
\end{equation}
It is not clear how closely this is related to the similar-sounding Theorem~\ref{thm:minkowski-closure}. 
\end{remark}

\begin{remark}
    \label{rmk:metric-topological-proof}
    An earlier version of this paper \cite[p. 36-41]{RRver2} gives yet another proof of Theorem~\ref{thm:minkowski-closure}. Extending scalars if necessary, one assumes that $\kk$ is a {\em metric topological field} and considers degree-truncating surjections $\hat{\Div} \twoheadrightarrow \Div_{< N}, S \twoheadrightarrow S_{< N}$ to finite-dimensional subspaces for $N$ sufficiently large. Theorem~\ref{thm:minkowski-closure} is deduced from a convergence statement \cite[Lem. 4.17]{RRver2} within Grassmannians of subspaces of these finite-dimensional vector spaces.
\end{remark}

\section{Harmonic algebras}
\label{sec:Harmonic}

Our motivation for Theorem~\ref{thm:minkowski-closure} was to define here
the {\it harmonic algebra} $\HHH_P$ as an approach
toward our main Conjecture~\ref{conj:intro-omnibus} generalizing the Classical Ehrhart Theorem.  Before defining $\HHH_P$, we explain why a less algebraic approach to Conjecture~\ref{conj:intro-omnibus} seems
elusive.

\subsection{The missing valuative property}
\label{subsec:missing-valuation}
Recall from the Introduction that the Classical Ehrhart Theorem
asserts several properties for the
{\it Ehrhart series} $\Eseries_P(t):=\sum_{m=0}^\infty \Ehr_P(m) t^m$
of a $d$-dimensional lattice polytope $P$ in $\RR^n$, where $\Ehr_P(m):=\#(mP \cap \ZZ^n)$.
Specifically, one has its {\it rationality}
$
\Eseries_P(t) = \sum_{i=0}^d h^*_i t^i/(1-t)^{d+1}
$
with a precise {\it denominator}, the {\it nonnegativity} of the numerator coefficients $\{h_i^*\}$, the combinatorial interpretation of the $\{h^*_i\}$
in terms of the semi-open parallelepiped \eqref{eq: semi-open-parallelepiped}
when $P$ is a simplex, and the {\it reciprocity} relating
$\Eseries_P(t^{-1})$ to $\intEseries_P(t)$.

These properties were originally given elementary proofs that avoid any commutative algebra, via the following strategy:  one first proves the result for lattice {\it simplices} (and semi-open simplices), and then generalizes to arbitrary lattice polytopes employing
triangulations (and sometimes {\it shellings} of a triangulation) of $P$.
The key tool in such proofs is the following {\it valuative property} that comes
directly from the definitions of $\Ehr_P(m)$ and $\Eseries_P(t)$:
when $P,Q$ are lattice polytopes and their intersection $P \cap Q$ is a proper face of both $P,Q$, then
\begin{align}
\label{eq:valuation-property-for-Ehrhart-function}
  i_{P \cup Q}(m) &= \Ehr_P(m)+\Ehr_Q(m)-i_{P\cap Q}(m)\\
  \label{eq:valuation-property-for-Ehrhart-series}
  \Eseries_{P \cup Q}(t) &= \Eseries_P(t)+\Eseries_Q(t)-\Eseries_{P\cap Q}(t).
\end{align}

One might therefore hope to prove the analogous
Conjecture~\ref{conj:intro-omnibus} on
\begin{align*}
\Ehr_P(m;q)&:=\Hilb(R(\ZZ^n \cap mP),q)=\Hilb(V_{\ZZ^n \cap mP},q), \text{ and }\\
\Eseries_P(t,q)&:=\sum_{m=0}^\infty \Ehr_P(m;q)t^m,
\end{align*}
via similar valuative techniques. The following simple example
illustrates some difficulty in identifying an analogous valuative property.

\begin{example}
    \label{ex:coprime-example}
    Let $a,b \geq 1$ be coprime positive integers, and let $P,Q \subseteq \RR^2$ be the lattice triangles
    \[P := \mathrm{conv}\{(0,0),(0,b),(a,0)| \quad \text{ and } \quad Q := \mathrm{conv}\{(0,b),(a,0),(a,b)\}.\]
    Their union is the lattice rectangle $P \cup Q = [0,a] \times [0,b]$, and their intersection is the line segment $P \cap Q = \mathrm{conv}\{(0,b),(a,0)\}$.
    The figure below depicts the case $a=5,b=4$.

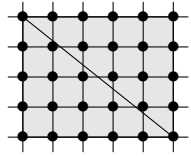
\begin{center}
    \begin{tikzpicture}[scale = 0.4]
        \draw[fill = black!10] (0,0) -- (0,4) -- (5,4) --  (5,0) -- (0,0);
        \draw[] (0,4) -- (5,0);
        \draw[step=1.0,black,thin] (-0.5,-0.5) grid (5.5,4.5);    
        \node at (0,0) {$\bullet$};
        \node at (0,1) {$\bullet$};
        \node at (0,2) {$\bullet$};
        \node at (0,3) {$\bullet$};
        \node at (0,4) {$\bullet$};
        \node at (1,0) {$\bullet$};
        \node at (1,1) {$\bullet$};
        \node at (1,2) {$\bullet$};
        \node at (1,3) {$\bullet$};
        \node at (1,4) {$\bullet$};
        \node at (2,0) {$\bullet$};
        \node at (2,1) {$\bullet$};
        \node at (2,2) {$\bullet$};
        \node at (2,3) {$\bullet$};
        \node at (2,4) {$\bullet$};
        \node at (3,0) {$\bullet$};
        \node at (3,1) {$\bullet$};
        \node at (3,2) {$\bullet$};
        \node at (3,3) {$\bullet$};
        \node at (3,4) {$\bullet$};
        \node at (4,0) {$\bullet$};
        \node at (4,1) {$\bullet$};
        \node at (4,2) {$\bullet$};
        \node at (4,3) {$\bullet$};
        \node at (4,4) {$\bullet$};
        \node at (5,0) {$\bullet$};
        \node at (5,1) {$\bullet$};
        \node at (5,2) {$\bullet$};
        \node at (5,3) {$\bullet$};
        \node at (5,4) {$\bullet$};
    \end{tikzpicture}
\end{center}
 
Since $P,Q$ are equivalent under $\Aff(\ZZ^2)$, and since $P$ and $P\cup Q$ are both antiblocking polytopes, one can apply Corollary~\ref{cor:harmonic-equal-naive-Ehrhart-for-antiblockers} to compute that  
\begin{align}
\label{eq:two-equivalent-triangles-q-Ehrhart-count}
\Ehr_P(m;q)=\Ehr_Q(m;q)
&=\sum_{\substack{\zz=(z_1,z_2)\\ \text{ in } \ZZ^2 \cap mP}} q^{z_1+z_2}
=\sum_{\substack{(z_1,z_2) \in  \ZZ_{\geq 0}: \\ bz_1+az_2 \, \leq \, mab}} q^{z_1+z_2}\\
\label{eq:rectangle-q-Ehrhart-count}
i_{P \cup Q}(m;q) 
 &=\sum_{\substack{\zz=(z_1,z_2)\\ \text{ in } \ZZ^2 \cap m(P\cup Q)}} q^{z_1+z_2} =
[ma]_q \cdot [mb]_q
\end{align}
Since $a,b$ are coprime, the set $P \cap Q \cap \ZZ^2 = \{(0,b),(a,0)\}$ consists of only two lattice points,
and its dilate $\ZZ^2 \cap m(P \cap Q)$ consists of $m+1$ collinear lattice points, so that
\begin{equation}
\label{eq:coprime-line-segment-q-Ehrhart-count}
i_{P \cap Q}(m;q)=1 + q+ q^2 + \cdots + q^m = [m+1]_q.
\end{equation}
It is not clear if there is a $q$-analogue of
the valuative property \eqref{eq:valuation-property-for-Ehrhart-function} that applies here
to relate \eqref{eq:two-equivalent-triangles-q-Ehrhart-count},\eqref{eq:rectangle-q-Ehrhart-count},\eqref{eq:coprime-line-segment-q-Ehrhart-count}.
This subtlety is not surprising-- for general finite loci $\Zpoints, \Zpoints' \subseteq \kk^n$, it is  difficult to predict the structure of 
$\gr \, \II(\Zpoints \cup \Zpoints')$ from that of $\gr \, \II(\Zpoints), \gr \, \II(\Zpoints')$ and $\gr \, \II(\Zpoints \cap \Zpoints)$.
\end{example}

\subsection{The harmonic algebra $\HHH_P$ and a conjecture} 

As explained in the Introduction, work of Stanley (see, e.g., \cite{Stanley-CCA}) provided an alternative proof for the assertions in the Classical Ehrhart Theorem, avoiding valuative techniques, and substituting
commutative algebra.
For a $d$-dimensional lattice polytope $P \subset \RR^n$,
he considered the {\it affine semigroup ring} 
$A_P$ inside 
the Laurent polynomial ring $\kk[y_0, y_1^{\pm 1},\ldots,y_n^{\pm 1}]$ and its 
{\it interior ideal} $\overline{A}_P$, defined as follows:
\begin{align*}
    A_P &:= \kk[\ZZ^{n+1} \cap \cone(P)]
    :=\spn_\kk\{y_0^m \yy^{\zz}: \zz \in mP\},\\
    \overline{A}_P &:=\kk[\ZZ^{n+1} \cap \interior{\cone(P)]}
    :=\spn_\kk\{y_0^m \yy^{\zz}: \zz \in \interior{mP}\},
\end{align*}
This ring $A_P$ is the affine semigroup ring
associated to the polyhedral $(d+1)$-dimensional cone 
$$
\cone(P) := \RR_{\geq 0} \cdot (\{1\} \times P) \subset \RR^{n+1},
$$
whose lattice points in $\ZZ^{n+1}$ form a semigroup under addition. 
The distinguished zero$^{th}$ coordinate on $\RR^{n+1}=\RR^1 \times \RR^n$ endows $A_P$ with the structure of an $\NN$-graded algebra of Krull dimension $d+1$, and one can re-interpret
\begin{align}
\label{eq:affine-semigroup-ring-Hilb-is-Eseries}
\Eseries_P(t)&=\Hilb(A_P,t),\\
\intEseries_P(t)&=\Hilb(\overline{A}_P,t).
\end{align}
As mentioned in the Introduction, Stanley explained the 
\begin{itemize}
 \item {\it rationality} of $\Eseries_P(t)$ via $A_P$ being a finitely generated
$\kk$-algebra, shown by Gordan \cite{Gordan},
\item 
{\it denominator} $(1-t)^{d+1}$ of $\Eseries_P(t)$ via Noether's Normalization Lemma \cite{Noether},
\item {\it nonnegativity} of the $\{h_i^*\}$ via Cohen-Macaulayness of $A_P$, a result of Hochster \cite{Hochster}, and
\item {\it reciprocity}
via $\overline{A}_P \cong \Omega A_P$, the {\it canonical module} of $A_P$, work of Danilov \cite{Danilov}, Stanley \cite{Stanley-Diophantine}.
\end{itemize}
Motivated by this, we would like to eventually
prove the $q$-analogous assertions of Conjecture~\ref{conj:intro-omnibus}, via the following commutative algebra.

\begin{definition} \rm
    \label{def:harmonic-algebra}
    Let $P$ be a lattice polytope in $\RR^n$, and introduce the polynomial ring
    $$
    \RR[y_0,\yy]:=\RR[y_0,y_1,\ldots,y_n]
    \quad 
    \left( 
    \cong \RR[y_0] \otimes_\RR \Div_\RR[\yy]
    \right).
    $$
    We will consider $\RR[y_0,\yy]$ as a {\it bigraded} or {\it $\NN^2$-graded} $\RR$-algebra in which 
    $$
    \deg(y_0)=(1,0), \quad \deg(y_1)=\cdots=\deg(y_n)=(0,1),
    $$
    with Hilbert series in $t,q$ tracking a monomial
    $y_0^{m} y_1^{z_1}\cdots y_n^{z_n}$ 
    by $t^m q^{z_1+\cdots+z_n}$.  Occasionally we will
    specialize this to a single $\NN$-grading by sending $q \mapsto 1$.
    Then define the {\em harmonic algebra} $\HHH_P$ and its {\em interior ideal} $\overline{\HHH}_P$ as the following two $\NN^2$-homogeneous $\RR$-linear subspaces of $\RR[y_0,\yy]$:
\begin{align} 
    \HHH_P &:= \bigoplus_{m=0}^\infty \RR \cdot y_0^m \otimes_\RR V_{\ZZ^n \cap mP},\\
    \overline{\HHH}_P &:= \bigoplus_{m=0}^\infty 
    \RR \cdot y_0^m \otimes_\RR V_{\ZZ^n \cap \interior{mP}}.
\end{align}
\end{definition}
We first justify the terms ``algebra", ``ideal" 
in the definition, starting with an easy observation.

\begin{lemma}
    \label{lem:harmonic-nesting-lemma}
    For nested finite subsets $\Zpoints \subseteq \Zpoints' \subseteq \kk^n$ over any field $\kk$,
    one has $V_\Zpoints \subseteq V_{\Zpoints'}$.
\end{lemma}

\begin{proof}
$\Zpoints \subseteq \Zpoints'$ 
implies the opposite inclusion $\II(\Zpoints) \supseteq \II(\Zpoints')$ within $S=\RR[\xx]$.  This then implies
$\gr \, \II(\Zpoints) \supseteq \gr \, \II(\Zpoints')$.
Taking perps in $\Div_\kk(\yy)$ reverses inclusion: 
$V_\Zpoints = \gr \, \II(\Zpoints)^\perp \subseteq \gr \, \II(\Zpoints')^\perp = V_{\Zpoints'}$.
\end{proof}

\begin{prop}
    \label{prop:harmonic-algebra-is-an-algebra}
    Within the ring $\RR[y_0,\yy]$, the $\RR$-linear subspace $\HHH_P$ is an $\NN^2$-graded $\RR$-subalgebra, and $\overline{\HHH}_P$ is an $\NN^2$-graded ideal of $\HHH_P$, with
\begin{align}
 \label{eq:harmonic-algebra-Hilb-is-q-Eseries}   \Eseries_P(t,q)&=\Hilb(\HHH_P,t,q),\\
\intEseries_P(t,q)&=\Hilb(\overline{\HHH}_P,t,q).
\end{align}
Furthermore, any finite subgroup $G$
of $GL_n(\ZZ)$ preserving $P$ acts by $\NN^2$-graded
automorphisms on $\HHH_P$. Thus in the
representation ring
$\Cl(G)[[t,q]]$ of $\NN^2$-graded $\RR G$-modules, one has
$$
\Eseries^G_P(t,q)=\sum_{m,j \geq 0} [(\HHH_P)_{m,j}] \cdot  t^m q^j 
$$
\end{prop}
\begin{proof}
For the first algebra assertion about $\HHH_P$, given $m,m' \geq 0$, one must check that
$$
\left( \RR \cdot y_0^{m} \otimes_\RR V_{\ZZ^n \cap m P} \right)
\cdot \left( \RR \cdot y_0^{m'} \otimes_\RR V_{\ZZ^n \cap m' P} \right)
\subseteq  \RR \cdot y_0^{m+m'} \otimes_\RR V_{\ZZ^n \cap (m+m') P}.
$$
However, note that one can
use Theorem~\ref{thm:minkowski-closure}
to deduce the first inclusion here
$$
V_{\ZZ^n \cap m P} 
\cdot V_{\ZZ^n \cap m' P} 
\quad  \subseteq  \quad 
V_{(\ZZ^n \cap m P)+(\ZZ^n \cap m' P)}
\quad \subseteq \quad
V_{\ZZ^n \cap (m+m') P}.
$$
The second inclusion follows from Proposition~\ref{lem:harmonic-nesting-lemma}
together with this inclusion
\begin{equation}
\label{eq:strict-inclusion-for-non-IDP}
(\ZZ^n \cap m P)+(\ZZ^n \cap m' P)
\,\, \subseteq \,\,
\ZZ^n \cap (m+m') P,
\end{equation}
derived from $mP + m'P=(m+m')P$.  The ideal assertion for $\overline{\HHH}_P$
is similar, replacing the last fact 
with $mP + \interior{m'P}=
\interior{(m+m')P}$.
The remaining assertions follow from the
definitions.
\end{proof}

We conjecture the following properties for the harmonic algebra $\HHH_P$,
analogous to those known for the affine semigroup ring $A_P$,
which would explain most of Conjecture~\ref{conj:intro-omnibus}.

\begin{conj}
\label{conj:harmonic-algebra-omnibus}
For any lattice polytope $P$, the harmonic algebra $\HHH_P$ is 
\begin{itemize}
\item [(i)] a Noetherian (finitely generated) $\RR$-subalgebra
of $\RR[y_0,\yy]$,
\item[(ii)] a Cohen-Macaulay algebra,
and
\item[(iii)]  its canonical module $\Omega \HHH_P$ is isomorphic to the ideal $\overline{\HHH}_P$, up to a shift in $\NN^2$-grading.
\end{itemize}
\end{conj}

\subsection{Relation of  Conjecture~\ref{conj:harmonic-algebra-omnibus} to Conjecture~\ref{conj:intro-omnibus}}
We explain here the implications of the algebraic Conjecture~\ref{conj:harmonic-algebra-omnibus} on $\HHH_P$ for the enumerative
 Conjecture~\ref{conj:intro-omnibus}
 on its Hilbert series $\Eseries_P(t,q)$.
 
Note that when we specialize the $\NN^2$-grading of the
harmonic algebra $\HHH_P$ to the $\NN$-grading, by setting $q=1$, its Hilbert series $\Eseries_P(t,q)$ becomes
the Ehrhart series $\Eseries_P(t)$, matching that of 
the affine semigroup ring $A_P$:
\begin{equation}
\label{eq:same-singly-graded-Hilbs}
   \left[ \Hilb(\HHH_P,t,q) \right]_{q=1}
   = \left[ \Eseries_P(t,q) \right]_{q=1}
   = \Eseries_P(t) 
   =\Hilb(A_P,t),
\end{equation}
where the first equality uses \eqref{eq:harmonic-algebra-Hilb-is-q-Eseries} and the last is \eqref{eq:affine-semigroup-ring-Hilb-is-Eseries}.
Consequently,
if Conjecture~\ref{conj:harmonic-algebra-omnibus}(i) holds, so $\HHH_P$ is a finitely generated algebra, we know $\HHH_P$ has {\it Krull dimension} $d+1$, the same that of $A_P$.

Also, one could then choose a finite set of
{\it $\NN^2$-homogeneous} algebra generators, e.g., by taking the set of all $\NN^2$-homogeneous components from any finite list of algebra generators.   
Assume one has chosen such a list
of $\NN^2$-homogeneous algebra generators for
$\HHH_P$, or even just some $\{\theta_1,\ldots,\theta_\nu\}$ satisfying the weaker condition that they generate a subalgebra $B \subseteq \HHH_P$ over which $\HHH_P$ is a finite extension.
Name their
$\NN^2$-degrees $\deg(\theta_i)=(b_i,a_i)$ for $i=1,2,\ldots,\nu$, and create an $\NN^2$-graded polynomial ring $S:=\RR[X_1,\ldots,X_\nu]$ with variables
having
$\deg(X_i):=(b_i,a_i)$. Then the map $S\rightarrow \HHH_P$ sending $X_i \mapsto \theta_i$
makes $\HHH_P$ a finitely generated  $\NN^2$-graded $S$-module. Hilbert's Syzygy Theorem therefore predicts the existence of a finite $\NN^2$-graded free $S$-resolution of $\HHH_P$
\begin{equation}
\label{eq:MFR}
0 \leftarrow \HHH_P \leftarrow F_0 \leftarrow F_1 \leftarrow F_2 \leftarrow \cdots \leftarrow F_\nu \leftarrow 0.
\end{equation}
Here each free $S$-module $F_j$ in the resolution has the form
$$
F_j= \bigoplus_{(b,a) \in \NN^2} S(-(b,a))^{\beta_{j,(b,a)}}, 
$$
for some $\NN^2$-graded Betti numbers $\beta_{j,(b,a)}$, meaning $F_j$ has exactly $\beta_{j,(b,a)}$ of its $S$-basis elements of $\NN^2$-degree $(b,a)$. Taking the $\NN^2$-graded Euler characteristic of the resolution \eqref{eq:MFR} implies 
\begin{align*}
\Eseries_P(t,q)=
\Hilb(\HHH_P,t,q)&=\sum_{j=0}^\nu (-1)^j
\Hilb(S(-(b,a))^{\beta_{j,(b,a)}},t,q)\\
&=
\Hilb(S,t,q) \cdot
  \sum_{j=0}^\nu (-1)^j \sum_{(b,a)\in \NN^2}
\beta_{j,(b,a)} q^b t^a\\
&=
\frac{\sum_{j=0}^\nu (-1)^j \sum_{(b,a)\in \NN^2}
\beta_{j,(b,a)} q^b t^a}
  { \prod_{i=1}^\nu (1-q^{a_i} t^{b_i}) }.
\end{align*}
which shows the first rationality assertion in Conjecture~\ref{conj:intro-omnibus} on
$\Eseries_P(t,q)$, and also 
Conjecture~\ref{conj:intro-omnibus}(i).
This argument also works for the rationality assertion on $\intEseries_P(t,q)$:
the ideal $\overline{\HHH}_P$ in $\HHH_P$
would be finitely generated, due to Noetherian-ness of $\HHH_P$, and hence also a finitely generated $S$-module.  

The Cohen-Macaulayness assertion Conjecture~\ref{conj:harmonic-algebra-omnibus}(ii) would 
imply an important special case of Conjecture~\ref{conj:intro-omnibus}(ii), as follows.  
Since $\HHH_P$ has Krull dimension $d+1$, when one picks $\{\theta_1,\ldots,\theta_\nu\}$ as in
the previous paragraph to generate
a subalgebra $B \subseteq \HHH_P$
having $\HHH_P$ as a finite extension,
necessarily $\nu \geq d+1$.  
Suppose $P$ is a simplex, and that $\nu=d+1$, meaning {\it
$\HHH_P$ contains an $\NN^2$-homogeneous system of parameters}\footnote{Noether's Normalization Lemma shows $\NN$-graded Noetherian $\RR$-algebras always contain an $\NN$-homogeneous system of parameters. But some
$\NN^2$-graded Noetherian $\RR$-algebras have {\it no} $\NN^2$-homogeneous system of parameters.
E.g., inside $\RR[t,x,y]$ with
$\NN^2$-grading where $\deg(t)=(1,0)$ and $\deg(x)=(0,1)=\deg(y)$,
the $\NN^2$-graded subalgebra $R=\RR[t,tx,ty^2,txy^2]$ of Krull dimension $3$ has
no $\NN^2$-homogeneous system of parameters $\{\theta_1,\theta_2,\theta_3\}$, and it has
$$\Hilb(R,t,q)=\frac{1-t^2q^3}{(1-t)(1-tq)(1-tq^2)(1-tq^3)}.$$}. Then the previous
paragraph shows that the hypothesis of
Conjecture~\ref{conj:intro-omnibus}(ii) holds, and we claim that Conjecture~\ref{conj:harmonic-algebra-omnibus}(ii) yields the conclusion of Conjecture~\ref{conj:intro-omnibus}(ii) for the numerator $N_P(t,q)$
of $\Eseries_P(t,q)$:  Cohen-Macaulayness implies $\HHH_P$ is a {\it free} $S$-module, so that the resolution \eqref{eq:MFR} stops at $F_0$, and
$$
\Eseries_P(t,q)=\Hilb(F_0,t,q)=
\frac{ \sum_{(b,a)\in \NN^2}
\beta_{0,(b,a)} q^b t^a}
  { \prod_{i=1}^\nu (1-q^{a_i} t^{b_i}) }.
$$
In other words, $N_P(t,q)=\sum_{(b,a)\in \NN^2}
\beta_{0,(b,a)} q^b t^a$.
A similar argument applies to the numerator $\overline{N}_P(t,q)$
of $\intEseries_P(t,q)$ in this case, assuming that Conjecture~\ref{conj:harmonic-algebra-omnibus}(iii) holds: this
would imply that $\overline{\HHH}_P$
is isomorphic to the Cohen-Macaulay module $\Omega \HHH_P$, and hence is also a free $S$-module.

Lastly, Conjecture~\ref{conj:harmonic-algebra-omnibus}(iii)
could imply the $q$-reciprocity assertion Conjecture~\ref{conj:intro-omnibus}(iii), if the shift in $\NN^2$-grading for the isomorphism $\Omega \HHH_P \cong \overline{\HHH}_P$ works out correctly: one knows
that the canonical module $\Omega \HHH_P$ for an
$\NN^2$-graded Cohen-Macaulay ring $\HHH_P$ of Krull dimension $d+1$ satisfies 
$$
\Hilb(\Omega\HHH_P,t,q)= (-1)^{d+1} q^B t^A \cdot \Hilb(\HHH_P,t^{-1},q^{-1})
$$
for some choice of $\NN^2$-degree $(B,A)$; see Stanley \cite[\S I.12, p.49]{Stanley-CCA}.

\subsection{The singly-graded rings $\HHH_P$
and $A_P$ are generally not isomorphic}
\label{sec:lack-of-singly-graded-isomorphism}

In light of the fact \eqref{eq:same-singly-graded-Hilbs} that $A_P$ and $\HHH_P$ have the same singly-graded Hilbert series, one might wonder whether they are isomorphic as $\NN$-graded algebras.  This fails already for small lattice polygons, e.g., this triangle $P=\conv\{(0,0),(1,2),(3,1)\}$,
for which a portion of $\ZZ^3 \cap \cone(P)$ is shown below:
\begin{center}
    \tdplotsetmaincoords{70}{110}
    \begin{tikzpicture}[scale = 0.4]
        \tdplotsetrotatedcoords{90}{0}{-10}

        \draw[fill = gray!30, tdplot_rotated_coords] (0,0,2) -- (1,2,2) -- (3,1,2) -- (0,0,2);
        \draw[fill = gray!30, tdplot_rotated_coords] (0,0,4) -- (2,4,4) -- (6,2,4) -- (0,0,4);
    
        \node[tdplot_rotated_coords] at (0,0,0) {$\bullet$};

        \node[tdplot_rotated_coords] at (0,0,2) {$\bullet$};
        \node[tdplot_rotated_coords] at (1,2,2) {$\bullet$};
        \node[tdplot_rotated_coords] at (3,1,2) {$\bullet$};
        \node[tdplot_rotated_coords] at (1,1,2) {$\bullet$};
        \node[tdplot_rotated_coords] at (2,1,2) {$\bullet$};

        \node[tdplot_rotated_coords] at (0,0,4) {$\bullet$};
        \node[tdplot_rotated_coords] at (2,4,4) {$\bullet$};
        \node[tdplot_rotated_coords] at (6,2,4) {$\bullet$};
        \node[tdplot_rotated_coords] at (3,1,4) {$\bullet$};
        \node[tdplot_rotated_coords] at (1,1,4) {$\bullet$};
        \node[tdplot_rotated_coords] at (1,2,4) {$\bullet$};
        \node[tdplot_rotated_coords] at (2,2,4) {$\bullet$};
        \node[tdplot_rotated_coords] at (2,3,4) {$\bullet$};
        \node[tdplot_rotated_coords] at (2,1,4) {$\bullet$};
        \node[tdplot_rotated_coords] at (3,2,4) {$\bullet$};
        \node[tdplot_rotated_coords] at (3,3,4) {$\bullet$};
        \node[tdplot_rotated_coords] at (4,2,4) {$\bullet$};
        \node[tdplot_rotated_coords] at (4,3,4) {$\bullet$};
        \node[tdplot_rotated_coords] at (5,2,4) {$\bullet$};
      
        \draw[very thick, tdplot_rotated_coords] (0,0,0) -- (0,0,5);
        \draw[dashed, tdplot_rotated_coords] (0,0,0) -- (2.5,5,5);
        \draw[dashed, tdplot_rotated_coords] (0,0,0) -- (7,2.5,5);
    \end{tikzpicture}
\end{center} 
To see that $A_P \not\cong \HHH_P$ as $\NN$-graded algebras, one can check that 
\begin{align}
\label{eq:this-triangle-is-IDP}
(A_P)_1 \cdot (A_P)_1 &=(A_P)_2,\text{ but }\\
\label{eq:this-triangle-has-harmonic-algebra-not-1-generated}
(\HHH_P)_1 \cdot (\HHH_P)_1 & \subsetneq (\HHH_P)_2.
\end{align}
Checking \eqref{eq:this-triangle-is-IDP} means showing 
$
(\ZZ^2 \cap P) + (\ZZ^2 \cap P) = (\ZZ^2 \cap 2P),
$
which can be done directly. 
Meanwhile, the proper inclusion \eqref{eq:this-triangle-has-harmonic-algebra-not-1-generated} follows from examining the first few terms of $\Eseries_P(t,q)$:
\[
\Eseries_P(t,q) = 1 + (1 + 2q + 2q^2) t + (1 + 2q + 3q^2 + 4 q^3 + 3 q^4 + q^5) t^2 + o(t^2) 
\]
One sees $(\HHH_P)_1 \cdot (\HHH_P)_1$ has as its highest nonvanishing $q$-degree  $q^2\cdot q^2=q^4$, but for  $(\HHH_P)_2$ it is $q^5$.
This example seems related to the deformation $S/\II(\Zpoints) \leadsto S/\gr \, \II(\Zpoints)$ forgetting algebra structure.

\begin{remark}
In contrast to the above example, it may be interesting to note that there is a very similar-sounding algebra to $\HHH_P$ which actually 
{\it is} isomorphic to $A_P$.
Recall from Section~\ref{sec:completions-and-exponentials} that for the inhomogeneous ideal $\II(\Zpoints)$ one can still define a perp space $\II(\Zpoints)^\perp$ inside the power series completion $\hat{\Div}=\hat{\Div}_\RR[\yy] \cong \RR[[y_1,\ldots,y_n]]$.  The space $\II(\Zpoints)^\perp$  has $\RR$-basis given in Lemma~\ref{lem:inhomogeneous-exponential-divided} by
$\{ \exp(\zz \cdot \yy) \}_{\zz \in \Zpoints}$, and \eqref{eq:exponential-law} asserts that these basis elements multiply via the usual
rule 
$$
\exp( \zz \cdot \yy) \cdot \exp( \zz' \cdot \yy) = \exp( (\zz+\zz') \cdot \yy).
$$
If one then considers the ring $\RR[y_0][[y_1,\ldots,y_n]]$ as $\NN$-graded by $\deg(y_0^m)=1$
and $\deg(y_i)=0$ for $i=1,2,\dots,n$, one can
compile the spaces
$\II(\ZZ^n \cap mP)^\perp$ for $m\geq 0$ to
define an $\NN$-graded subalgebra 
$$
\hat{\HHH}_P:=\bigoplus_{m=0}^\infty \RR \cdot y_0^m \otimes \II(\ZZ^n \cap mP)^\perp
\quad \text{ inside }
\RR[y_0][[y_1,\ldots,y_n]].
$$
Since $A_P$ has $\RR$-basis $\{y_0^m \yy^\zz: \zz \in \ZZ^n \cap mP\}$,
with $y_0^m \yy^\zz \cdot y_0^{m'} \yy^{\zz'} = y_0^{m+m'} \yy^{\zz+\zz'}$,
the map 
$$
\begin{array}{rcl}
A_P &\longrightarrow &\hat{\HHH}_P\\
y_0^m \yy^\zz &\longmapsto& y_0^m \otimes \exp(\zz \cdot \yy)
\end{array}
$$
is an $\NN$-graded $\RR$-algebra isomorphism.
\end{remark}

\subsection{Example: antiblocking polytopes revisited}

In spite of the example in Section~\ref{sec:lack-of-singly-graded-isomorphism}, there is a subfamily of lattice polytopes $P$ for which one has not only an $\NN$-graded,
but even an $\NN^2$-graded algebra isomorphism 
$\HHH_P \cong A_P$, and Conjecture~\ref{conj:harmonic-algebra-omnibus} holds:  the {\it antiblocking polytopes}
of Section~\ref{sec:antiblocking}.  

In order to discuss this, we first note that
for a lattice polytope $P \subset \RR^n$, the
affine semigroup ring $A_P$ has a finer $\NN \times \ZZ^n$-grading.
This is because it is a homogeneous subalgebra
of $\RR[y_0,y_1^{\pm 1},\ldots,y_n^{\pm 1}]$
with respect to the $\NN \times \ZZ^n$-multigrading
in which $\deg(y_i)=(0,\ldots,0,1,0,\ldots,0)$
is the $i^{th}$ standard basis vector.  One can
then specialize this to a $\NN^2$-grading via $\deg(y_0)=(1,0)$ and $\deg(y_i)=(0,1)$ for $i=1,2,\ldots,n$.  In other words, this $\NN^2$-grading has $\deg(y_0^m y_1^{z_1} \cdots y_n^{z_n})=(m,z_1+\cdots+z_n)$, tracked
by the Hilbert series monomial $t^m q^{z_1+\cdots+z_n}$.

\begin{prop}
    \label{prop:lower-convex-semigroup}
    For any antiblocking lattice polytope $P \subseteq \RR^n$ one has an $\NN^2$-graded $\RR$-algebra isomorphism $A_P \longrightarrow \HHH_P$ defined by the bijection on $\RR$-bases $y_0^m \yy^\zz \longmapsto y_0^m \otimes \yy^\zz$ for $\zz \in \ZZ^n \cap mP$. 
    That is, identifying $\RR[y_0] \otimes_\RR \Div_\RR(\yy)
    \cong \RR[y_0,\yy]$, one has equality $A_P=\HHH_P$
    of the two subalgebras. 
    
Consequently, Conjecture~\ref{conj:harmonic-algebra-omnibus} holds for antiblocking polytopes.
\end{prop}

\begin{proof}
    The identification $A_P=\HHH_P$  follows from Lemma~\ref{lem:lower-harmonic}(iii). Then Conjecture~\ref{conj:harmonic-algebra-omnibus} follows from the previously mentioned results of Gordan, of Hochster and of Danilov and Stanley, showing that $A_P$ is finitely-generated, Cohen-Macaulay,
    and with canonical module $\Omega A_P \cong \overline{A}_P$, respectively.
\end{proof}

\subsection{Example: Chain and order polytopes}
\label{sec: chain-order-polytope}
Let $(\PPP, \prec)$ be a finite poset. Stanley \cite{stanley1986two} associated two polytopes in the positive orthant of $\RR^\PPP$ to the poset $\PPP$ as follows. The {\em chain polytope} $C_\PPP$ consists of functions $f: \PPP \to \RR$ satisfying 
\begin{center}
$0 \leq f(p_1) + \cdots + f(p_r) \leq 1$ for each chain $p_1 \prec \cdots \prec p_r$ in $\PPP$.
\end{center} 
The {\em order polytope} $O_\PPP$ consists of functions $g: \PPP \to \RR$ satisfying
\[\begin{cases}  g(p) \geq 0 & \text{for all $p \in \PPP$,} \\ g(p) \leq g(p') & \text{whenever $p \prec p'$.} \end{cases}\]

The chain polytope $C_\PPP$ is antiblocking, so the previous example applies to describe $\HHH_{C_\PPP}$ and proves that Conjecture~\ref{conj:harmonic-algebra-omnibus} holds. Although the order polytope $O_\PPP$ is almost never antiblocking, we will show that Conjecture~\ref{conj:harmonic-algebra-omnibus} holds for this family of polytopes, as well.

For any $m \geq 0$, Stanley introduced a piecewise-linear map
\begin{equation}
    \varphi: m O_\PPP \to m C_\PPP
\end{equation}
given by the formula
\begin{equation}
    \varphi g(p) = g(p) - \max_{p' \prec p} g(p')
\end{equation}
where the maximum is taken over elements $p' \in \PPP$ covered by $p$. If $p$ is a minimal element of $\PPP$, this maximum is interpreted to be 0. Stanley proved that $\varphi$ is bijective, and restricts to a bijection
\begin{equation}
\label{eq:Stanley-bijection}
m O_{\PPP} \cap \ZZ^{\PPP} \xrightarrow{\, \, \sim \, \,} m C_\PPP \cap \ZZ^{\PPP}.
\end{equation}
Consequently, $\Ehr_{O_\PPP}(m)=\Ehr_{C_\PPP}(m)$,
and $\Eseries_{O_\PPP}(t)=\Eseries_{C_\PPP}(t)$.
Observe that the map $\varphi$ is only piecewise-linear; indeed, the chain and order polytopes are not in general affine-equivalent. Nevertheless, we have the following result.

\begin{thm}
    \label{thm:chain-order}
    For finite posets $\PPP$,
    inside $\RR[y_0,\yy]$,
    one has equality $\HHH_{C_\PPP}=\HHH_{O_\PPP}$
    of 
    the harmonic algebras of its chain and  order polytopes, as well as equality
    $\overline{\HHH}_{C_\PPP}= \overline{\HHH}_{O_\PPP}$
    for their interior ideals.
\end{thm}

\begin{proof}
    To prove $\HHH_{C_\PPP} = \HHH_{O_\PPP}$ it suffices to establish for each $m \geq 0$, the equality of ideals $$\gr \, \II(m O_{\PPP} \cap \ZZ^{\PPP}) = \gr \, \II(m C_\PPP \cap \ZZ^{\PPP})$$ inside the polynomial ring $\RR[\xx]$ where $\xx = \{x_p \,:\, p \in \PPP\}$. 
    In fact, it suffices to show an inclusion,
\begin{equation}
\label{eq:chain-order-ideal-gr-ideal-inclusion}
    \gr \, \II(m C_\PPP \cap \ZZ^{\PPP}) \subseteq
    \gr \, \II(m O_\PPP \cap \ZZ^{\PPP}),
\end{equation}
since then the equality
$$
\#(m O_{\PPP} \cap \ZZ^{\PPP}) =\Ehr_{O_\PPP}(m)= \Ehr_{C_\PPP}(m)= \#(m C_\PPP \cap \ZZ^{\PPP})
$$
would show the
surjection
$
R(m C_\PPP \cap \ZZ^{\PPP}) \twoheadrightarrow R(m O_\PPP \cap \ZZ^{\PPP})
$
is bijective, via dimension-counting.

For any $f: \PPP \to \ZZ_{\geq 0}$, let $\xx^f := \prod_{p \in \PPP} x_p^{f(p)}$ in $\RR[\xx]$. Since $C_\PPP$ is antiblocking, we have
    \begin{equation*}
        \gr \, \II(m C_\PPP \cap \ZZ^{\PPP}) = \mathrm{span}_\RR \{ \xx^f \,:\, f: \PPP \to \ZZ_{\geq 0} \text{ is not a lattice point in $m C_\PPP$} \}.
    \end{equation*}
The desired inclusion \eqref{eq:chain-order-ideal-gr-ideal-inclusion} is then implied by the following claim.
    \begin{quote}
    {\bf Claim:} Suppose that $f: P \to \ZZ_{\geq 0}$ is not a lattice point of $m C_\PPP$. Then there exists an element $g \in \II(m O_{\PPP} \cap \ZZ^\PPP)$ such that $\xx^f = \tau(g)$.
    \end{quote}
To prove the Claim, first observe that 
    $$m O_\PPP \cap \ZZ^\PPP = \{ f: \PPP \to \{0,1,\dots,m\} \,:\, p \preceq p' \, \Rightarrow \, f(p) \leq f(p') \}.$$
    The required element $g \in \II(m O_{\PPP} \cap \ZZ^\PPP)$ may be constructed explicitly. Since $f \notin m C_\PPP$, there is a chain $p_1 \prec \cdots \prec p_r$ in $\PPP$ such that  $f(p_1) + \cdots + f(p_r) > m$.  Consider the polynomial
\begin{align*}
    g(\xx) &= x_{p_1} (x_{p_1} - 1) \cdots (x_{p_1} - (f(p_1) - 1)) \\
    & \quad \cdot (x_{p_2} - f(p_1)) (x_{p_2} - (f(p_1) + 1)) \cdots (x_{p_2} - (f(p_1) + f(p_2) - 1)) \\
    & \quad \cdot (x_{p_2} - (f(p_1)+f(p_2))) (x_{p_2} - (f(p_1)+f(p_2) + 1)) \cdots (x_{p_2} - (f(p_1) + f(p_2)+f(p_3) - 1)) \\
    &\qquad \qquad \qquad \vdots  \\
    & \quad \cdot \left(x_{p_r} - \sum_{i=1}^{r-1}f(p_i)\right) 
    \left(x_{p_r} - \sum_{i=1}^{r-1}f(p_i)-1)\right) \cdots \left(x_{p_r} - \sum_{i=1}^{r}f(p_i) \right)\\
    & \quad \cdot \prod_{p \, \neq \, p_1, \dots, p_r} x_p^{f(p)}
\end{align*}
One can readily check that 
$
\tau(g)=\prod_{p \, \neq \, p_1, \dots, p_r} x_p^{f(p)} \cdot \prod_{i=1}^r x_{p_r}^{f(p_r)} =\xx^f.
$
It remains to show that $g(\xx)=0$ for all $\xx$ in $\ZZ^\PPP \cap m O_\PPP$.  To see this, assume for the sake of contradication that $\xx$ lies in
$\ZZ^\PPP \cap m O_\PPP$ but $g(\xx) \neq 0$.  Adopting the
convention that $x_{p_0}:=0$, one can show by induction on $i$ that $x_{p_i} \geq f(p_1)+f(p_2) + \cdots + f(p_i)$
for each $i=0,1,2,\ldots,r$, as follows. 
The base case $i=0$ follows from $x_0=0$. In the inductive step, note that $p_{i-1} \prec p_i$ and $\xx \in O_\PPP$ implies that 
\begin{equation}
\label{eq:order-polytope-induction-inequality}
x_{p_i} \geq x_{p_{i-1}} \geq f(p_1)+f(p_2) + \cdots + f(p_{i-1}).
\end{equation}
On the other hand, the fact that $g(\xx) \neq 0$ implies that
\begin{equation}
\label{eq:root-avoidance}
x_{p_i} \not\in \left\{\sum_{j=1}^{i-1}f(p_j)\,\, ,\,\,
1+\sum_{j=1}^{i-1}f(p_j)\,\, , \,\, 
2+\sum_{j=1}^{i-1}f(p_j)\,\, , \,\, 
\cdots \,\, , \,\,  
f(p_i)-1+\sum_{j=1}^{i-1}f(p_j)\right\}.
\end{equation}
But then since $\xx$ lies in $\ZZ^\PPP$, together \eqref{eq:order-polytope-induction-inequality}, \eqref{eq:root-avoidance}  imply $x_{p_i} \geq f(p_1)+f(p_2) + \cdots + f(p_i)$, completing the inductive step.  However, one then reaches the contradiction $x_{p_r} \geq \sum_{i=1}^r f(p_i) > m$.

The proof that $\overline{\HHH}_{C_\PPP} = \overline{\HHH}_{O_\PPP}$ is similar.
It again suffices to prove for $m \geq 1$ the inclusions
\begin{equation}
\label{eq:chain-order-interior-gr-inclusion}
\gr \, \II(\interior{m C_\PPP} \cap \ZZ^\PPP) \subseteq \gr \, \II(\interior{m O_\PPP} \cap \ZZ^\PPP),
\end{equation}
since they would then be equalities
via dimension-count: as $C_\PPP, O_\PPP$ have the same classical Ehrhart series, then Ehrhart-Macdonald Reciprocity implies  
$\#(\interior{m C_\PPP} \cap \ZZ^\PPP) = \#(\interior{m O_\PPP} \cap \ZZ^\PPP)$.

To argue \eqref{eq:chain-order-interior-gr-inclusion},
note that since $C_\PPP$ is antiblocking, the ideal $\gr \, \II(\interior{m C_\PPP} \cap \ZZ^\PPP)$ is spanned over $\RR$ by monomials $\xx^f$ where $f: \PPP \to \ZZ_{\geq 0}$ satisfies  $N := f(p_1) + \cdots + f(p_r) > m-2$ for some chain $p_1 \prec \cdots \prec p_r$ in $\PPP$. If $\xx^f$ is such a monomial, consider the polynomial
\begin{align*}
g(\xx) 
&= (x_{p_1} - 1) (x_{p_1} - 2) \cdots (x_{p_1} - f(p_1)) \\
   &\quad \cdot (x_{p_1} -(f(p_1) + 1)) (x_{p_1} - (f(p_1) + 2)) \cdots (x_{p_1} - (f(p_1) +f(p_2))) \\
   &\quad \cdot (x_{p_1} -(f(p_1) +f(p_2) + 1)) (x_{p_1} - (f(p_1) + f(p_2)+2)) \cdots (x_{p_1} - (f(p_1) +f(p_2)+f(p_3))) \\
    &\qquad \qquad \quad \vdots\\
   &\quad \cdot 
    \left(x_{p_r} - \left( \sum_{i=1}^{r-1}f(p_i)+1\right)\right) 
    \left(x_{p_r} - \left( \sum_{i=1}^{r-1}f(p_i)+2\right)\right)
    \cdots 
     \left(x_{p_r} -  \sum_{i=1}^{r}f(p_i)\right) \\
    &\quad \cdot \prod_{p \, \neq \, p_1, \dots, p_r} x_p^{f(p)}
\end{align*}
which one can readily check has $\tau(g)=\xx^f$. Using the description
$$\interior{m O_\PPP} \cap \ZZ^\PPP = \{ f: \PPP \to \{1,\dots,m-1\} \,:\, p \prec p' \, \Rightarrow \, f(p) < f(p') \}$$
one similarly checks $g(\xx)$ vanishes on $\interior{m O_\PPP} \cap \ZZ^\PPP$:  any $\xx$ in 
$\interior{m O_\PPP} \cap \ZZ^\PPP$ with $g(\xx)\neq 0$
would have $x_{p_i} \geq 1+\sum_{j=1}^i f(p_j)$ for $i=1,\ldots,r$,
and $x_{p_r} \geq 1+\sum_{j=1}^r f(p_j)>m-1$ is a contradiction.
\end{proof}

In contrast to Theorem~\ref{thm:chain-order}, for most posets $\PPP$ the affine semigroup rings of $C_\PPP$ and $O_\PPP$ are not isomorphic. Indeed, if $P, Q$ are lattice polytopes, even in the category of ungraded algebras it follows from \cite[Chap. 5]{bruns2009polytopes} that $A_P \cong A_Q$ if and only if $P$ is unimodularly equivalent to $Q$. Hibi and Li proved \cite{hibi2016unimodular} that for a poset $\PPP$, the chain polytope $C_\PPP$ and the order polytope $O_\PPP$ are unimodularly equivalent if and only if $\PPP$ does not contain the 5-element `X-shape' shown below as a subposet. Theorem~\ref{thm:chain-order} therefore gives an infinite family of lattice polytopes with isomorphic harmonic algebras but nonisomorphic semigroup rings.

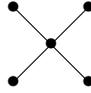
\begin{center}
    \begin{tikzpicture}[scale = 0.5]
        \node (A) at (0,0) {$\bullet$};
        \node (B) at (2,0) {$\bullet$};
        \node (C) at (1,1) {$\bullet$};
        \node (D) at (0,2) {$\bullet$};
        \node (E) at (2,2) {$\bullet$};

        \draw (0,0) -- (1,1);
        \draw (2,0) -- (1,1);
        \draw (1,1) -- (0,2);
        \draw (1,1) -- (2,2);
    \end{tikzpicture}
\end{center}

\begin{remark}
\label{rmk:two-extreme-posets}
We note that it is easy to compute $\Eseries_{C_\PPP}(t,q)=\Eseries_{O_\PPP}(t,q)$ for the two ``extreme" posets $\PPP$ on $n$ elements: chains and antichains.  When $\PPP$ is the chain $1<2<\cdots<n$ on  
$\{1,2,\ldots,n\}$, by definition, $C_\PPP$ is the simplex $\pyr(\Delta^{n-1})$ considered in Section~\ref{sec:standard-simplices}. Therefore one has
$$
\Eseries_{C_\PPP}(t,q)=\Eseries_{O_\PPP}(t,q)
=\Eseries_{\pyr(\Delta^{n-1})}(t,q)=\frac{1}{(1-t)(1-tq)^n}.
$$
When $\PPP$ is an antichain on $n$ elements, by definition $O_\PPP=C_\PPP=[0,1]^n$, the {\it $n$-dimensional cube}.
Using the fact that these are antiblocking polytopes,
from Corollary~\ref{cor:harmonic-equal-naive-Ehrhart-for-antiblockers} one can conclude that
$$
\Ehr_{C_\PPP}(m;q)=\Ehr_{O_\PPP}(m;q)
=\Ehr_{[0,1]^n}(m;q)
=\sum_{\zz \in \{0,1,\ldots,m\}^n} q^{z_1+\cdots+z_n}
=( [m+1]_q )^n.
$$
This then implies that
\begin{align}
\notag
\Eseries_{C_\PPP}(t,q)=\Eseries_{O_\PPP}(t,q)
=\Eseries_{[0,1]^n}(t,q)
&= \sum_{m = 0}^\infty t^m  ( [m+1]_q )^n\\
\label{eq:Carlitz-Macmahon}
&= \frac{\sum_{w \in \symm_n}
t^{\mathrm{des}(w)} q^{\mathrm{maj}(w)}}{ (1-t) (1-t q) (1-t q^2) ... (1-t q^n) },
\end{align}
where the {\it descent number} $\mathrm{des}(w)$
and {\it major index} $\mathrm{maj}(w)$ are defined by
\begin{align*}
\mathrm{des}(w)&=\#\{i: w(i) > w(i+1)\},\\
\mathrm{maj}(w)&=\sum_{i:w(i) > w(i+1)} i.
\end{align*}
The last equality in \eqref{eq:Carlitz-Macmahon} is a famous result of Carlitz, sometimes also credited to MacMahon;  see the historical discussion 
surrounding Theorem 1.1 in Braun and Olsen \cite{BraunOlsen}.
\end{remark}

\subsection{Case study: an interesting triangle}
\label{sec:interesting-triangle-in-depth}

For lattice polytopes $P$ that are {\it not} antiblocking, 
the analysis of $\HHH_P$ can be significantly more complicated. 
We return to study the lattice triangle $P := \mathrm{conv}\{\origin, (1,2), (2,1)\} \subset \RR^2$ 
from Remark~\ref{remark-first-interesting-triangle-discussion}, Example~\ref{ex:second-interesting-triangle-discussion}, with $P, 2P, 3P$ shown below:
\begin{center}
    \begin{tikzpicture}[scale = 0.3]
        \draw[fill = black!10] (0,0) -- (1,2) -- (2,1) --  (0,0);
        \draw[step=1.0,black,thin] (-0.5,-0.5) grid (2.5,2.5);
        \node at (0,0) {$\bullet$};
        \node at (1,1) {$\bullet$};
        \node at (1,2) {$\bullet$};
        \node at (2,1) {$\bullet$};
    \end{tikzpicture}
    \qquad 
       \begin{tikzpicture}[scale = 0.3]
        \draw[fill = black!10] (0,0) -- (2,4) -- (4,2) --  (0,0);
        \draw[step=1.0,black,thin] (-0.5,-0.5) grid (4.5,4.5);
        \node at (0,0) {$\bullet$};
        \node at (1,1) {$\bullet$};
        \node at (1,2) {$\bullet$};
        \node at (2,1) {$\bullet$};
        \node at (2,2) {$\bullet$};
        \node at (3,2) {$\bullet$};
        \node at (2,3) {$\bullet$};
        \node at (4,2) {$\bullet$};
        \node at (3,3) {$\bullet$};
        \node at (2,4) {$\bullet$};
    \end{tikzpicture}
     \qquad 
        \begin{tikzpicture}[scale = 0.3]
        \draw[fill = black!10] (0,0) -- (3,6) -- (6,3) --  (0,0);
        \draw[step=1.0,black,thin] (-0.5,-0.5) grid (6.5,6.5);
        \node at (0,0) {$\bullet$};
        \node at (1,1) {$\bullet$};
        \node at (1,2) {$\bullet$};
        \node at (2,1) {$\bullet$};
        \node at (2,2) {$\bullet$};
        \node at (3,2) {$\bullet$};
        \node at (2,3) {$\bullet$};
        \node at (4,2) {$\bullet$};
        \node at (3,3) {$\bullet$};
        \node at (2,4) {$\bullet$};
        \node at (4,3) {$\bullet$};
        \node at (3,4) {$\bullet$};
        \node at (5,3) {$\bullet$};
        \node at (4,4) {$\bullet$};
        \node at (3,5) {$\bullet$};
        \node at (6,3) {$\bullet$};
        \node at (5,4) {$\bullet$};
        \node at (4,5) {$\bullet$};
        \node at (3,6) {$\bullet$};
    \end{tikzpicture}
\end{center}
This triangle $P$ is not antiblocking, and is $\Aff(\ZZ^2)$-equivalent to the second lattice triangle of area $3$ in Figure~\ref{fig:q-Ehrhart-examples}, with these Ehrhart and $q$-Ehrhart series
\begin{align*}
\Eseries_P(t) &=\frac{1+ t + t^2 }{(1-t)^3}=1 +4t + 10t^2 + 19t^3 + \cdots\\
\Eseries_P(t,q) &= \frac{(1+qt)(1+ q t + q^2 t^2)}{(1-t)(1- q^2 t)(1 - q^3 t^2)}\\
&=1 + (1+2q+q^2)t+ (1+2q+3q^2+3q^3+q^4)t^2+(1+2q+3q^2+4q^3+5^4+3q^5+q^6)t^3+\cdots
\end{align*}
We explain below how one can use the harmonic algebra $\HHH_p$ to prove the above calculation of $\Eseries_P(t,q)$ is correct.  Before doing so, we note
two interesting features of $\Eseries_P(t,q)$.

First, note $[\Eseries_P(t,q)]_{q=1}=\Eseries_P(t)$ requires canceling a numerator/denominator factor at $q=1$.

Second, note that the bidegrees $t, q^2 t, q^3 t^2$ appearing in the denominator factors of $\Eseries_P(t,q)$ are {\it all different}, even though the triangle $P$ has a $GL(\ZZ^2)$-symmetry that swaps the two vertices $(2,1),(1,2)$.   Thus even in the special case where both  Conjecture~\ref{conj:harmonic-algebra-omnibus}(ii) and Conjecture~\ref{conj:intro-omnibus}(ii) hold and $P$ is a lattice triangle with $\nu=d+1=3$, one should not expect some simple and natural bijection between the three vertices of $P$ and the $\NN^2$-graded system of parameters $\{\theta_1,\theta_2,\theta_3\}$.

Expecting that Conjecure~\ref{conj:harmonic-algebra-omnibus} might hold, and examining $\Eseries_P(t,q)$, shown color-coded here,
\begin{equation}
\label{interesting-triangle-q-Ehrhhart-series-colored}
\Eseries_P(t,q) =
    \frac{{\color{brown}1+2qt+2 q^2 t^2+q^3 t^3}}{(1-{\color{red}t})(1- {\color{red}q^2 t})(1 - {\color{red}{q^3 t^2}})}.
\end{equation}
one might expect from the denominator that the harmonic algebra 
$\HHH_P$ contains a homogeneous system of parameters $\theta_1, \theta_2, \theta_3$ of $\NN^2$-degrees $(0,1),(1,2),(2,3)$. From the numerator one might expect to 
find six $\RR[\theta_1, \theta_2, \theta_3]$-basis elements whose $\NN^2$-degrees match $1 + 2qt + 2q^2 t^2 + q^3 t^3.$
Using {\tt Macaulay2}, one can compute $\NN$-graded
$\RR$-bases for the harmonic spaces $V_{\ZZ^2 \cap mP}$ with $m=0,1,2,3$, and hence $\NN^2$-graded $\RR$-bases for $(\HHH_P)_0, (\HHH_P)_1, (\HHH_P)_2, (\HHH_P)_3$.  We then rewrote these bases to make the action of 
$G=\ZZ/2\ZZ$ apparent in each graded piece, and identified candidates
for $\theta_1,\theta_2,\theta_3$ and the six $\RR[\theta_1, \theta_2, \theta_3]$-basis elements, color-coded in this table:
$$
\tiny
\begin{array}{|c||c|c|c|c|c|c|c|}\hline
(\HHH_P)_0&{\color{brown}1}& & & & & &\\\hline
(\HHH_P)_1&{\color{red}\theta_1=}&{\color{brown}y_0y_1}&{\color{red}\theta_2=}& & & &\\
           &{\color{red} y_0}&{\color{brown}y_0y_2}&{\color{red}y_0(y_1^2+y_1y_2+y_2^2)}          & & & &\\\hline
(\HHH_P)_2&y_0^2&y_0^2y_1&{\color{brown}y_0^2y_1^2}&y_0^2y_1^3& y_0^2(y_1^4+2y_1^3y_2& & \\
           & &y_0^2y_2&y_0^2y_1y_2&{\color{red}\theta_3=y_0^2(y_1^2y_2 + y_1y_2^2)} & +3y_1^2y_2^2+2y_1y_2^3 & & \\
           & & &{\color{brown}y_0^2y_2^2}& y_0^2y_2^3&+y_2^4) & & \\\hline
(\HHH_P)_3&y_0^3&y_0^3y_1&y_0^3y_1^2&y_0^3y_1^3&y_0^3y_1^4& y_0^3(y_1^5+y_1^4y_2+y_1^3y_2^2), 
&y_0^3(y_1^6+3y_1^5y_2
\\
& &y_0^3y_2 &y_0^3y_1y_2 &y_0^3(y_1^2y_2 +y_1y_2^2)&y_0^3y_1^3y_2 &y_0^3(y_2^5+y_1y_2^4+y_1^2y_2^3) &+6y_1^4y_2^2+7y_1^3y_2^3+6y_1^2y_2^4 \\
& & &y_0^3y_2^2 &{\color{brown} y_0^3(y_1^2y_2-y_1y_2^2)} &y_0^3y_1^2y_2^2 &y_0^3(y_1^4y_2+2y_1^3y_2^2 & +3y_1y_2^5 +y_2^6) \\
& & & & y_0^3y_2^3 &y_0^3y_1y_2^3 &+2y_1^2y_2^3+y_1y_2^4) & \\
& & & & & y_0^3y_2^4& & \\ \hline
\end{array}
$$
With these candidates, one can use
{\tt Macaulay2} to verify 
\eqref{interesting-triangle-q-Ehrhhart-series-colored} as follows. One
can check that $\theta_1,\theta_2,\theta_3$, which have $\NN$-degrees $1,1,2$, 
do generate a subalgebra $B$ of $\HHH_P \subset \RR[y_0,y_1,y_2]$ having $\NN$-graded Hilbert series $\frac{1}{(1-t)^2(1-t^2)}$.  Thus $\theta_1,\theta_2,\theta_3$ are algebraically independent.
And then one can check that together with the six elements colored brown in
the table, they generate a subalgebra of $\HHH_P$ having $\NN$-graded
Hilbert series $\frac{1+t+t^2}{(1-t)^3}=\frac{1+2t+2t^2+t^3}{(1-t)^2(1-t^2)}$.  Since this matches
$\Hilb(\HHH_P,t)=\Eseries_P(t)$, these nine red and brown elements must generate {\it all} of $\HHH_P$.  Furthermore, this shows that the six brown elements must be free $B$-basis elements for $\HHH_P$, so it is Cohen-Macaulay, as predicted in 
Conjecture~\ref{conj:harmonic-algebra-omnibus}(ii).

\begin{remark} 
\label{rmk:algorithm-for-q-Ehrhart-series}
Note that this algorithm computed
$\Eseries_P(t,q)=\Hilb(\HHH_P,t,q)$ without finding generators for the 
ideals $\gr \, \II(mP \cap \ZZ^n)$ for all $m \geq 0$. 
We began with the guess
\eqref{interesting-triangle-q-Ehrhhart-series-colored} for $\Eseries_P(t,q)$
that came from computing the Hilbert series of $V_{mP \cap \ZZ^n} 
= \gr \, \II(mP \cap \ZZ^n)^\perp$ for small values of $m$. 
The answer suggested that we might
find generators for $\HHH_P$ whose
bidegrees $(b_i,a_i)$ all had $b_i \leq 3$.  And indeed, after computing $\RR$-bases for $V_{mP \cap \ZZ^n}$ for $m=1,2,3$, we were then able to prove that they generate $\HHH_P$.  Such an algorithm is not guaranteed to terminate in all cases, but was successfully used to verify all of the $\Eseries_P(t,q)$ in Figures~\ref{fig:q-Ehrhart-examples}, \ref{fig:area-4-polygons}, \ref{fig:area-5-polygons}.
\end{remark}

Lastly, we use the above description of $\HHH_P$ to compute the equivariant
$q$-Ehrhart series
$\Eseries^G_P(t,q)$ in $\Cl_\RR(G)[q][[t]]$
for the action of the group $G$ of order two that swaps $y_1, y_2$.
As in Example~\ref{ex:repeat-from-intro-with-symmetry}, the representation ring  
$
\Cl_\RR(G) \cong \ZZ[\epsilon]/(\epsilon^2-1) = \spn_\ZZ \{1,\epsilon\},
$
where $1, \epsilon$ denote the isomorphism classes of the one-dimension trivial and nontrivial $\RR[G]$-modules, respectively.
Examining the color-coded elements in the above table, one sees that
they were chosen so that the red $\theta_i$ for $i=1,2,3$ are all $G$-fixed, while the brown basis elements are either $G$-fixed (carrying the representation $1$), or $G$-negated (carrying the representation $\epsilon$), or swapped as a pair (carrying the regular representation $1+\epsilon$).  Consequently, one has
\begin{equation}
\Eseries_{P}^G(t,q) = \ch_{t,q}(\HHH_P) = 
    \frac{{\color{brown}1+(1+\epsilon) qt+ (1+\epsilon) q^2 t^2 + \epsilon q^3 t^3}}{(1-{\color{red}t})(1- {\color{red}q^2 t})(1 - {\color{red}{q^3 t^2}})}.
\end{equation}

\subsection{Some cautionary remarks}
We give here a few further cautionary remarks regarding the definition of $\HHH_P$ and
Conjecture~\ref{conj:harmonic-algebra-omnibus}.

\begin{remark}
One might be tempted to define the harmonic algebra differently, in a more general context.
Starting with any finite subset $\Zpoints \subseteq \kk^n$, over any field $\kk$, one could introduce
an $\NN^2$-graded $\kk$-subalgebra of the ring $\kk[y_0] \otimes_\kk \Div_\kk(\yy)$ defined by
\begin{equation}
\label{eq:suspect-definition}
    \HHH_\Zpoints := 
    \bigoplus_{m=0}^\infty \kk \cdot y_0^m \otimes_\kk V_{\underbrace{\Zpoints + \cdots + \Zpoints}_{m\text{ times}}}
\end{equation}
Here we adopt the convention that the 0-fold Minkowski sum of $\Zpoints$ with itself is $\{\origin\}$, 
so that $1=y_0^0 \otimes_\kk \yy^\origin \in (\HHH_\Zpoints)_0$.
Theorem~\ref{thm:minkowski-closure} shows that this does define a $\kk$-subalgebra of $\kk[y_0] \otimes_\kk \Div_\kk(\yy)$.
However, there are two issues with this definition \eqref{eq:suspect-definition} for $\HHH_{\Zpoints}$,
even when $\kk=\RR$.

The first issue is that when one takes $\Zpoints=\ZZ^n \cap P$ for a lattice polytope $P \subset \RR^n$, the above algebra $\HHH_\Zpoints$ can be a {\it proper subalgebra} of 
the harmonic algebra $\HHH_P$. This is because the inclusion
$$
\underbrace{(\ZZ^n \cap P)+ \cdots + (\ZZ^n \cap P)}_{m \text{ times}} \subseteq \ZZ^n \cap mP
$$
can be strict, namely for lattice polytopes $P$ failing to have the {\it integer decomposition property (IDP)} discussed in Defintion~\ref{def:IDP-property}.
Although not every $d$-dimensional lattice polytope $P$ has the IDP,
a result of Cox, Haase, Hibi and Higashitani \cite{cox2012integer} shows that its dilation $dP$
always has the IDP; see \cite{cox2012integer}. It follows that the affine semigroup ring $A_P$ is always generated by elements
of $\NN$-degree at most $d=\dim(P)$.
However, the interesting lattice triangle $P$
discussed in Section~\ref{sec:interesting-triangle-in-depth} required an algebra generator for its harmonic algebra $\HHH_P$ of $\NN$-degree $3 > 2=d$, showing that sometimes we must go to higher degrees than $d$ to generate $\HHH_P$.

A second issue with the above definition
\eqref{eq:suspect-definition} for $\HHH_\Zpoints$ is that it is not always 
finitely generated as an algebra, even when $\kk=\RR$
and $n=1$.
For example, let $\Zpoints = \{0,2,3\} \subset \RR^1$. It is not hard to check that the $m$-fold Minkowski sum $\Zpoints+\cdots+\Zpoints= \{0,2,3,4,\ldots,3m-1,3m\}
$ for all $m \geq 1$.  This means that, after identifying $\Div_\RR(\yy)$ with $\RR[y_1]$,
one has $V_{\underbrace{\Zpoints+\cdots+\Zpoints}_{m\text{ times}}}=\spn_\RR\{1,y_1,y_1^2,\ldots,y_1^{3m-1}\}$,
and
\begin{equation}
    \HHH_\Zpoints = \mathrm{span}_\RR \{1\} \oplus \mathrm{span}_\RR \{ y_0^m y_1^j \,:\, m \geq 0 \text{ and } 0 \leq j \leq 3m-1 \}.
\end{equation}
Therefore $\HHH_\Zpoints \subset \RR[y_0,y_1]$ is the semigroup ring for the additive subsemigroup of $\NN^2$
that may be visualized as follows, with a dot
at coordinates $(m,j)$ representing $y_0^m y_1^j$ in $\HHH_\Zpoints$:
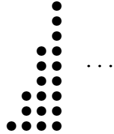
\begin{center}
\begin{small}
    \begin{tikzpicture}[scale = 0.2]

    \node at (0,0) {$\bullet$};
    \node at (1,0) {$\bullet$};
    \node at (1,1) {$\bullet$};
    \node at (1,2) {$\bullet$};
    \node at (2,0) {$\bullet$};
    \node at (2,1) {$\bullet$};
    \node at (2,2) {$\bullet$};
    \node at (2,3) {$\bullet$};
    \node at (2,4) {$\bullet$};
    \node at (2,5) {$\bullet$};
    \node at (3,0) {$\bullet$};
    \node at (3,1) {$\bullet$};
    \node at (3,2) {$\bullet$};
    \node at (3,3) {$\bullet$};
    \node at (3,4) {$\bullet$};
    \node at (3,5) {$\bullet$};
    \node at (3,6) {$\bullet$};
    \node at (3,7) {$\bullet$};
    \node at (3,8) {$\bullet$};

    \node at (6,4) {$\cdots$};
    \end{tikzpicture}
\end{small}
\end{center}
This illustrates the lack of finite generation:  minimal monomial $\RR$-algebra generators for $\HHH_\Zpoints$ are
$$
\{y_0 y_1^2,\,\, \,\, y_0 y_1^3,\,\, y_0^2 y_1^5,\,\, y_0^3 y_1^8,\,\, y_0^4 y_1^{11},\,\, \ldots\} = \{y_0 y_1^2\} \cup \,\, \{y_0^m y_1^{3m-1}\}_{m=2,3,4,\ldots}.
$$
This does not violate Conjecture~\ref{conj:harmonic-algebra-omnibus}(i), since $\HHH_\Zpoints\not\cong \HHH_P$
for any lattice polytope $P$. But proving Conjecture~\ref{conj:harmonic-algebra-omnibus} must involve extra features of harmonic spaces of lattice points inside polytopes.
\end{remark}

\begin{remark}
    \label{rmk:harmonic-functor}
    One can regard $P \leadsto \HHH_P$ as a functor $\HHH: \mathcal{C} \rightarrow \mathcal{D}$, where $\mathcal{D}$ is the category of $\NN^2$-graded $\RR$-algebras.  Here
$\mathcal{C}$ is a category of lattice polytopes $P \subset \RR^n$ whose morphisms $P \to P'$ are $\ZZ$-linear maps $\varphi:\ZZ^{n} \to \ZZ^{n'}$ with $f(P)\subseteq P'$.
The fact that this induces an $\RR$-algebra map $\HHH_P \to \HHH_{P'}$ stems from the fact that the point configurations 
$\Zpoints:=\ZZ^n \cap mP$ and $\Zpoints':=\ZZ^{n'} \cap mP'$
have $f(\Zpoints) \subseteq \Zpoints'$. This implies that the map 
$\varphi^\sharp: \RR[x_1,\ldots,x_{n'}] \rightarrow \RR[x_1,\ldots,x_n]$
that precomposes $f \mapsto f\circ \varphi$ will have $\varphi^\sharp (\II(\Zpoints')) \subseteq \II(\Zpoints)$, and hence $\varphi^\sharp(\gr\,\II(\Zpoints')) \subseteq \gr\,\II(\Zpoints)$.
Therefore the adjoint map $\varphi: \RR^n \rightarrow \RR^{n'}$,
when extended to a ring map $\RR[y_1,\ldots,y_n] \rightarrow \RR[y_1,\ldots,y_{n'}]$ satisfies
$$
\varphi(V_\Zpoints)=\varphi(\gr\,\II(\Zpoints)^\perp) \subseteq \gr\,\II(\Zpoints')^\perp = \varphi(V_{\Zpoints'}).
$$

If one instead considers $\HHH: P \leadsto \HHH_P$ as a functor into
the category of $\NN$-graded $\RR$-algebras, one can compare it to
the similar such functor $A: P \leadsto A_P$.  These two functors $\HHH_{(-)}$ and $A_{(-)}$ are not isomorphic; see the example in Section~\ref{sec:lack-of-singly-graded-isomorphism}. The authors do not know if they are equivalent.
\end{remark}

\begin{remark}
\label{rmk:inhomogeneous-basis}
One might hope to find some canonical $\RR$-linear basis
for $\HHH_P$. Although the authors are unaware of 
such a basis which is $\NN^2$-homogeneous, there is at least a canonical $\NN$-homogeneous $\RR$-basis,
coming from such $\RR$-bases for harmonic spaces $V_\Zpoints$ of finite subsets $\Zpoints \subset \RR^n$.

To describe these, assume $\kk$ is a field of characteristic zero, and identify $\kk^n$ and its basis
$y_1,\ldots,y_n$ with $(\kk^n)^*$ and its basis $x_1,\ldots,x_n$.  Then mapping $y_n \longmapsto x_n$
gives the middle isomorphism here  
$$
\Div:=\Div_\kk(\yy)=\kk[y_1,\ldots,y_n] \cong \kk[x_1,\ldots,x_n]=:S.
$$
This lets one view the  subspace 
$V_\Zpoints :=\gr \, \II(\Zpoints)^\perp \subset \Div$ as a subspace of $S=\kk[\xx]$.

\begin{prop}
    \label{prop:interpolation-general-loci}
    For $\kk$ a field of characteristic zero, and any finite point set $\Zpoints \subset \kk^n$, there is a unique basis $\{ g_\zz(\xx) \,:\, \zz \in \Zpoints \}$ of the harmonic space $V_\Zpoints$ such that
    $g_\zz(\zz') = \delta_{\zz,\zz'}$ for all $\zz' \in \Zpoints$.
\end{prop}

\begin{proof}
  Since $V_\Zpoints:=\gr\,\II(\Zpoints)^\perp$, one has a  $\kk$-vector space direct sum decomposition
  $S=V_\Zpoints \oplus \gr\,\II(\Zpoints)$.  Consequently,
  the composite map
  $
  V_\Zpoints \hookrightarrow S \twoheadrightarrow S/\gr\,\II(\Zpoints)
  $
  is a $\kk$-vector space isomorphism.
  Thus any set of homogeneous polynomials that give a $\kk$-basis for the (graded) subspace $V_\Zpoints \subset S$ will descend to a $\kk$-basis
  of $S/\gr\,\II(\Zpoints)$. Proposition~\ref{prop:graded-quotient-is-a-gr}(ii) implies that the same set of polynomials
  will also descend to a basis of $S/\II(\Zpoints)$.
  Hence this composite is also a  $\kk$-vector space isomorphism:
\begin{equation}
\label{eq:harmonics-give-coset-reps}
V_\Zpoints \hookrightarrow S \twoheadrightarrow S/\II(\Zpoints).
\end{equation}
  Now multivariate Lagrange interpolation identifies 
\begin{equation}
\label{eq:Lagrange-interpoliation-iso}
  S/\II(\Zpoints) \cong \{\text{ all functions }g:\Zpoints \to \kk\,\,\},
\end{equation}
  and the space on the right has a unique $\kk$-basis of $\{g_\zz\}_{\zz \in \ZZ}$ defined by $g_\zz(\zz')=\delta_{\zz,\zz'}$.
  The assertion then follows from combining this with
  \eqref{eq:harmonics-give-coset-reps} and
  \eqref{eq:Lagrange-interpoliation-iso}.
\end{proof}
Note that Proposition~\ref{prop:interpolation-general-loci} provides $\kk$-bases $\{g_\zz(\xx): \zz \in \Zpoints\}$
that are generally inhomogeneous.  Nevertheless, given a lattice polytope $P$, and taking $\kk=\RR$, one can assemble all of these distinguished $\RR$-bases for $V_{\ZZ^n \cap mP}$ into an $\NN$-graded $\RR$-basis $\{y_0^m g_{\zz}(\yy): \zz \in \ZZ^n \cap mP\}$
for $\HHH_P$.   Note that, even though this $\RR$-basis for $\HHH_P$ is in bijection with an $\RR$-basis for the semigroup ring $A_P$,
the bijection does not respect multiplication.  This is to be
expected, since the example discussed in Section~\ref{sec:lack-of-singly-graded-isomorphism} shows that these two rings are not isomorphic as $\NN$-graded $\RR$-algebras.
\end{remark}

\section{Dilations, products, and joins: proof of Theorem~\ref{thm:intro-three-constructions-on-series}}
\label{sec: dilations-products-joins}
Recall from the Introduction these three basic operations on polytopes:
\begin{itemize}
\item {\it dilation} by an positive integer factor $d$, sending $P$ to $dP$,
\item {\it Cartesian product}, sending $P \subset \RR^n$ and $Q \subset \RR^m$ to $P \times Q \subset \RR^{n+m}$,
\item {\it free join},
sending $P,Q$ to  
$P * Q \subseteq \RR^{1+n+m}$ defined by
\begin{align*}
    P*Q &:= \{ (t,t\pp,(1-t)\qq) \,:\, 0 \leq t \leq 1, \, \pp \in P, \, \qq \in Q \}\\
    &= \conv \left(\quad
    \{1\} \times P \times \origin_m 
    \quad \sqcup \quad 
    \{0\} \times \origin_n \times Q
    \quad \right)
\end{align*}
\end{itemize}

We wish to understand
how these operations interact with $q$-Ehrhart series $\Eseries_P(t,q)$.  In the process we will
see how they interact with harmonic algebras $\HHH_P$, by relating them to the  commutative algebra constructions for $\NN$-graded algebras of 
{\it Veronese subalgebras, Segre products} and {\it graded tensor products}; 
When applying these
constructions to the $\NN^2$-graded algebra $\HHH_P$, we will always use
the grading specialization $\deg(y_0^m \yy^\zz)=m$ to regard it as an $\NN$-graded
algebra.  The analysis will also prove Theorem~\ref{thm:intro-three-constructions-on-series}, whose statement we
recall here.

\vskip.1in
\noindent
{\bf Theorem~\ref{thm:intro-three-constructions-on-series}.}
{\it
Let $P, Q$ be lattice polytopes.
    \begin{itemize}
        \item[(i)] For positive integers $d$, the dilation $dP$ has $\Eseries_{dP}(t,q)$ given by
        $$
        \Eseries_{dP}(t,q)=\sum_{m=0}^\infty \Ehr_P(dm;q) t^m. 
        $$   
        \item[(ii)] The Cartesian product $P\times Q$ has $\Eseries_{P \times Q}(t,q)$ given by the Hadamard product
        $$
        \Eseries_{P \times Q}(t,q)=
        \sum_{m=0}^\infty \Ehr_P(m;q) \cdot  \Ehr_Q(m;q) \cdot t^{m}.
        $$
        \item[(iii)] The free join $P * Q$ has $\Eseries_{P * Q}(t,q)$ given by
        $$
        \Eseries_{P * Q}(t,q)=
        \frac{1-t}{1-qt} \cdot 
        \Eseries_P(t,q) \cdot \Eseries_Q(t,q).
        $$
    \end{itemize}
}
\vskip.1in

We deal with each of the three parts of this theorem in
the next three subsections.

\subsection{Dilations and the Veronese construction} 
\label{sec:dilation-Veronese}

The first and simplest of these is given by dilation $P \mapsto dP$. If $A = \bigoplus_{m \geq 0} A_m$ is a graded algebra, the $d^{th}$ {\em Veronese subalgebra} is 
\begin{equation}
    \Ver_d(A) := \bigoplus_{m \, \geq \, 0} A_{dm}.
\end{equation}
Then Theorem~\ref{thm:intro-three-constructions-on-series}(i) along with the following result are immediate from the definitions.

\begin{prop}
    \label{prop:veronese}
    For a lattice polytope $P$ and any $d \geq 0$ one has  
    $\HHH_{dP} \cong \Ver_d(\HHH_P)$.
\end{prop}

\subsection{Cartesian products and the Segre construction} 
\label{sec:product-Segre}

The Cartesian product $P \times Q$ of lattice polytopes is again a lattice polytope. Since $m(P \times Q)=mP \times mQ$, it is easily seen that
$$
i_{P \times Q}(m)=\Ehr_P(m) \cdot \Ehr_Q(m).
$$
There is a similarly easy relation for their
harmonic algebras, which comes from a general
lemma on the behavior of orbit harmonic rings
$R(\Zpoints)=S/\gr\,\II(\Zpoints)$ 
and harmonic spaces $V_\Zpoints$ when one has a Cartesian product of point loci.

\begin{lemma}
    \label{lem:harmonics-across-product}
    For a field $\kk$ and finite subsets
    $\Zpoints \subseteq \kk^n, \Zpoints' \subseteq \kk^{n'}$, one
    has graded isomorphisms
    \begin{align*}
    R(\Zpoints \times \Zpoints') &\cong 
    R(\Zpoints) \otimes_\kk R(\Zpoints'),\\
    V_{\Zpoints \times \Zpoints'} &\cong V_\Zpoints \otimes_\kk V_{\Zpoints'}.
    \end{align*}
\end{lemma}

\begin{proof}
    Write $S_n, S_{n'}$ for the polynomial rings $\kk[x_1, \dots, x_n], \kk[x_1, \dots, x_{n'}]$ which contain $\II(\Zpoints),\II(\Zpoints')$, so that $\II(\Zpoints \times \Zpoints')$ will be an ideal in $S_n \otimes S_{n'}$.
    We claim one has the following equality of ideals  
    \begin{equation}
    \label{eq:Cartesian-product-grs}
        I:=\gr  \, \II(\Zpoints \times \Zpoints')
        \quad = \quad 
        \gr \, \II(\Zpoints) \otimes S_{n'} + S_n \otimes \gr \, \II(\Zpoints)=:J,
    \end{equation}
    from which the remaining assertions will follow.
    To see this claim, first check the containment $I \supseteq J$:  for
     $f \in \II(\Zpoints)$ and $f' \in \II(\Zpoints')$, one has both $f \otimes 1$ and $1 \otimes f'$ lying in $\II(\Zpoints \times \Zpoints')$, and if  $f,f' \neq 0$,
     their top degree components $\tau(f \otimes 1) = \tau(f) \otimes 1$ and $\tau(1 \otimes f') = 1 \otimes \tau(f')$ both lie in $I$.  On the other hand, the inclusion $I \supseteq J$ must be an equality because 
    the surjection
    $$
     (S_n \otimes S_{n'})/J=
    \,\, \twoheadrightarrow \,\,
    (S_n \otimes S_{n'})/ I
    $$
    is an isomorphism via dimension-counting: 
    $$
    \dim_\kk  (S_n \otimes S_{n'})/J \leq 
    \dim_\kk S_n/\gr \, \II(\Zpoints) \cdot 
    \dim_\kk S_{n'}/\gr \, \II(\Zpoints')
    = \#(\Zpoints \times \Zpoints')=\dim_\kk  (S_n \otimes S_{n'})/I. \qedhere
    $$
\end{proof}

For $\NN$-graded $\kk$-algebras
$A, B$, their {\em Segre product} $\Segre(A,B)$ is the graded algebra
\begin{equation}
    \Segre(A,B) = \bigoplus_{m=0}^\infty \,\, \underbrace{A_m \otimes_\kk B_m}_{\Segre(A,B)_m}
\end{equation}

\begin{prop}
    \label{prop:segre}
    For any lattice polytopes $P$ and $Q$, one has
    $\HHH_{P \times Q} \cong \Segre(\HHH_P, \HHH_Q).$
\end{prop}

\begin{proof}
    This comes via Lemma~\ref{lem:harmonics-across-product} and the definitions,
    as $\ZZ^n \cap m(P \times Q)
    = (\ZZ^n \cap mP) \times (\ZZ^n \cap mQ )$
\end{proof}

\subsection{Joins} 
\label{sec:join-tensor}
Recall that for polytopes $P, Q \subseteq \RR^n, \RR^{n'}$, their {\it (free) join} $P * Q \subseteq \RR^{1+n+n'}$ is
\begin{align*}
    P*Q := \{ (t,t\pp,(1-t)\qq) \,:\, 0 \leq t \leq 1, \, \pp \in P, \, \qq \in Q \}
    = \conv \left(\,\,
    \{1\} \times P \times \origin_{n'} 
    \quad \sqcup \quad 
    \{0\} \times \origin_n \times Q
    \,\, \right)
\end{align*}
If $P$ and $Q$ are lattice polytopes, the join $P*Q$ is again a lattice polytope. The effect of join on harmonic algebras is more complicated than dilation and Cartesian product.  We will eventually see, in Theorem~\ref{thm:join-isomorphism} below, that it is almost, but not quite, the {\it $\NN$-graded tensor product} 
$$
A \otimes B = \bigoplus_{m=0}^\infty \,\,
\underbrace{ \bigoplus_{i+j \, = \, m} A_i \otimes B_j}_{(A\otimes B)_m}.
$$

We first recall some facts on joins 
of lattice polytopes $P,Q$ in $\RR^n, \RR^{n'}$.  Let $x_0$ denote the extra first coordinate in the ambient space $\RR^1 \times \RR^n \times \RR^{n'} =\RR^{1+n+n'}$ for the join $P*Q$. Slicing by the hyperplanes $x_0=r$ gives a disjoint decomposition
of the dilates $m(P*Q)$ for $m \geq 0$:

\begin{equation}
\label{eq:first-stratification}
\ZZ^{1+n+n'} \cap m(P*Q) 
= \bigsqcup_{\substack{(r,r'):\\ r+r'=m}} 
\{r\} \times
(\ZZ^n \cap rP) \times (\ZZ^{n'} \cap r'Q).
\end{equation}

We illustrate \eqref{eq:first-stratification} for $m(P * Q)$ with $m=1,2,3$,  when 
$P=[0,1] \subset \RR^1$ and  $Q=[0,2] \subset \RR^1$:
$$
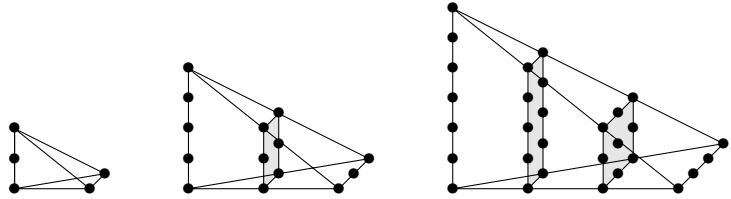
\begin{tikzpicture}[scale = 0.2]
        \node at (0,0) {$\bullet$};
        \node at (0,2) {$\bullet$};
        \node at (0,4) {$\bullet$};
        \draw[] (0,0) -- (0,4);
        \node at (5,0) {$\bullet$};
        \node at (6,1) {$\bullet$};
        \draw[] (0,0) -- (0,4) --(6,1) -- (5,0) -- (0,0) -- (6,1);
        \draw[] (5,0) -- (0,4);
 \end{tikzpicture}
    \qquad 
\begin{tikzpicture}[scale = 0.2]
        \node at (0,0) {$\bullet$};
        \node at (0,2) {$\bullet$};
        \node at (0,4) {$\bullet$};
        \node at (0,6) {$\bullet$};
        \node at (0,8) {$\bullet$};
        %
        \draw[fill = black!10] (5,0) -- (5,4) --(6,5) -- (6,1) -- (5,0);
        \node at (5,0) {$\bullet$};
        \node at (6,1) {$\bullet$};
        \node at (5,2) {$\bullet$};
        \node at (6,3) {$\bullet$};
        \node at (5,4) {$\bullet$};
        \node at (6,5) {$\bullet$};
        \node at (10,0) {$\bullet$};
        \node at (11,1) {$\bullet$};
        \node at (12,2) {$\bullet$};
        \draw[] (0,0) -- (0,8) --(10,0) -- (12,2) -- (0,0)--(10,0);
        \draw[] (0,8) -- (12,2);
 \end{tikzpicture}
    \qquad 
\begin{tikzpicture}[scale = 0.2]
        \node at (0,0) {$\bullet$};
        \node at (0,2) {$\bullet$};
        \node at (0,4) {$\bullet$};
        \node at (0,6) {$\bullet$};
        \node at (0,8) {$\bullet$};
        \node at (0,10) {$\bullet$};
        \node at (0,12) {$\bullet$};
        \draw[] (0,0) -- (0,12);
         \draw[fill = black!10] (5,0) -- (5,8) --(6,9) -- (6,1) -- (5,0);
        \node at (5,0) {$\bullet$};
        \node at (5,2) {$\bullet$};
        \node at (5,4) {$\bullet$};
        \node at (5,6) {$\bullet$};
        \node at (5,8) {$\bullet$};
        \node at (6,1) {$\bullet$};
        \node at (6,3) {$\bullet$};
        \node at (6,5) {$\bullet$};
        \node at (6,7) {$\bullet$};
        \node at (6,9) {$\bullet$};
        \draw[fill = black!10] (10,0) -- (12,2) --(12,6) -- (10,4) -- (10,0);
        \node at (10,0) {$\bullet$};
        \node at (11,1) {$\bullet$};
        \node at (12,2) {$\bullet$};
        \node at (10,2) {$\bullet$};
        \node at (11,3) {$\bullet$};
        \node at (12,4) {$\bullet$};
        \node at (10,4) {$\bullet$};
        \node at (11,5) {$\bullet$};
        \node at (12,6) {$\bullet$};
        \node at (15,0) {$\bullet$};
        \node at (16,1) {$\bullet$};
        \node at (17,2) {$\bullet$};
        \node at (18,3) {$\bullet$};
        \draw[] (15,0) -- (18,3);

        \draw[] (0,0) -- (15,0) -- (0,12) -- (18,3) -- (0,0);
    \end{tikzpicture}
$$

This decomposition \eqref{eq:first-stratification} implies a well-known fact (see, e.g., Beck and Robins \cite[Exer.~3.33]{BeckRobins}):
\begin{equation}
\label{eq:Beck-Robins-exercise}
\Ehr_{P*Q}(m)
=\sum_{\substack{(r,r'):\\r+r'=m}}
\Ehr_P(r)\cdot \Ehr_Q(r')
\end{equation}
This fact \eqref{eq:Beck-Robins-exercise} is easily seen to be equivalent to the following simple relation
of Ehrhart series:
\begin{equation}
\label{eq:join-Ehrhart-series-relation}
 \Eseries_{P*Q}(t)=\Eseries_P(t) \cdot \Eseries_Q(t).
\end{equation}

Recalling that $\Eseries_P(t,q)=\Hilb(\HHH_P,t,q)$,
the next result is the $q$-analogue of \eqref{eq:join-Ehrhart-series-relation}.

\begin{thm}
 \label{thm:join-q-Ehrharts}
For lattice polytopes $P,Q$, one has the following relation among the  $\NN^2$-graded Hilbert series for the harmonic algebras $\HHH_P, \HHH_Q, \HHH_{P * Q}$:
    \[
    \Eseries_{P*Q}(t,q) = \frac{1-t}{1-qt} \cdot \Eseries_P(t,q)\cdot \Eseries_Q(t,q).
    \]
\end{thm}

As mentioned in Remark~\ref{rmk:join-pyramid}, if we take $Q$ to be a single point, Theorem~\ref{thm:join-q-Ehrharts} says that the $q$-Ehrhart series of the pyramid $\pyr(P) := P * \{\mathrm{pt}\}$ is $\Eseries_{\pyr(P)}(t,q) = (1-qt) \cdot \Eseries_P(t,q)$.

\begin{proof}
Letting $\Zpoints_m:=\ZZ^{1+n+n'} \cap m(P*Q) \subset \RR^{1 +n+n'}$, we start by describing the ideal
$$
\gr\,\II(\Zpoints_m) \,\, \subset \,\, \RR[x_0] \otimes \RR[x_1,\ldots,x_n] \otimes \RR[x_1,\ldots,x_{n'}].
$$  
The product structure in each term of
\eqref{eq:first-stratification} leads to some natural elements
of $\II(\Zpoints_m)$ and $\gr\,\II(\Zpoints_m)$.  First translate $P, Q$, without loss of generality, so that $\origin_n \in P$ and $\origin_{n'} \in Q$, implying nestings 
\begin{equation}
\label{eq:nestings}
P \subseteq 2P \subseteq 3P \subseteq \cdots \text{ and }
Q \subseteq 2Q \subseteq 3Q \subseteq \cdots.
\end{equation}
Suppose $\ell,\ell' \geq 0$ have $\ell+\ell' \leq m$, and
one is given $f \in \II(\ZZ^n \cap \ell P)$ and $f' \in \II(\ZZ^{n'}\cap \ell' Q)$.  
Then we claim that the nestings \eqref{eq:nestings} imply that the polynomial
\begin{equation}
\hat{f}:= \prod_{r \, = \, \ell+1}^{m-(\ell'+1)} (x_0 - r) \otimes f \otimes f' 
\end{equation}
vanishes on $\Zpoints_\ell$. To see this, note that 
a typical point of $(r,\zz,\zz')$ of $\Zpoints_\ell$ will either 
\begin{itemize}
    \item have $\ell+1 \leq r \leq m-(\ell'+1)$ so $\hat{f}$ vanishes due to one of its factors $x_0-r$, 
 
\item or else have $r \leq \ell$, and then $\hat{f}$ vanishes
due to the factor $f$, since $\zz \in rP \subseteq \ell P$,
\item or else have $m-r \leq \ell'$, and then $\hat{f}$ vanishes 
due to the factor $f'$, since $\zz' \in (m-r)P \subseteq \ell' P$
\end{itemize}
Thus $\hat{f}$ lies in $\II(\Zpoints_m)$, implying that $\gr \, \II(\Zpoints_m)$ contains its top degree homogeneous component
\begin{equation}
\tau(\hat{f})= 
x_0^{m-\ell-\ell'} \otimes \tau(f) \otimes \tau(f').
\end{equation}
Note also that $x_0 (x_0 - 1) \cdots (x_0 - m)$ vanishes on $\Zpoints_m$.  Hence $\gr \, \II(\Zpoints_m)$ contains this ideal:
\begin{equation}
J_m :=
(x_0^{m+1} \otimes 1 \otimes 1) + 
\left( (x_0^{m-\ell-\ell'})  \otimes 
\sum_{\ell + \ell' \, \leq \, m}
\gr \, \II(\ZZ^n \cap \ell P) \otimes
\gr \, \II(\ZZ^{n'} \cap \ell' Q) \right).
\end{equation} 
We wish to show that, in fact, this containment $J_m \subseteq \gr \, \II(\Zpoints_m)$ is an equality.
To this end, we calculate
the perp space $J_m^\perp$ within
$\RR[y_0] \otimes \RR[y_1,\ldots,y_n] \otimes \RR[y_1,\ldots,y_{n'}]$ for the ideal $J_m$.
Abbreviate
\begin{equation}
V^P_\ell := \gr \, \II(\ZZ^n \cap \ell P )^\perp 
\quad \text{and} \quad 
V^Q_{\ell'} := \gr \, \II(\ZZ^{n'} \cap \ell' P)^\perp.
\end{equation}
By Lemma~\ref{lem:harmonic-nesting-lemma} we have
\begin{equation}
\label{eq:nesting-harmonics}
\RR \cdot 1 = V^P_0 \subseteq V^P_1 \subseteq 
\cdots \subseteq V^P_\ell \quad \text{and} \quad
\RR \cdot 1 = V^Q_0 \subseteq V^Q_1 \subseteq 
\cdots \subseteq V^Q_\ell.
\end{equation}
The following space $V_m$ of polynomials is annihilated 
by $J_m$ under the $\odot$-action, so lies in $J_m^\perp$:
\begin{equation}
    \label{eq:J-harmonics}
    V_m :=
\bigoplus_{s \, = \, 0}^m \left(
\sum_{\ell + \ell' \, = \, s}
\RR \cdot y_0^{m-s} \otimes V_\ell^P \otimes V_{\ell'}^Q 
\right).
\end{equation}
By \eqref{eq:two-filtrations-equation-at-q=1} in Lemma~\ref{lem:nested-spaces-lemma} below, $V_m$ has the following dimension:
\begin{align*}
    \dim_\RR V_m = 
    \sum_{s \, = \, 0}^m
    \dim \left(  
        \sum_{\ell+\ell' \, = \, s} V_\ell^P \otimes V_{\ell'}^Q
    \right) 
    &= 
    \sum_{r+r' \, = \, m}
    \dim(V_r^P) \cdot \dim(V_{r'}^Q) 
    \\ 
    &= \sum_{r+r' \, = \, m}
    \Ehr_P(r) \cdot \Ehr_Q(r') =\Ehr_{P*Q}(m) 
    = \# \Zpoints_m.
\end{align*}
Therefor these inequalities
\begin{equation*}
\# \Zpoints_m =  \dim_\RR V_m \leq \dim_\RR J_m^{\perp} =  \dim_\RR S/J_m \leq \dim_\RR S/\gr \, I(\Zpoints_m) = \# \Zpoints_m
\end{equation*}
must all be equalities, implying 
\begin{align}
\gr \, \II(\Zpoints_m) &= J_m, \\
\label{eq:join-harmonic-components-identified}
V_{\Zpoints_m}=\gr\,\II(\Zpoints_m)^\perp&=J_m^\perp=V_m.
\end{align}

This now allows us to compute $\Hilb(V_{\Zpoints_m},q)$, using \eqref{eq:J-harmonics}, as
\begin{equation}
\label{eq:J-harmonic-hilb}
\Hilb(V_{\Zpoints_m},q)
=\Hilb(V_m,q)
=\sum_{s=0}^m q^{m-s}
\Hilb\left(
\sum_{\ell+\ell'=s} V_\ell^P \otimes V_{\ell'}^Q,q
\right).
\end{equation}
On the other hand, direct calculation shows
that the coefficient of $t^m$ in $\frac{1-t}{1-qt} \Eseries_P(t,q) \cdot \Eseries_Q(t,q)$ is
$$
\sum_{s=0}^m q^{m-s}
\sum_{\ell+\ell'=s}
\Hilb( V_\ell^P \otimes V_{\ell'}^Q,q)
-
\sum_{s=0}^{m-1} q^{m-1-s}
\sum_{\ell+\ell'=s}
\Hilb( V_\ell^P \otimes V_{\ell'}^Q,q).
$$
This last expression coincides, by \eqref{eq:two-filtrations-equation} in Lemma~\ref{lem:nested-spaces-lemma} below,
with \eqref{eq:J-harmonic-hilb}.
In other words,
$$
\frac{1-t}{1-qt} \Eseries_P(t,q) \cdot \Eseries_Q(t,q)
=\sum_{m=0}^\infty t^m  \cdot \Hilb(V_{\Zpoints_m},q)
=\Hilb( \HHH_{P*Q},t,q) =\Eseries_{P*Q}(t,q).\qedhere
$$
\end{proof}

The previous proof used the following
technical lemma, on Hilbert series arising from graded tensor products $V \otimes V'$ of spaces $V, V'$ equipped with homogeneous filtrations.

\begin{lemma}
    \label{lem:nested-spaces-lemma}
    Let $V, V'$ be graded vector spaces, each
    with nested sequences of homogeneous subspaces
$$      
0 \subseteq V_1 \subseteq V_2 \subseteq \cdots \subseteq V \text{ and } \quad 0 \subseteq V'_1 \subseteq V'_2 \subseteq \cdots \subseteq V'
$$
Then one has the following Hilbert series identity for each $n$:
    \begin{multline}
    \label{eq:two-filtrations-equation}
        \sum_{k  =  0}^n q^{n-k} \cdot \Hilb\left( \sum_{a+b \, = \, k} V_a \otimes V'_b, q \right) = \\
        \sum_{k  =  0}^n q^{n-k} \sum_{a+b \, = \, k} \Hilb(V_a \otimes V'_b, q) - \sum_{k  =  0}^{n-1} q^{n-k-1} \sum_{a+b \, = \, k} \Hilb(V_a \otimes V'_b, q).
    \end{multline}
    In particular, setting $q=1$, one has
    \begin{equation}
    \label{eq:two-filtrations-equation-at-q=1}
        \sum_{k  =  0}^n \dim \left( \sum_{a+b \, = \, k} V_a \otimes V'_b \right) = \sum_{c+d \, = \, n} \dim(V_c) \cdot \dim(V_d).
    \end{equation}
\end{lemma}

\begin{proof}
    Since the spaces $V_a$ and $V'_b$ nest, the sum $\sum_{a+b \, = \, k} V_a \otimes V'_b$ of subspaces of $V \otimes V'$ is {\it not} a direct sum.
    However, after 
    abbreviating
    $$
    h_{ab}:=\Hilb(V_a/V_{a-1},q) \cdot \Hilb(V'_b/V'_{b-1},q)
    =\Hilb(V_a/V_{a-1} \otimes V'_b/V'_{b-1},q),
    $$
    the left side of the lemma can be rewritten
    \begin{equation}
    \label{eq:basis-lemma-one}
        \sum_{k  =  0}^n q^{n-k} \cdot 
          \sum_{a+b\leq k} h_{ab},
    \end{equation}
whose coefficient of $h_{ab}$ is
$$
\sum_{k=a+b}^n q^{n-k}=1+q+q^2+\cdots+q^{n-(a+b)}
=[\hat{n}]_q 
$$
abbreviating $\hat{n}:=n-(a+b)+1$. Meanwhile, the right side of the lemma can
 be rewritten
    \begin{equation}
    \label{eq:basis-lemma-two}
        \sum_{k  =  0}^n q^{n-k} 
        \sum_{a'+b'  \, = \,  k} \,\,
        \sum_{\substack{a: 0 \leq a \leq a'\\b: 0 \leq b \leq b'}} h_{ab}
        \quad - \quad
        \sum_{k \, = \, 0}^{n-1} q^{n-1-k} 
        \sum_{a'+b' \,  = \,  k} \,\,
        \sum_{\substack{a: 0 \leq a \leq a'\\b: 0 \leq b \leq b'}} h_{ab}.
    \end{equation}
Hence the coefficient of $h_{ab}$ on the right side of the lemma is
\begin{align*}
    &\sum_{\substack{(a',b'):\\ 
    a' \geq a\\b'\geq b\\a'+b'\leq n}} q^{n-(a'+b')}
    \,\, - \,\,
    \sum_{\substack{(a',b'):\\ 
    a' \geq a\\b'\geq b\\a'+b'\leq n-1}} q^{n-1-(a'+b')}\\
    &=\hat{n}q^0 +
    (\hat{n}-1) q^1 +
    (\hat{n}-2)q^2 + 
    \cdots +
    2q^{\hat{n}-2}+q^{\hat{n}-1}\\
    &\qquad -\left(
    (\hat{n}-1)q^0 +
    (\hat{n}-2)q^1 +
    (\hat{n}-3)q^2 + 
    \cdots +
    2q^{\hat{n}-3}+q^{\hat{n}-2}
    \right)\\
     &=q^0+q^1+\cdots+q^{\hat{n}-2}+q^{\hat{n}-1}
    =[\hat{n}]_q \qedhere
\end{align*}
\end{proof}

We are ready to relate the harmonic algebras $\HHH_{P*Q}, \HHH_P,$ and $\HHH_Q$, starting with the following suggestive rewriting of 
Theorem~\ref{thm:join-q-Ehrharts}:
\begin{align}
\label{eq:suggestive-rewriting}
    (1-qt) \cdot \Hilb( \HHH_{P*Q}, t,q) 
= (1-t) \cdot \Hilb(\HHH_P \otimes \HHH_Q,t,q).
\end{align}
To interpret this identity, it helps to name
the ``auxiliary" variables in $\HHH_{P*Q}, \HHH_P, \HHH_Q$ as $y_{P*Q}, y_P, y_Q$, 
and to abbreviate the other variable sets
$\yy_n:=(y_1,\ldots,y_n)$,
$\yy_{n'}:=(y_1,\ldots,y_{n'})$.
This means that the rings
$\HHH_{P*Q}, \HHH_P \otimes \HHH_Q$,
are $\RR$-subalgebras of the following ambient polynomial algebras
\begin{align*}
\HHH_{P*Q} &\subset 
\RR[y_{P*Q}] \otimes \RR[y_0]  \otimes \RR[\yy_n]  \otimes \RR[\yy_{n'}]\\
\HHH_P \otimes \HHH_Q &\subset 
 \RR[y_P, \yy_n]  \otimes 
 \RR[y_Q,\yy_{n'}].
\end{align*}
Their $\NN^2$-gradings have $\deg(y_P)=\deg(y_Q)=\deg(y_{P*Q})=(1,0)$,
and all other variables have $\deg(y_0)=\deg(y_i)=(0,1)$.
Since the ambient polynomial algebras are integral domains,
so are the harmonic algebras inside them.  Hence
one concludes that 
\begin{itemize}
    \item $y_{P*Q} \cdot y_0 \,\,(:=y_{P*Q} \otimes y_0 \otimes 1 \otimes 1)$ is a nonzero divisor inside $\HHH_{P*Q}$, with $\NN^2$-degree $(1,1)$, and
    \item $y_P \otimes 1 - 1 \otimes y_Q$ is a nonzero divisor inside $\HHH_P \otimes \HHH_Q$, with $\NN^2$-degree $(1,0)$.
\end{itemize} 
In a $\NN^2$-graded algebra $A$
with a homogeneous nonzerodivisor $\theta$
of degree $(a,b)$, the exact sequence
$$
0 \rightarrow A(-(a,b)) \overset{\cdot \theta}{\longrightarrow} A \rightarrow A/(\theta) \rightarrow 0
$$
shows that one has the Hilbert series relation
$$
\Hilb(A/(\theta),t,q) = (1-t^a q^b) \cdot \Hilb(A,t,q).
$$
Thus the left and right sides of \eqref{eq:suggestive-rewriting} can reinterpreted as Hilberts series for these quotient rings:
\begin{align*}
(1-qt) \cdot \Hilb( \HHH_{P*Q}, t,q)
&=
\Hilb(\,\, \HHH_{P*Q}/(y_{P*Q} \cdot y_0) \,\, ,t,q), \\
(1-t) \cdot \Hilb(\HHH_P \otimes \HHH_Q,t,q)
&=\Hilb(\,\, (\HHH_P \otimes \HHH_Q) /(y_P \otimes 1 - 1 \otimes y_Q)\,\, ,t,q).
\end{align*}

In this way, the following result is an algebraic strengthening of Theorem~\ref{thm:join-q-Ehrharts}.

\begin{thm}
\label{thm:join-isomorphism}
For any lattice polytopes $P$ and $Q$, one has an $\NN^2$-graded algebra isomorphism
$$
(\HHH_P \otimes \HHH_Q)/(y_P \otimes 1 - 1 \otimes y_Q) 
\quad \cong \quad 
\HHH_{P*Q}/(y_{P*Q} \cdot y_0) 
$$
\end{thm}

\begin{proof}
In light of the preceding discussion on Hilbert series, it suffices by dimension-counting to exhibit an $\NN^2$-graded 
$\RR$-algebra surjection
from the left side to the right side in the theorem.

Consider the $\RR$-linear map $\varphi$
defined via
$$
\begin{array}{rcl}
 \RR[y_P, \yy_n]  \otimes \RR[y_Q,\yy_{n'}]
 & \longrightarrow & 
\RR[y_{P*Q}] \otimes \RR[y_0]  \otimes \RR[\yy_n]  \otimes \RR[\yy_{n'}]\\
y_P^\ell \cdot g(\yy_n) \otimes 
y_Q^{\ell'} \cdot g'(\yy_{n'})
&\longmapsto&
y_{P*Q}^{\ell+\ell'} \otimes 1 \otimes g(\yy_n) \otimes g'(\yy_{n'}).
\end{array}
$$
One can readily check that 
$\varphi$ 
\begin{itemize}
\item
is actually a map of $\NN^2$-graded $\RR$-algebras, 
\item 
that it restricts to a map $\HHH_P \otimes \HHH_Q \rightarrow \HHH_{P*Q}$, using
    \eqref{eq:J-harmonics}, \eqref{eq:join-harmonic-components-identified}, and
\item
    that $y_P \otimes 1 - 1 \otimes y_Q$ lies in its kernel. 
\end{itemize}
Consequently, $\varphi$ descends to a map on the quotient
\begin{equation}
    \label{eq:join-isomorphism-descend}
    \varphi: 
    (\HHH_P \otimes \HHH_Q)/
    (y_P \otimes 1 - 1 \otimes y_Q)
    \longrightarrow
    \HHH_{P*Q}.
\end{equation}
This map is not surjective.  However, after post-composing it with the canonical projection map
$\pi: \HHH_{P*Q} \twoheadrightarrow
    \HHH_{P*Q}/(y_{P*Q} \,y_0)$,
then we claim that 
\begin{equation}
    \label{eq:join-isomorphism-ultimate}
    \pi \circ \varphi:
    (\HHH_P \otimes \HHH_Q) / 
    (y_P \otimes 1 - 1 \otimes y_Q)
    \longrightarrow
    \HHH_{P*Q}/(y_{P*Q}\, y_0)
\end{equation}
actually {\it is} the surjection that we seek.
To see this, note that the image of $\varphi$ 
within $\HHH_{P*Q}$ is
$$
    \bigoplus_{m = 0}^\infty \,\,
    \sum_{\ell + \ell' \, = \, m} 
    \RR \cdot y_{P*Q}^m \cdot y_0^0 \otimes 
    V_\ell^P \otimes V_{\ell'}^Q.
$$
On the other hand,  \eqref{eq:J-harmonics} shows the remaining part of
$\HHH_{P*Q}$ not lying in this image has
the form
$$
\bigoplus_{m=1}^\infty \,\,
\bigoplus_{s  =  0}^{m-1} \left(
\sum_{\ell + \ell' \, = \, s}
\RR \cdot y_{P*Q}^m \cdot y_0^{m-s} \otimes V_\ell^P \otimes V_{\ell'}^Q \right),
$$
which is contained in the ideal $(y_{P*Q} \, y_0)$ and
therefore vanishes in the
quotient $\HHH_{P*Q}/(y_{P*Q} \, y_0)$.
\end{proof}

\bibliographystyle{abbrv}
\bibliography{bibliography}
\end{document}